\documentclass[final]{cmslatex}
\usepackage{latexsym,amssymb}
\usepackage{subfig, amsmath, amsfonts,amssymb}
\usepackage{graphicx,color}
\usepackage{multirow}
\usepackage[ruled]{algorithm2e}
\usepackage{algorithmic}
\usepackage{placeins}

\renewcommand{\theequation}{\arabic{section}.\arabic{equation}}

\newcommand{\bvec}[1]{\mathbf{#1}}
\newcommand{\abs}[1]{\left\lvert#1\right\rvert}
\newcommand{\norm}[1]{\left\lVert#1\right\rVert}
\newcommand{\normET}[1]{\left\lVert#1\right\rVert_{E,\mathcal{T}}}
\newcommand{\average}[1]{\left\langle#1\right\rangle}

\newcommand{\ud}{\,\mathrm{d}}

\newcommand{\Ecut}{E_{\mathrm{cut}}}

\newcommand{\mc}[1]{\mathcal{#1}}

\newcommand{\wt}[1]{\widetilde{#1}}
\newcommand{\vn}{\bvec{n}}
\newcommand{\vr}{\bvec{r}}

\newcommand{\vq}{\bvec{q}}
\newcommand{\vw}{\bvec{w}}

\newcommand{\dist}{\operatorname{dist}} 
\newcommand{\Vhxc}{V_{\mathrm{hxc}}}
\newcommand{\Vion}{V_{\mathrm{ion}}}
\newcommand{\eff}{\mathrm{eff}}
\newcommand{\DG}{\mathrm{DG}}

\newcommand{\ie}{\textit{i.e.}\xspace}
\newcommand{\eg}{\textit{e.g.}\xspace}

\newcommand{\jump}[1]{\big[\hspace{-0.7mm} \big[ #1 \big]
  \hspace{-0.7mm} \big]} 
\newcommand{\mean}[1] {\big\{ \hspace{-0.7mm} \big\{ #1 \big\}
  \hspace{-0.7mm} \big\}}

\DeclareMathOperator{\spanop}{span}

\renewcommand{\Re}{\mathrm{Re}}

\newtheorem{remark}[theorem]{Remark}
\newtheorem{assumption}[theorem]{Assumption}

\begin{document}

\title{A posteriori error estimator for adaptive local basis functions
to solve Kohn-Sham density functional theory}
\author{
Jason Kaye \thanks{Part III of the Mathematical Tripos, Department of Applied Mathematics and Theoretical
Physics, University of Cambridge, Cambridge, UK, jk582@cam.ac.uk.} \and
Lin Lin \thanks{
Computational Research Division, Lawrence Berkeley National
Laboratory, Berkeley, CA 94720, linlin@lbl.gov.} 
\and Chao Yang \thanks{Computational Research Division, Lawrence Berkeley National
  Laboratory, Berkeley, CA 94720, cyang@lbl.gov.}
}

\pagestyle{myheadings} \markboth{A POSTERIORI ESTIMATOR FOR ADAPTIVE LOCAL
BASIS}{J. KAYE AND L. LIN AND C. YANG}\maketitle

\begin{abstract}
Kohn-Sham density functional theory is one of the most widely used
electronic structure theories.  The recently developed adaptive local basis
functions form an accurate and systematically improvable basis set for
solving Kohn-Sham density functional theory using discontinuous Galerkin
methods, requiring a small number of basis functions per atom.  In this
paper we develop residual-based a posteriori error estimates for the
adaptive local basis approach, which can be used to guide non-uniform basis
refinement for highly inhomogeneous systems such as surfaces and large
molecules. The adaptive local basis functions are non-polynomial basis
functions, and standard a posteriori error estimates for $hp$-refinement
using polynomial basis functions do not directly apply. We generalize the
error estimates for $hp$-refinement to non-polynomial basis functions. We
demonstrate the practical use of the a posteriori error estimator in
performing three-dimensional Kohn-Sham density functional theory
calculations for quasi-2D aluminum surfaces and a single-layer graphene
oxide system in water.
\end{abstract}



\section{Introduction}\label{sec:intro}
In this paper we consider an a posteriori error estimator of
the eigenvalues and eigenvectors of the following linear eigenvalue
problem
\begin{equation}
  \begin{split}
    &\left(-\frac{1}{2}\Delta + V\right) \psi_{i} = \varepsilon_{i}\psi_{i}, \\ 
    &\int \psi^{*}_{i}(\vr) \psi_{j}(\vr) \ud \vr = \delta_{ij}, \quad
    i,j=1,\ldots,N.
  \end{split}
  \label{}
\end{equation}
$\{\psi_i\}_{i=1}^{N}$ are the eigenvectors corresponding to the lowest
$N$ eigenvalues $\{\varepsilon_i\}_{i=1}^{N}$.
This problem arises in solving the Kohn-Sham nonlinear eigenvalue
problem
\begin{equation}
  \begin{split}
  &H[\rho]\psi_{i} = \varepsilon_{i} \psi_{i},\\
	&\rho(\vr) = \sum_{i=1}^{N} \abs{\psi_{i}(\vr)}^2, \quad \int
  \psi^{*}_{i}(\vr) \psi_{j}(\vr) \ud \vr = \delta_{ij},
  \end{split}
  \label{eqn:KS}
\end{equation}
where $N$ is the number of electrons (spin degeneracy is omitted
here for simplicity), $\{\varepsilon_{i}\}_{i=1}^{N}$ are the Kohn-Sham
eigenvalues, and $\{\psi_{i}\}_{i=1}^{N}$ are
called the Kohn-Sham eigenfunctions or orbitals. These eigenfunctions
define the electron density $\rho(\vr)$, which in
turn defines the Kohn-Sham Hamiltonian 
\begin{equation}
H[\rho] = -\frac12 \Delta + \Vhxc[\rho] + \Vion,
\label{eqn:ksham}
\end{equation}
where $\Delta$ is the Laplacian operator for characterizing the kinetic
energy of electrons, $\Vhxc[\rho]$ is a nonlinear
function of $\rho$ which includes the 
electro-static interaction (Hartree) potential among electrons (h) and the
exchange-correlation potential (xc), and $\Vion$ is the electron-ion interaction
potential, which is independent of $\rho$.  We denote by $\Omega$ the
global computational domain, and for simplicity we assume each
eigenfunction $\psi_{i}$ has periodic boundary conditions on $\Omega$.
This nonlinear eigenvalue problem is the key problem to be solved in the
Kohn-Sham density functional theory
(KSDFT)~\cite{HohenbergKohn1964,KohnSham1965}, which is the most widely
used electronic structure theory for studying properties of molecules,
solids and other nano structures.  

Since the eigenvalue problem (\ref{eqn:KS}) is nonlinear, it
is often solved iteratively by a class of algorithms called {\em
self-consistent field iterations} (SCF)~\cite{Martin2004}. At each SCF
step, a linear eigenvalue problem with a fixed Kohn-Sham Hamiltonian
defined by a fixed electron density $\rho$ \eqref{eqn:ksham} is solved.
The solution to this linear eigenvalue problem is used to update
$\rho$ and $H$ in the SCF iteration.  Solving \eqref{eqn:ksham} is the most
computationally expensive part of the SCF iteration. Although the
asymptotic complexity of the computation with respect to the number of
atoms depends on the algorithm used to solve the algebraic eigenvalue
problem, the prefactor, which is related to the number of basis
functions per atom, is characterized by how the problem is discretized.
In this paper, we consider a discretization scheme in which 
an eigenfunction of $H[\rho]$ is expressed as a
linear combination of basis functions that have localized nonzero support.
The use of these localized basis functions yields a compact and yet
sparse representation of the Kohn-Sham Hamiltonian, so that a relatively
small prefactor in SCF iteration complexity can be achieved.

The generation of localized basis functions for discretizing the
Kohn-Sham problem is described in~\cite{LinLuYingE2012}.  The basic idea
is to partition the global domain into a number of subdomains (called
elements), and solve the Kohn-Sham problem locally around each element to
generate local basis functions.
The constructed basis functions are not continuous across the boundaries
of different elements.  Therefore,  we use the discontinuous Galerkin
(DG) method to construct a finite dimensional Kohn-Sham Hamiltonian
represented by these types of discontinuous basis functions.  

Furthermore, these basis functions are modified in each SCF cycle as the
electron density $\rho$ and consequently the Kohn-Sham Hamiltonian is
changed.  Hence these basis functions are called adaptive local basis
functions (ALB) because they adapt to the changes in $\rho$ and $H[\rho]$,
which include all the information of the atomic configuration and the
electronic structure.  The use of ALB combines the systematically improvable 
standard discretization methods such as the planewave
method~\cite{PayneTeterAllenEtAl1992}, the finite difference
method~\cite{ChelikowskyTroullierSaad1994}, and the finite element
method~\cite{TsuchidaTsukada1995} with the small number of degrees of
freedom of ``mesh free'' basis sets such as numerical atomic orbitals and
Gaussian type
orbitals~\cite{Ozaki:03,BlumGehrkeHankeEtAl2009,Junquera:01,ChenGuoHe2010,AndersenSaha2000,QianLiQiEtAl2008}.

We have already demonstrated the effectiveness of ALB for spatially
homogeneous systems such as disordered bulk Na and Si systems studied in
Ref.~\cite{LinLuYingE2012}. For these systems, the same number of local
basis functions are constructed in each element.  For inhomogeneous systems
such as large molecules and surfaces systems, there are usually large
vacuum regions in the computational domain. It is conceivable that one does
not need to compute too many eigenfunctions of the Kohn-Sham Hamiltonian
restricted to these regions to generate the adaptive local basis functions.
By reducing the number of basis functions generated from these regions, we
can reduce the total number of basis functions.  In this paper, we develop
an adaptive refinement strategy which allows different numbers of localized
basis functions to be generated on different elements.  The decision of how
many basis functions to generate on each element will be guided by an
\textit{a posteriori} error estimator. We demonstrate that the resulting
non-uniform generation of ALB functions is highly efficient for KSDFT
calculations on inhomogeneous systems.

A posteriori error estimates for solutions to elliptic partial differential
equations (PDEs) using DG methods have been established in
~\cite{HoustonSchotzauWihler2007,KarakashianPascal2003,SchotzauZhu2009}.
The authors considered error estimation for polynomial basis functions, in
the context of both $h$-refinement and $hp$-refinement.  The most relevant
work to our study appeared recently for the eigenvalue problem of the
Laplacian operator~\cite{GianiHall2012}. Recent work in applying DG
$hp$-refinement to solve eigenvalue problems arising from the analysis of
photonic crystals can be found in~\cite{Giani2012}.  The a posteriori error
estimator proposed in~\cite{GianiHall2012} is residual-based.  
Alternative approaches based on flux reconstruction have also
been proposed~\cite{BraessPillweinSchoberl2009,ErnNicaiseVohralik2007,LuceWohlmuth2004,Repin2008}.
The key tool
used in the analysis in~\cite{GianiHall2012} is the decomposition of the
discontinuous solution to an eigenvalue problem into an $H^{1}$
component in the global domain (the conforming part) and a remaining
component (the non-conforming part).  The construction of the a
posteriori error estimator depends explicitly on the analytic properties
of the broken polynomial space used in $hp$-refinement.  This type of
technique cannot be directly applied to non-polynomial basis functions
such as ALBs for the Kohn-Sham problem. 

We also acknowledge that the use of a posteriori error estimates for
solving eigenvalue problems discretized with a continuous basis set has
also been largely investigated, see \eg
Refs.~\cite{Verfurth1996,BeckerRannacher2001,Larson2000,DuranPadraRodriguez2003}.
Recently, adaptive refinement in KSDFT calculations has also been proposed
in the context of the finite element
method~\cite{DaiXuZhou2008,ChenDaiGongEtAl2013} and the
finite volume method~\cite{DaiGongYangEtAl2011}.

The contribution of this paper is twofold: 1) We illustrate the
theory of residual-based a posteriori error estimates for solving linear
eigenvalue problems using non-polynomial basis functions, and 2) We present
numerical results of using the residual-based a posteriori error estimates
to solve the nonlinear Kohn-Sham equations efficiently. From a theoretical
point of view, we use the same strategy as that employed
in~\cite{GianiHall2012} to develop residual-based a posteriori error
estimates for solving KSDFT using ALBs. Besides the Laplacian term, our
estimator takes into account the presence of a non-constant potential term
as required in KSDFT. To address the lack of analytical expressions for the
basis functions, we propose a set of assumptions on the function space
spanned by a basis set which would render the resulting a posteriori error
estimator an upper bound for the errors of both eigenvalues and
eigenvectors, up to terms which are of higher order in the context of
standard $hp$-refinement.  The main limitation of the current approach is
that it is difficult to directly verify the postulated assumptions for a
specific non-polynomial basis set such as the ALBs, since the a priori
error analysis of ALBs is not yet available. This also makes it difficult
to verify that the neglected terms are indeed of higher order than the
estimator, and to prove that the estimator also gives a lower bound for
the errors of eigenvalues and eigenvectors.

From a numerical point of view, the results are encouraging.  As a first
attempt to apply our analysis to practical calculations, we use the same
form of a posteriori error estimator as that given in~\cite{GianiHall2012},
but reinterpret $p$ as the number of ALBs rather than the polynomial
degree.  The numerical results from 3D KSDFT calculations indicate that the
a posteriori error estimator captures the spatial inhomogeneity of the
system, and therefore gives a promising approach to improving the
accuracy of solutions to KSDFT without increasing the computational cost.

The rest of the manuscript is organized as follows.  In
Section~\ref{sec:DG}, we introduce the DG framework for solving KSDFT, and
the construction of the ALBs.  Section~\ref{sec:apost} is devoted to the
derivation of the residual-based a posteriori error estimator for the ALBs,
as well as an explanation of the non-uniform refinement strategy.  The
effectiveness of the non-uniform refinement strategy is verified in
Section~\ref{sec:numerical3D} by applying the refinement strategy to the
solution of the Kohn-Sham problem for a quasi-2D aluminum surface and a 3D
graphene oxide in water system. The conclusion and discussion of future
work on refining the a posteriori error estimator are given in
Section~\ref{sec:conclusion}.  The details of the proofs used in
Section~\ref{sec:apost} are provided in the Appendix.

\section{Preliminaries}\label{sec:DG}

\subsection{Discontinuous Galerkin framework for KSDFT}

In a standard Galerkin method, we seek the solution to the 
Kohn-Sham nonlinear eigenvalue problem by working with its
weak form 
\begin{equation}
\langle v, H[\rho]u_i \rangle_\Omega = 
\varepsilon_{i} \langle v, u_i \rangle_\Omega,
\label{eqn:weakbase}
\end{equation}
where $\langle \cdot, \cdot \rangle_\Omega$ is an appropriately
chosen inner product defined on the global domain $\Omega$, and $v$ is a
test function.  To comply with
standard notation in the DG analysis, we use $u_{i}(\vr)\equiv
\psi_{i}(\vr)$ to represent the $i$th Kohn-Sham orbital corresponding to
the eigenvalue $\varepsilon_{i}$.

For example, we may choose $\langle \cdot, \cdot \rangle_\Omega$
to be the standard $L_2$ inner product 
\begin{equation}
	\average{u,v}_{\Omega}=\int_{\Omega} u^{\ast}(\vr) v(\vr) \ud \vr.
	\label{eq:uvprod}
\end{equation}
with the induced norm $\| \cdot \|$ defined by
\[
\norm{u}_{\Omega} = \average{u,u}_{\Omega}^{\frac12}.
\]

Both the approximate eigenfunctions $u_i$ and the test function $v$
must be chosen from an appropriate function space so that 
the weak form \eqref{eqn:weakbase} is well defined. For example, if
we let $L^2(\Omega)$ be the space of square integrable functions 
on $\Omega$, and $H^{1}(\Omega)$ be the space of functions in 
$L^{2}(\Omega)$ with gradient in $[L^{2}(\Omega)]^{d}$ where $d$ 
is the spatial dimension of the system, $u_i$ and $v$ can be
chosen from $H^{1}_{\pi}(\Omega)$, the subspace of 
$H^{1}(\Omega)$ functions with periodic boundary conditions.  

Although KSDFT is formulated as a nonlinear eigenvalue problem, the a
posteriori error estimator developed in this paper is for a linear
eigenvalue problem, with the linear Hamiltonian operator obtained from
each step of the SCF iteration.
For a fixed $\rho$ we define an effective potential
\begin{equation}
		V_{\eff}[\rho]=\Vhxc[\rho] + \Vion.
	\label{}
\end{equation}
$V_{\eff}[\rho]$ defines an effective Hamiltonian operator
$H_{\eff}[\rho]=-\frac12 \Delta + V_{\eff}[\rho]$. For brevity, we will
omit the dependence on $\rho$.
The effective Kohn-Sham potential $V_{\eff}$ is, in general, a symmetric
operator, and thus satisfies
\begin{equation}
	\average{u,V_{\eff} v}_{\Omega} = \average{V_{\eff}u,
	v}_{\Omega}, \quad \forall u,v\in
	H^{1}_{\pi}(\Omega).
	\label{}
\end{equation}
When $\Vion$ is constructed using the pseudopotential
method~\cite{TroullierMartins1991}, $V_{\eff}$ is bounded from above and
from below, thereby ensuring that \eqref{eqn:weakbase} is indeed well
defined when $u_i, v \in H^{1}_{\pi}(\Omega)$.

Among the various Galerkin methods, the discontinuous Galerkin (DG)
methods relax the continuity constraint on basis
functions, and provide flexibility in choosing the basis
set.  For instance, the adaptive local basis functions are given by
solutions to Kohn-Sham problems on local subdomains, and are naturally
discontinuous in the global domain.
The DG methods have been developed for
efficiently solving various types of PDEs, and there is an abundant
literature about them; see \eg Refs.~\cite{BabuskaZlamal:73, Wheeler:78, Arnold1982,
CockburnKarniadakisShu:00, CockburnShu:01,
ArnoldBrezziCockburnEtAl2002, StammWihler2010}.  In the ALB
approach~\cite{LinLuYingE2012},  the interior penalty (IP) method
\cite{BabuskaZlamal:73,Arnold1982} is used to discretize the
Kohn-Sham Hamiltonian operator.

In this paper, we assume that the global domain $\Omega$ is a
$d$-dimensional rectangular domain, and that it is partitioned into a collection
of uniform rectangular subdomains, denoted by
\begin{equation}
  \mc{T} = \{K_1, K_2, \cdots, K_M \}.
\end{equation}
Each subdomain $K_i\in \mc{T}$ is called an element of $\Omega$. 
Associated with each $K_i$ is an inner product $\average{\cdot,\cdot}_{K_i}$ defined by simply replacing $\Omega$ in \eqref{eq:uvprod} with $K_i$.
The induced norm of a function $u$ defined on $K_i$ is denoted
by $\|u\|_{K_i}\equiv \average{u,u}_{K_i}^{1/2}$.

We refer to $F$ as a face of $\mc{T}$ if $F\subset \partial K$ is a face
of the $d$-dimensional rectangular subdomain $K\in\mc{T}$. 
We refer to $F$ as an interior face of $\mc{T}$ if $F=\partial
K^{+}\cap \partial K^{-}$ for some neighboring elements
$K^{+},K^{-}\in\mc{T}$.  We note that when using periodic boundary conditions, all
faces are interior faces.  Other types of boundary conditions, such as
Dirichlet boundary condition, can be used as well, and the discussion
below can be naturally generalized to cases with boundary faces. 

In the DG framework, $u_i$ is constructed as a linear combination of
$M$ local basis functions that form a subspace from 
which $v$ is drawn. Because each basis function has support contained in one of the elements, and is not necessarily continuous
over the entire domain $\Omega$, an appropriate inner product
should be chosen to make the weak form~\eqref{eqn:weakbase} well
defined.

To define such an inner product, let us first define the discontinuous
function spaces  
\begin{equation}
  \mc{V}(\mc{T})=\{v\in L^{2}(\Omega): v\vert_{K}\in H^{1}(K), \ K\in \mc{T}\}.
\label{eq:Vspace}
\end{equation}
The inner product associated with $\mc{V}(\mc{T})$ is
\begin{equation}
\average{u, v}_{\mc{T}} = \sum_{K \in \mc{T} }
	\average{u, v}_{K} \equiv \sum_{ K \in \mc{T} }
	\int_{K} u^{\ast}(\vr) v(\vr) \ud \vr,  u,v\in \mc{V}(\mc{T}).
	\label{eq:dotT}
\end{equation}

We denote by $\mc{S}$ the collection of all the faces $\{F\}$. 
An inner product associated with functions defined on $\mc{S}$ is
\begin{equation}
\average{u, v}_{\mc{S}} = \sum_{ F \in \mc{S} }
\average{u, v}_{F} \equiv \sum_{ F \in \mc{S} }
\int_F u^{\ast}(\vr) v(\vr) \ud s(\vr). 
\label{eq:SdotT}
\end{equation}

The gradient of a function defined on $K$ has $d$ components. 
A natural inner product for this type of vector function space is
\begin{equation}
\average{\vq,\vw}_{K} = \int_{K} \vq^{\ast}(\vr) \cdot \vw(\vr) \ud \vr,
\label{eq:vecdot}
\end{equation}
for $\vq, \vw \in [L^2(K)]^d$.

If we define the vector function space $\mc{W}(\mc{T})\equiv
[L^{2}(\Omega)]^{d}$, then
\begin{equation}
\average{\vq, \vw}_{\mc{T}} = \sum_{ K \in \mc{T} }
	\average{\vq, \vw}_{K} \equiv  \sum_{ K \in \mc{T} }
	\int_{K} \vq^{\ast}(\vr) \cdot \vw(\vr) \ud \vr, \vq,\vw\in
	\mc{W}(\mc{T})
	\label{eq:WdotT}
\end{equation}
is a natural inner product associated with $\mc{W}(\mc{T})$.

Similarly, we can define an inner product for vector functions defined
only on the collection of faces $\mc{S}$, i.e.
\begin{align}
  & \average{\vq, \vw}_{\mc{S}} = \sum_{ F \in \mc{S} }
	\average{\vq, \vw}_{F}\equiv \sum_{ F \in \mc{S} }
	\int_F \vq^{\ast}(\vr) \cdot \vw(\vr) \ud s(\vr). 
	\label{eq:WdotS}
\end{align}

The inner products defined by \eqref{eq:dotT}, \eqref{eq:SdotT}, 
\eqref{eq:WdotT} and \eqref{eq:WdotS} induce the following norms
\begin{equation}
  \norm{u}_{\mc{T}}^2 = \average{u,u}_{\mc{T}},\quad
  \norm{u}_{\mc{S}}^2 =
  \average{u,u}_{\mc{S}},\quad
  \norm{\vq}_{\mc{T}}^2 = \average{\vq,\vq}_{\mc{T}},\quad
  \norm{\vq}_{\mc{S}}^2 =
  \average{\vq,\vq}_{\mc{S}}.
  \label{eq:normTS}
\end{equation}

If $u_i$ and $v$ are chosen from $H_\pi^1(\Omega)$ in a standard
Galerkin method, the weak form \eqref{eqn:weakbase} can be written as
\begin{equation}
A(u_{i},v) + \average{V_{\eff}u_{i},v}_\Omega = \varepsilon_{i}\average{u_{i},v}_{\Omega},
\label{eqn:weakeig}
\end{equation}
where the bilinear form $A(u,v)$ is defined by
\begin{equation}
  A(u,v)\equiv \frac{1}{2}\average{\nabla u,\nabla v}_{\Omega},
  \label{eqn:bilinear}
\end{equation}
and the eigenfunctions are orthonormal; $\average{u_i,u_j}=\delta_{i,j}$.

In a DG approach, when $u_i$ and $v$ are chosen from $\mc{V}(\mc{T})$, 
the bilinear form $A(u_i,v)$ includes 
$\frac{1}{2}\average{\nabla u_i,\nabla v}_{\mc{T}}$
and some additional terms that account for the discontinuity of the basis
functions. To define these terms, let $K^{+}$ and $K^{-}$ be two adjacent elements in $\mc{T}$, and
$F=\partial K^{+}\cap \partial K^{-}$ be the face shared by 
$K^{+}$ and $K^{-}$. The normal vectors on $F$ are defined to be 
$\vn^{+}$ and $\vn^{-}$, and point towards the exteriors of  $K^{+}$ and
$K^{-}$, respectively.
Let $v^{+}=v\vert_{K^{+}}$, $v^{-}=v\vert_{K^{-}}$, 
$\vq^{+}=\vq\vert_{K^{+}}$ and $\vq^{-}=\vq\vert_{K^{-}}$. Then the
average operator $\mean{\cdot}$ on $F$ is defined as
\begin{align}
	&\mean{v} = \tfrac{1}{2} (v^{+} + v^{-}),\quad \mbox{on~} F,
	&\mean{\vq} = \tfrac{1}{2} (\vq^{+} + \vq^{-}),\quad \mbox{on~} F.
	\label{eqn:averageOp}
\end{align}
We define the jump operator $\jump{\cdot}$ on $F$ by 
\begin{align}
	&\jump{v} = v^{+} \vn^{+} + v^{-} \vn^{-},\quad \mbox{on~} F,
	&\jump{\vq} = \vq^{+} \cdot \vn^{+} + \vq^{-} \cdot \vn^{-},\quad
	\mbox{on~} F.
	\label{eqn:jumpOp}
\end{align}

To derive a finite dimensional representation of the weak form 
\eqref{eqn:weakeig}, let us assume that we have chosen for each
element $K$ a set of orthonormal basis functions
$\{\varphi_{K,j}\}_{j=1}^{J_K}$, where $\varphi_{K,j}\in H^{1}(K)$ and
$J_K$ is the number of basis functions in $K$. The collection of 
basis functions for each element is denoted by $\mc{J} =
\{J_1,\cdots,J_{M}\}$, where $M$ is the total number of elements.
We extend each $\varphi_{K,j}$ to the whole computational domain
$\Omega$ by setting it to $0$ on $\Omega\backslash K$.  
Then the function space $\mc{V}_{\mc{J}}(\mc{T})$, which is a subspace of
$\mc{V}(\mc{T})$ used to approximate each $u_{i}$, is defined as
\begin{equation}
  \mc{V}_{\mc{J}}(\mc{T}) = 
	\spanop\{ \varphi_{K,j},\, K \in \mc{T},\, j = 1, \ldots, J_K \}.
\end{equation}
For reasons that will be made clear in subsequent analysis, we assume that any
function that is piecewise constant over element boundaries is in 
$\mc{V}_{\mc{J}}(\mc{T})$, \ie 
\begin{equation}
  \mathbf{1}_K \in \spanop	\{\varphi_{K,j}\}_{j=1}^{J_K}, \quad \forall K\in
	\mc{T},
	\label{eqn:constantmode}
\end{equation}
where $\mathbf{1}_K$ denotes the characteristic function on $K$.

A particular example of $\mc{V}_{\mc{J}}(\mc{T})$ is the \textit{broken
polynomial space}, which takes $\varphi_{K,j}$ to be a polynomial on
$K$.  A broken polynomial space is said to be of order $p$  if
$\{\varphi_{K,j}\}_{j=1}^{J_K}$, restricted to $K$, consists of
polynomials of degree up to $p$.

The derivative space associated with $\mc{V}_{\mc{J}}(\mc{T})$ is defined as
\begin{equation}
	\mc{W}_{\mc{J}}(\mc{T}) = \spanop\{ \nabla \varphi_{K,j},\, K \in
        \mc{T},\, j = 1, \ldots, J_K \}.
	\label{}
\end{equation}
It is clear that $\mc{W}_{\mc{J}}(\mc{T})\subset \mc{W}(\mc{T})$. It is
worth noting that in the broken polynomial case,
each component in the derivative space $\partial_{i} \varphi_{K,j},
i=1,\ldots,d$  is a polynomial of lower order, and therefore is in the
function space $\mc{V}_{\mc{J}}(\mc{T})$.  However, this property does
not hold for general basis sets, so that the functions in
  $\mc{V}_{\mc{J}}(\mc{T})$ and the components of the functions in
$\mc{W}_{\mc{J}}(\mc{T})$ may belong to different function spaces. 

Given the function space $\mc{V}_{\mc{J}}(\mc{T})$ we formulate the
eigenvalue problem~\eqref{eqn:weakeig} in the discontinuous Galerkin
framework using the interior penalty method described in
Ref.~\cite{Arnold1982}. Namely, find the first N approximate eigenpairs
$(\varepsilon_{i,\mc{J}},u_{i,\mc{J}})\in \mathbb{R}\times
\mc{V}_{\mc{J}}(\mc{T})$ such that
\begin{equation}
	A_{\mc{J}}(u_{i,\mc{J}},v_{\mc{J}}) +
	\average{V_{\eff}u_{i,\mc{J}},v_{\mc{J}}}_{\mc{T}} =
	\varepsilon_{i,\mc{J}}\average{u_{i,\mc{J}},v_{\mc{J}}}, \quad
	\forall v_{\mc{J}} \in \mc{V}_{\mc{J}}(\mc{T}).
	\label{eqn:weakDGeig}
\end{equation}
and $\average{u_{i,\mc{J}},u_{j,\mc{J}}}_{\mc{T}}=\delta_{ij}$.
Here, for $u_{\mc{J}},v_{\mc{J}} \in \mc{V}_{\mc{J}}(\mc{T})$,
the bilinear form $A_{\mc{J}}(u_{\mc{J}},v_{\mc{J}})$ is given by 
\begin{equation}
	\begin{split}
	A_{\mc{J}}(u_{\mc{J}},v_{\mc{J}}) =& \frac{1}{2}\average{\nabla u_{\mc{J}},\nabla
	v_{\mc{J}}}_{\mc{T}} 
	- \frac{1}{2} \average{\mean{\nabla u_{\mc{J}}},\jump{v_{\mc{J}}}}_{\mc{S}}
	- \frac{1}{2} \average{\mean{\nabla
	v_{\mc{J}}},\jump{u_{\mc{J}}}}_{\mc{S}}\\
	&+ \sum_{F\in \mc{S}} \alpha(J_F)
	\average{\jump{u_{\mc{J}}},\jump{v_{\mc{J}}}}_{F},
	\end{split}
	\label{eqn:bilinearDG}
\end{equation}
where $\alpha(J_F)$ denotes the interior penalty parameter on
the face $F$ which remains to be determined.  This parameter penalizes
discontinuities of functions across element faces.  The values
$\{\alpha(J_F)\}$ must be large enough to guarantee that the bilinear
form $A_{\mc{J}}$ is coercive on $\mc{V}_{\mc{J}}$. We provide one
sufficient condition on $\alpha(J_{F})$ for general non-polynomial basis
functions in Eq.~\eqref{eqn:alphacondition} and in
Corollary~\ref{cor:ADGcoercive}.

An alternative way to define $\alpha(J_{F})$, which is used in this
paper, is to first define a penalty parameter $\alpha(J_{K})$ for each
element $K\in \mc{T}$.  Then for each face $F=\partial
K^{+}\cap \partial K^{-}$, we define
\begin{equation}
  \alpha(J_{F}) \equiv \max\{\alpha(J_{K^{+}}),\alpha(J_{K^{-}})\}.
  \label{eqn:alphaJF}
\end{equation}
Again $\{\alpha(J_{K})\}$ should be chosen such that $A_{\mc{J}}$ is
coercive on $\mc{V}_{\mc{J}}$.

~

\begin{remark}
In the context of standard $hp$-refinement, let $h_K$ be the diameter of
the element $K$, and let $p_K$ be highest degree of the polynomials used in
$K$.  Then the choice 
\begin{equation}
  \alpha(p_{K}) = \frac{\gamma p_K^2}{h_K}
	\label{eqn:alphaHP}
\end{equation}
leads to a coercive bilinear form $A_{\mc{J}}$ 
for a sufficiently large positive
value of $\gamma$ which is independent of $p_K$ and
$h_K$~\cite{Arnold1982,ArnoldBrezziCockburnEtAl2002}. 
The choice of $\gamma$ is in general system-dependent.
Condition~\eqref{eqn:alphacondition} generalizes
Eq.~\eqref{eqn:alphaHP}.
\end{remark}

~

To solve Eq.~\eqref{eqn:weakDGeig} numerically, we write
\begin{equation}
	u_{i,\mc{J}} = \sum_{K \in \mc{T}} \sum_{j=1}^{J_K} c_{K, j; i}
  \varphi_{K, j}.
\end{equation}
By choosing $v_{\mc{J}}=\varphi_{K',j'}$, we turn Eq.~\eqref{eqn:weakDGeig}
into the following matrix eigenvalue problem
\begin{equation}
	\sum_{K,j} H^{\DG}_{K,j; K', j'} c_{K,j;i}
  = \varepsilon_{i,\mc{J}} c_{K',j';i},
	\label{eqn:DGeig}
\end{equation}
where the matrix elements for the DG Hamiltonian operator are given by
\begin{equation}
		H^{\DG}_{K,j; K', j'} = A_{\mc{J}}(\varphi_{K,j},\varphi_{K',j'})
		+ \average{V_{\eff}\varphi_{K,j},\varphi_{K',j'}}_{\mc{T}}.
	\label{eqn:DGham}
\end{equation}
Since we choose the basis functions to be orthonormal, 
\eqref{eqn:DGeig} is a standard eigenvalue problem.
Because the basis functions have local support,
the DG Hamiltonian matrix~\eqref{eqn:DGham} is a sparse matrix,
and its eigenvalues and eigenvectors can be computed 
efficiently on high performance computers.

\subsection{Adaptive local basis functions}

The local basis functions $\{\varphi_{K,j}\}_{j=1}^{J_K}$ which we use to 
discretize the Kohn-Sham problem are constructed as
follows.  For each $K\in \mc{T}$, we introduce an associated {\em extended
element} $Q_{K} \supset K$, with $Q_{K} \backslash K$ a buffer
region surrounding $K$.  We define $V_{\eff}^{Q_{K}}=V_{\eff}\vert_{Q_{K}}$
to be the restriction of the effective potential at the current SCF
step to $Q_{K}$, and solve the local eigenvalue problem
\begin{equation} 
	\begin{split}
		&\left(-\frac12 \Delta + V_{\eff}^{Q_{K}}\right) \wt{\varphi}_{K,j} = \lambda_{K,j}
		\wt{\varphi}_{K,j},\\
		&\int_{Q_{K}} \wt{\varphi}^{*}_{K,j}(\vr)\wt{\varphi}_{K,j'}(\vr) =
		\delta_{jj'}.
	\end{split}
	\label{eqn:localproblem} 
\end{equation}
The lowest $J_{K}$ eigenvalues
$\{\lambda_{K,j}\}_{j=1}^{J_{K}}$ and the corresponding eigenfunctions
$\{\wt{\varphi}_{K,j}\}_{j=1}^{J_{K}}$ are computed.
We then restrict $\{\wt{\varphi}_{K, j}\}_{j=1}^{J_{K}}$ from $Q_{K}$ to
$K$. The truncated vectors are not necessarily orthonormal. 
Therefore, we apply a singular value decomposition (SVD) to the set of 
truncated eigenvectors to obtain $\{\varphi_{K,j}\}_{j=1}^{J_K}$.
We then set each $\varphi_{K,j}$ equal to zero outside of $K$, so that it
is in general discontinuous across the boundary of $K$.

There are a number of possible ways to set the boundary conditions 
for the local problem~\eqref{eqn:localproblem}.  
In practice, we use either Dirichlet or periodic boundary conditions 
for all of the eigenfunctions 
$\{\wt{\varphi}_{K,j}\}_{j=1}^{J_{K}}$ in $Q_{K}$.  
It is not yet clear what are the optimal boundary conditions.
Our ultimate goal is not to solve ~\eqref{eqn:localproblem} 
but to use the approximate solution of~\eqref{eqn:localproblem} 
to construct localized basis functions, and any choice of boundary
conditions that guarantees $\left(-\frac12 \Delta + V_{\eff}^{Q_{K}}\right)$ 
to be a self adjoint operator on $Q_{K}$ will generate a complete basis set on
$K$. In this sense, the choice of boundary conditions is not
critical for the purposes of this paper.  
The size of each extended element should be large enough to capture
the effect of the chemical environment, but should not be so large that
it makes the local problem costly to solve.

We solve the local eigenvalue problem ~\eqref{eqn:localproblem}
using a planewave discretization scheme.  It should be noted that the use of
a finite dimensional discretization for solving
Eq.~\eqref{eqn:localproblem} would introduce additional error on top of
that incurred by using a finite number of adaptive local basis functions in each
element.  Numerical results indicate that inaccurately solved adaptive
local basis functions are less effective in solving the Kohn-Sham
equations.  In our numerical results, we always use a sufficiently fine
planewave discretization to ensure that the error caused by the planewave
discretization is negligible compared to the error introduced by the use of a
finite number of adaptive local basis functions.  We find that the
energy cutoff of such planewave discretizations is comparable to or
slightly larger than the energy cutoff of converged calculations obtained
from standard electronic structure software packages such as ABINIT.

For a typical choice of grid used in practice, the elements are chosen to
be of the same size. Numerical results indicate that it is most efficient
if each element contains on average a few atoms. The grid does not need
to be updated even if the atomic configuration is changed, as in the case of
structure optimization and molecular dynamics. Dangling bonds may form when
atoms are present on the faces of the extended elements, but we emphasize
that these dangling bonds do not need to be passivated by introducing
auxiliary atoms near the faces of the extended
elements~\cite{ZhaoMezaWang2008}.  This is because the potential is not
obtained self-consistently within the extended element $Q_{K}$, but instead
from the restriction of the screened potential in the global domain
$\Omega$ to $Q_{K}$ in each SCF iteration, which mutes the catastrophic
damage of the dangling bonds. On the other hand, one can still introduce
auxiliary atoms near the faces of the extended elements as in the quantum
mechanics/molecular mechanics (QM/MM)~\cite{WarshelLevitt1976} approach to
achieve a better approximation of the exact boundary conditions for
$\wt{\varphi}_{K,j}$ than the Dirichlet or periodic boundary condition used
here. This is beyond the scope of this manuscript, but can be explored in
future work.

\section{Residual-based a posteriori error estimator}\label{sec:apost}

In this section, we develop a posteriori error estimates for ALB sets.
Although ALB sets are constructed from solutions to the Kohn-Sham equations
restricted to local domains, so far little can be derived analytically in
order to obtain the approximation properties of the function space spanned
by the ALBs. Therefore we do not restrict our scope to ALBs, but attempt to
directly develop a posteriori error estimates for general non-polynomial
basis sets.  In the ALB method, the size of each element is usually chosen
to include one or several atoms (usually less than $10$ per element) to
balance efficiency and accuracy.  These basis functions are not scale
invariant, and the meaning of $h$-refinement is not straightforward.
Therefore we do not consider $h$-refinement here and the number of
elements $M$ is fixed. We only consider the refinement of the
number of basis functions $J_{K}$ in each element $K\in \mc{T}$.  This
is analogous to $p$-refinement.

\subsection{Theory}

In order to address the apparent difficulty associated with the
analysis of general non-polynomial basis sets, we need a set of assumptions 
on the function space spanned by the ALB functions. These assumptions
directly generalize the results in~\cite{GianiHall2012} for
polynomial basis functions. However, we acknowledge that the assumptions
are not easy to verify directly for a given specific basis set such as
the ALB set. We will defer the detailed verification of these
assumptions for the ALB set to our future work, and focus in this paper
on the type of error estimates we can obtain and use to adaptively
refine the basis set if these assumptions hold. In the following
discussion, ``$a \lesssim b$'' is taken to mean that ``$a$ is less than
or equal to $b$ up to a constant scaling''. The scaling factor
$a/b$ depends on the dimensionality $d$, the total number of elements
$M$, the size of each element $K$, and the size of the global
domain $\Omega$.  In particular, the constant is independent of the
number of basis functions $J_{K}$ for each $K\in \mc{T}$.

\begin{assumption}
	\label{asp:regularity}
	\begin{enumerate}
		\item (Approximation properties) For any $u\in H^{1}_{\pi}(\Omega)$, there exists a function
			$u_{\mc{J}} \in \mc{V}_{\mc{J}}(\mc{T})$ such that for each $K \in \mc{T}$
			\begin{align}
				\norm{u-u_{\mc{J}}}^2_{K}&\lesssim \gamma_{1}(J_{K})
				\norm{\nabla u}^2_{K},\label{eqn:gamma1}\\
				\norm{u-u_{\mc{J}}}^{2}_{\partial K} &\lesssim
				\gamma_{2}(J_{K}) \norm{\nabla u}^2_{K},\label{eqn:gamma2}\\
				\norm{\nabla(u-u_{\mc{J}})}^{2}_{K} &\lesssim
				\norm{\nabla u}^2_{K},\label{eqn:gamma3}
			\end{align}
			where the constants
			$\gamma_{1}$ and $\gamma_{2}$
			depend only on $J_{K}$. 
      For a given $\mc{J} = \{J_1,\cdots,J_{M}\}$, we define
			$\gamma_{1,\mc{J}}=\max_{K\in \mc{T}} \gamma_{1}(J_{K})$. 
      Furthermore,
			$\gamma_{1}(\cdot)$ and $\gamma_{2}(\cdot)$
      are bounded from above.  The parameter $\gamma_{2}$  satisfies
			\begin{equation}
				1 \lesssim \gamma_{2}(J_K)\alpha(J_K).
				\label{eqn:constantCondition2}
			\end{equation}
      
		\item (Interpolation properties) There exists an interpolation operator $I_{\mc{J}}:
			\mc{V}_{\mc{J}}(\mc{T})\to H^{1}_{\pi}(\Omega)$ such that for any
			$u_{\mc{J}}\in \mc{V}_{\mc{J}}(\mc{T})$
			\begin{equation}
				\sum_{K\in \mc{T}} \norm{\nabla (u_{\mc{J}} -
				I_{\mc{J}} u_{\mc{J}})}^{2}_{K} \lesssim
				\sum_{F\in \mc{S}} \gamma_{2}(J_{F}) \alpha^{2}(J_{F})
				\norm{\jump{u_{\mc{J}}}}_{F}^{2}.
				\label{eqn:IJcondition1}
			\end{equation}
      Here for $F=\partial K^{+}\cap \partial K^{-}$, we write
      \[
      \gamma_{2}(J_{F}) \equiv
      \max\{\gamma_{2}(J_{K^{+}}),\gamma_{2}(J_{K^{-}})\}, \quad
      \alpha(J_{F}) \equiv
      \max\{\alpha(J_{K^{+}}),\alpha(J_{K^{-}})\}.
      \]
		\item (Inverse trace inequalities) $\forall u_{\mc{J}}\in
			\mc{V}_{\mc{J}}(\mc{T})$,
			\begin{equation}
				\norm{u_{\mc{J}}}^2_{\partial K} \lesssim 
				\gamma_{2}(J_K) \alpha^{2}(J_{K}) \norm{u_{\mc{J}}}^2_{K},\quad
				\norm{\nabla u_{\mc{J}}}^2_{\partial K} \lesssim 
				\gamma_{2}(J_K) \alpha^{2}(J_{K}) \norm{\nabla
				u_{\mc{J}}}^2_{K}.
				\label{eqn:inversetrace}
			\end{equation}
	\end{enumerate}
\end{assumption}

The constants in the estimates~\eqref{eqn:gamma1} and~\eqref{eqn:gamma2}
are important components of the a posteriori error estimator. For each
element $K$, we define the local estimator $\eta_{i,K}$ for estimating
the error of the $i$th eigenpair
$(\varepsilon_{i,\mc{J}},u_{i,\mc{J}})$ as
\begin{equation}
	\label{eqn:localestimator} 
	\eta_{i,K}^2 = \eta_{i,R_K}^2 + \eta_{i,G_K}^2 + \eta_{i,V_K}^2 
\end{equation}
where
\begin{align}
	\label{eqn:localestimatorterms} 
	\eta_{i,R_K}^2 &= \gamma_{1}(J_K)
	\norm{\left(-\frac12 \Delta + V_{\eff} - \varepsilon_{i,\mc{J}}\right) 
	u_{i,\mc{J}}}_{K}^2 \\ 
	\eta_{i,G_K}^2 &= \frac{1}{4}
	\sum_{F \subset \partial K} \gamma_{2}(J_F) \norm{\jump{\nabla
	u_{i,\mc{J}}}}_{F}^2 \\ 
	\eta_{i,V_K}^2 &= \frac{1}{4} 
	\sum_{F \subset \partial K} \gamma_{2}(J_F) \alpha^2(J_F) 
	\norm{\jump{u_{i,\mc{J}}}}_{F}^2.
\end{align} 
Here $\eta_{i,R_K}$ measures the residual (R) of the $i$th computed
eigenfunction on the interior of $K$, $\eta_{i,G_K}$ measures the
discontinuity of the gradient (G) of the $i$th computed eigenfunction over
the faces of $K$, and $\eta_{i,V_K}$ measures the discontinuity of the value (V)
of the $i$th computed eigenfunction over the faces of $K$. The factor
$\frac14$ in $\eta_{i,G_{K}}$ and $\eta_{i,V_K}$ comes from the
$\frac12$ in front of the Laplacian operator in Eq.~\eqref{eqn:ksham},
as well as the double counting of face terms by replacing
$\sum_{F\in \mc{S}}$ with $\sum_{K \in \mc{T}}\sum_{F \subset \partial
K}$. $F \subset \partial K$ means that $F$ is a face
of the $d$-dimensional rectangular subdomain.  We can use the local
estimators to form a global estimator of the error in the $i$th computed
eigenfunction, given by

\begin{equation}
	\label{eqn:globalestimator} 
	\eta_i^2 = \sum_{K \in \mathcal{T}} \eta_{i,K}^2.  
\end{equation}

\begin{remark}
	For $hp$-refinement using the broken polynomial space, the constants
	$\gamma_{1}$ and $\gamma_{2}$ can be defined in terms of 
        $h_K$ and $p_K$~\cite{HoustonSchotzauWihler2007,GianiHall2012}, i.e.,
	\begin{equation}
		\gamma_{1}(p_K) = \frac{h_K^2}{p_K^2},\quad
		\gamma_{2}(p_K) = \frac{h_K}{p_K},\quad
		\alpha(p_K)     = \frac{\gamma p_K^2}{h_K}
		\label{eqn:constantHP}
	\end{equation}
	Furthermore, it can be shown that 
	\begin{equation}
    1 \lesssim \gamma p_K \le \gamma_{2}(p_K)\alpha(p_K),
	\end{equation}
	\ie , the inequality~\eqref{eqn:constantCondition2} holds.

\end{remark}

To quantify the error associated with the approximate solutions to the 
Kohn-Sham equations obtained from the DG approach, we need to define
the following energy norm, which is induced from the bilinear
form~\eqref{eqn:bilinearDG}, and a corresponding distance function. 

\begin{definition}[Energy norm]
  For any $u\in H^{1}(\Omega)\oplus \mc{V}_{\mc{J}}(\mc{T})$, 
  \begin{equation} 
    \label{eqn:energynorm}
    \normET{u}^2 := \sum_{K \in \mathcal{T}} \frac12 \norm{\nabla
    u_{K}}^2 +
    \sum_{F \in \mathcal{S}} \alpha(J_F) \norm{\jump{u}}_{F}^2.
  \end{equation}
\end{definition}

\begin{definition}
For $v_{\mc{J}} \in \mc{V}_{\mc{J}}(\mc{T})$ and a
finite dimensional subspace $\mc{P}\subset H^{1}_{\pi}(\Omega)$,
the distance between $v_{\mc{J}}$ and $\mc{P}$ in the energy norm is defined as
	\begin{equation}
		\dist(v_{\mc{J}},\mc{P})_{E,\mc{T}} \equiv \inf_{w\in \mc{P}}
		\normET{v_{\mc{J}}-w}.
		\label{eqn:energydist}
	\end{equation}
\end{definition}

We now first give the main result for measuring the accuracy of the
eigenfunctions using the a posteriori error estimator.

\begin{theorem}[Reliability of eigenfunctions]
	\label{thm:eigfunc}
	Denote by $M(\varepsilon_{i})$ the span of all 
	eigenfunctions corresponding to the eigenvalue $\varepsilon_{i}$ for
	the eigenvalue problem~\eqref{eqn:weakeig}.  Let
	$(\varepsilon_{i,\mc{J}},u_{i,\mc{J}})$ be a computed eigenpair
	corresponding to~\eqref{eqn:weakDGeig}, with $\varepsilon_{i,\mc{J}}$
	converging to the true eigenvalue $\varepsilon_{i}$. 
	If Assumption~\ref{asp:regularity} holds, then
	\begin{equation}
		\dist\left( u_{i,\mc{J}}, M(\varepsilon_{i})
		\right)_{E,\mc{T}} \lesssim \eta_{i} + 
    (1+\sqrt{\gamma_{1,\mc{J}}}) \inf_{\substack{u_{i}\in
                M(\varepsilon_{i}) \\ \norm{u_{i}}_{\Omega}=1}}\xi_{i},
		\label{eqn:eigfuncreliable}
	\end{equation}
        where
	\begin{equation}
		\xi_{i}^{2} \equiv \sum_{K\in \mc{T}}
		\xi_{i,K}^{2},\quad\mbox{and}\quad  
		\xi_{i,K} = \norm{\varepsilon_{i}u_{i}-\varepsilon_{i,\mc{J}}u_{i,\mc{J}}}_{K}
		+ \norm{V_{\eff}u_{i} - V_{\eff}u_{i,\mc{J}}}_{K},
		\label{eqn:xidef}
	\end{equation}
\end{theorem}

\begin{proof}
	We decompose a computed
	eigenfunction $u_{i,\mc{J}}$ into a conforming part and a remainder part
	\begin{equation}
		u_{i,\mc{J}} = u_{i,\mc{J}}^{c} + u_{i,\mc{J}}^{r},
		\label{eqn:splitu}
	\end{equation}
	where $u_{i,\mc{J}}^{c}=I_{\mc{J}} u_{i,\mc{J}}\in H^{1}_{\pi}(\Omega)$ is defined
	using the interpolation operator satisfying the
	condition~\eqref{eqn:IJcondition1}. 
	From the triangle inequality
	\begin{equation}
		\normET{u_{i}-u_{i,\mc{J}}} \le \normET{u_{i,\mc{J}}^{r}} + \normET{u_{i}-u_{i,\mc{J}}^{c}},
		\label{eqn:splituerror}
	\end{equation}
	it is sufficient to prove that both terms on the right hand side of
	Eq.~\eqref{eqn:splituerror} are bounded by $\eta_{i}$.
	Applying Lemma~\ref{lem:urbound} and ~\ref{lem:ucbound} to
	Eq.~\eqref{eqn:splituerror}, we have
	\begin{equation}
		\normET{u_{i}-u_{i,\mc{J}}} \lesssim \eta_{i} + 
    (1+\sqrt{\gamma_{1,\mc{J}}}) \xi_{i}.
		\label{eqn:uerrbound}
	\end{equation}
	The theorem follows from Eq.~\eqref{eqn:uerrbound} directly, once 
	we minimize $\xi_{i}$ among all normalized eigenfunction $u_{i}\in
	M(\varepsilon_{i})$. 
\end{proof}

Theorem~\ref{thm:eigfunc} illustrates that the error of each
eigenfunction attributed to the finite dimensional approximation
$\mc{V}_{\mc{J}}$ is bounded by the residual-based error estimator
$\eta_{i}$, up to a constant factor independent of $\mc{V}_{\mc{J}}$ and
a remaining term which is of high order in the context of
$hp$-refinement.  As will be shown below, Theorem~\ref{thm:eigfunc} is
also used for error estimates of the eigenvalues.  In the context of
KSDFT, the error of eigenvalues directly indicates the error of physical
observables such as total energies.

In order to develop a posteriori error estimates for eigenvalues, we
need to address the technical difficulty that the bilinear operator $A_{\mc{J}}$ is neither
coercive nor bounded with respect to the energy norm $\normET{\cdot}$ on
the space $H^{1}_{\pi}(\Omega)\oplus \mc{V}_{\mc{J}}(\mc{T})$. We
therefore need to extend the definition of the bilinear operator
$A_{\mc{J}}$ in Eq.~\eqref{eqn:weakDGeig}. The extended bilinear
operator is coercive and bounded on the joint space
$H^{1}_{\pi}(\Omega)\oplus \mc{V}_{\mc{J}}(\mc{T})$, as will be shown in
Lemma~\ref{lem:Atildebound}.
We introduce the lifting operator which is used in
Refs.~\cite{ArnoldBrezziCockburnEtAl2002,HoustonSchotzauWihler2007,GianiHall2012}
for $hp$-refinement,
and generalize it to non-polynomial basis set as follows.

\begin{definition}[Lifting operator]\label{def:lifting}
For any $v\in
H^{1}_{\pi}(\Omega)\oplus \mc{V}_{\mc{J}}(\mc{T})$, define
$\mc{L}v\in \mc{W}_{\mc{J}}(\mc{T})$ by
\begin{equation}
	\average{\mc{L}v,\vq_{\mc{J}}}_{\mc{T}} \equiv
	\average{\jump{v},\mean{\vq_{\mc{J}}}}_{\mc{S}},\quad \forall
	\vq_{\mc{J}}\in \mc{W}_{\mc{J}}(\mc{T}).
	\label{eqn:lifting}
\end{equation}
\end{definition}

With the lifting operator, for any $u,v\in H^{1}_{\pi}(\Omega)\oplus
\mc{V}_{\mc{J}}(\mc{T})$,
we define the following extended bilinear form $\wt{A}_{\mc{J}}(u,v)$
\begin{equation}
	\wt{A}_{\mc{J}}(u,v) = \frac{1}{2}\average{\nabla u,\nabla
	v}_{\mc{T}}
	- \frac{1}{2} \average{\mc{L}u, \nabla v}_{\mc{T}}
	- \frac{1}{2} \average{\mc{L}v, \nabla u}_{\mc{T}}
	+ \sum_{F\in \mc{S}} \alpha(J_F)
	\average{\jump{u},\jump{v}}_{F}
	\label{eqn:bilinearDGextend}
\end{equation}
It is clear that $\wt{A}_{\mc{J}}$ is consistent with both $A$ and
$A_{\mc{J}}$ in the sense that 
$\wt{A}_{\mc{J}}(u,v)=A(u,v), \forall u,v\in H^{1}_{\pi}(\Omega)$,
and $\wt{A}_{\mc{J}}(u_{\mc{J}}, v_{\mc{J}}) = A_{\mc{J}}(u_{\mc{J}},
v_{\mc{J}}),\forall u_{\mc{J}}, v_{\mc{J}}\in \mc{V}_{\mc{J}}(\mc{T})$.

In the pseudopotential framework, $V_{\eff}$ is bounded from above
and from below. Since any constant shift in the potential only causes
the same constant shift in each Kohn-Sham eigenvalue $\varepsilon_{i}$
without changing the Kohn-Sham orbitals, without loss of generality we
may add a positive constant to $V_{\eff}$ if needed, so that
$V_{\eff}$ is coercive and continuous  on $H^{1}_{\pi}(\Omega)\oplus
\mc{V}_{\mc{J}}(\mc{T})$, \ie
\begin{equation}
	\average{u,u}_{\mc{T}}
	\lesssim \average{u,V_{\eff} u}_{\Omega} 
	\lesssim \average{u,u}_{\mc{T}}, \quad \forall u\in H^{1}_{\pi}(\Omega)\oplus
	\mc{V}_{\mc{J}}(\mc{T}).
	\label{eqn:Vcoerciveextended}
\end{equation}
This technique has also been
used in previous work such as~\cite{Giani2012}.  
Combining Eq.~\eqref{eqn:Vcoerciveextended} with Lemma~\ref{lem:Atildebound}, we have
\begin{equation}
	\average{u,u}_{\mc{T}}
 \lesssim \wt{A}_{\mc{J}}(u, u) + 
	\average{V_{\eff}u,u}_{\mc{T}}
	\lesssim \average{u,u}_{\mc{T}}. 
	\label{eqn:wtAVcoercive}
\end{equation}
Eq.~\eqref{eqn:wtAVcoercive} implies that all true eigenvalues
$\varepsilon_{i}$ and computed eigenvalues $\varepsilon_{i,\mc{J}}$ are
positive.

Using the extended bilinear form $\wt{A}_{\mc{J}}$, we can define the
extended residual as follows.

\begin{definition}[Extended residual]
For any $v\in H^{1}_{\pi}(\Omega)\oplus \mc{V}_{\mc{J}}(\mc{T})$, 
and the eigenpair $(\varepsilon_{i},u_{i})$, the extended residual
corresponding to the eigenvalue problem~\eqref{eqn:weakeig} is
\begin{equation}
	\mc{R}(u_{i},v) = \wt{A}_{\mc{J}}(u_i,v) + (u_i,V_{\eff} v)
	- \varepsilon_{i}(u_i,v).
	\label{eqn:weakresidual}
\end{equation}
Similarly, for any $v\in H^{1}_{\pi}(\Omega)\oplus
\mc{V}_{\mc{J}}(\mc{T})$, and the computed eigenpair
$(\varepsilon_{i,\mc{J}},u_{i},\mc{J})$, the extended residual
corresponding to the eigenvalue problem~\eqref{eqn:weakDGeig} is
\begin{equation}
	\mc{R}_{\mc{J}}(u_{i,\mc{J}},v) = \wt{A}_{\mc{J}}(u_{i,\mc{J}},v) +
	(u_{i,\mc{J}},V_{\eff} v) - \varepsilon_{i,\mc{J}}(u_{i,\mc{J}},v).
	\label{eqn:weakDGresidual}
\end{equation}
\end{definition}

Now we give the result used to measure the error of the
eigenvalues using the a posteriori error estimator.

\begin{theorem}[Reliability of eigenvalues]
	\label{thm:eigval}
	If Assumption~\ref{asp:regularity} holds, 
	let $(\varepsilon_{i,\mc{J}},u_{i,\mc{J}})$ be a computed eigenpair
	corresponding to~\eqref{eqn:weakDGeig}, with $\varepsilon_{i,\mc{J}}$
	converging to the true eigenvalue $\varepsilon_{i}$ of multiplicity
	greater than or equal to $1$. 
	Then we have
	\begin{equation}
		\abs{\varepsilon_{i}-\varepsilon_{i,\mc{J}}} \lesssim
		\eta_{i}^2 + \inf_{\substack{u_{i}\in
                M(\varepsilon_{i}) \\ \norm{u_{i}}_{\Omega}=1}}G_{i},
		\label{eqn:eigvalreliable}
	\end{equation}
where 
	\begin{equation}
    G_{i} = (1+\sqrt{\gamma_{1,\mc{J}}})^2\xi_{i}^2 + 2
    \eta_{i}(1+\sqrt{\gamma_{1,\mc{J}}})\xi_{i}
		+ 2 \abs{\mc{R}(u_{i},u_{i}-u_{i,\mc{J}})} 
		+ 2 \abs{\mc{R}_{\mc{J}}(u_{i,\mc{J}},u_{i}-u_{i,\mc{J}})}.
		\label{eqn:Gdef}
	\end{equation}
\end{theorem}
\begin{proof}
  Using Eq.~\eqref{eqn:wtAidentity1},
  $\norm{u_{i,\mc{J}}}_{\mc{T}}=1$, and $\varepsilon_{i}>0$, we have
  \begin{equation}
    \begin{split}
    \varepsilon_{i,\mc{J}} - \varepsilon_{i}
    =& \wt{A}_{\mc{J}}(u_{i}-u_{i,\mc{J}}, u_{i}-u_{i,\mc{J}}) 
		+ \average{V_{\eff}(u_{i}-u_{i,\mc{J}}),
		(u_{i}-u_{i,\mc{J}})}_{\mc{T}}\\
		&- \varepsilon_{i} \norm{u_{i}-u_{i,\mc{J}}}^2_{\Omega} 
		- 2 \Re \mc{R}(u_{i}, u_{i} - u_{i,\mc{J}})\\
    \le& \wt{A}_{\mc{J}}(u_{i}-u_{i,\mc{J}}, u_{i}-u_{i,\mc{J}}) 
		+ \average{V_{\eff}(u_{i}-u_{i,\mc{J}}),
		(u_{i}-u_{i,\mc{J}})}_{\mc{T}} \\
    &+ 2 \abs{\mc{R}(u_{i}, u_{i} - u_{i,\mc{J}})}.
    \end{split}
    \label{eqn:eigvalaux1}
  \end{equation}
  Similarly, Eq.~\eqref{eqn:wtAidentity2} in Lemma~\ref{lem:wtAidentity},
  $\norm{u_{i}}_{\Omega}=1$, and $\varepsilon_{i,\mc{J}}>0$ give 
  \begin{equation}
    \begin{split}
    \varepsilon_{i} - \varepsilon_{i,\mc{J}} 
    =& \wt{A}_{\mc{J}}(u_{i}-u_{i,\mc{J}}, u_{i}-u_{i,\mc{J}}) 
		+ \average{V_{\eff}(u_{i}-u_{i,\mc{J}}),
		(u_{i}-u_{i,\mc{J}})}_{\mc{T}}\\
		&- \varepsilon_{i,\mc{J}} \norm{u_{i}-u_{i,\mc{J}}}^2_{\Omega} 
    - 2 \Re \mc{R}_{\mc{J}}(u_{i,\mc{J}}, u_{i,\mc{J}} - u_{i})\\
    \le& \wt{A}_{\mc{J}}(u_{i}-u_{i,\mc{J}}, u_{i}-u_{i,\mc{J}}) 
		+ \average{V_{\eff}(u_{i}-u_{i,\mc{J}}),
		(u_{i}-u_{i,\mc{J}})}_{\mc{T}}\\
    & + 2 \abs{\mc{R}_{\mc{J}}(u_{i,\mc{J}}, u_{i,\mc{J}} - u_{i})}.
    \end{split}
    \label{eqn:eigvalaux2}
  \end{equation}
  Combining Eq.~\eqref{eqn:eigvalaux1} and \eqref{eqn:eigvalaux2}, we
  have
	\begin{equation}
		\begin{split}
		\abs{\varepsilon_{i}-\varepsilon_{i,\mc{J}}} 
		\le & \wt{A}_{\mc{J}}(u_{i}-u_{i,\mc{J}}, u_{i}-u_{i,\mc{J}}) 
		+ \average{V_{\eff}(u_{i}-u_{i,\mc{J}}),
		u_{i}-u_{i,\mc{J}}}_{\mc{T}}\\
		&+  2 \abs{\mc{R}(u_{i}, u_{i} - u_{i,\mc{J}})}  
		+  2 \abs{\mc{R}_{\mc{J}}(u_{i,\mc{J}}, u_{i,\mc{J}} -
                u_{i})}.
		\end{split}
		\label{}
	\end{equation}
	Using Eq.~\eqref{eqn:wtAVcoercive}, we obtain
	\begin{equation}
		\abs{\varepsilon_{i}-\varepsilon_{i,\mc{J}}} \lesssim
		\normET{u_{i,\mc{J}}-u_{i}}^2 
		+ 2 \abs{\mc{R}(u_{i}, u_{i} - u_{i,\mc{J}})}
		+ 2 \abs{\mc{R}_{\mc{J}}(u_{i,\mc{J}}, u_{i} - u_{i,\mc{J}})} 
		\label{}
	\end{equation}
	which, using Eq.~\eqref{eqn:uerrbound} in
        Theorem~\ref{thm:eigfunc}, becomes
	\begin{equation}
    \begin{split}
		\abs{\varepsilon_{i}-\varepsilon_{i,\mc{J}}} \lesssim&
    \left(\eta_{i} + (1+\sqrt{\gamma_{1,\mc{J}}}) \xi_{i}\right)^2  
		+ 2 \abs{\mc{R}(u_{i}, u_{i} - u_{i,\mc{J}})}
		+ 2 \abs{\mc{R}_{\mc{J}}(u_{i,\mc{J}}, u_{i} - u_{i,\mc{J}})}\\
		\equiv& \eta_{i}^2 + G_{i}.
    \end{split}
		\label{eqn:eigerrbound}
	\end{equation}
	The theorem then follows directly from Eq.~\eqref{eqn:eigerrbound},
        once we minimize $G_{i}$ among all normalized eigenfunctions $u_{i}\in
	M(\varepsilon_{i})$.  
\end{proof}

\begin{remark}
  It remains to be shown that $\eta_{i}$ dominates $\xi_{i}$ in
  Eq.~\eqref{eqn:eigfuncreliable}.  To this end, we require a priori
  estimates of the convergence rate of $u_{i,\mc{J}}$, which we do not
  have at this stage.  However, in the case of standard $hp$-refinement for
  the Laplacian eigenvalue problem, it can be shown that
  $\xi_{i}$ is of a higher order than $\dist\left(
  u_{i,\mc{J}}, M(\varepsilon_{i})
  \right)_{E,\mc{T}}$~\cite{GianiHall2012}. This implies that $\eta_{i}$
  is the leading order term in the estimate of the
  eigenfunctions~\eqref{eqn:eigfuncreliable}.

  Similarly, in the case of standard $hp$-refinement for solving the Laplacian
  eigenvalue problem, it can be shown that
  $\mc{R}(u_{i},u_{i}-u_{i,\mc{J}})$ is of a higher order than
  $\abs{\varepsilon_{i}-\varepsilon_{i,\mc{J}}}$.  One can also show
  that $\mc{R}_{\mc{J}}(u_{i,\mc{J}},u_{i}-u_{i,\mc{J}})$ is of a higher
  order than $\abs{\varepsilon_{i}-\varepsilon_{i,\mc{J}}}$.  Combining
  this with the statement above, we can then conclude that $\eta_{i}^2$
  is the leading order term in the estimate of the
  eigenvalues~\eqref{eqn:eigvalreliable} in this case.

  For general basis sets, the proof that these terms are of higher order relies on
  an \textit{a priori} error analysis of the basis set. According to our
  current understanding, this analysis might be even more difficult than
  the \textit{a posteriori} error analysis given here for general
  non-polynomial basis functions.  
\end{remark}

\subsection{Non-uniform adaptive local basis refinement strategy}

Theorem~\ref{thm:eigfunc} demonstrates that the a posteriori error
estimator is reliable in quantifying the error of each eigenfunction on the
global domain for a suitable set of local basis functions.  In practical
calculations, we use the local error estimator $\eta_{i,K}$ to guide the
adaptive refinement of the basis functions.  An accurate and efficient
solution should not be under-resolved nor over-resolved in any element;
thus a non-uniform local basis refinement strategy should yield a solution
for which the value of the local estimator is close to uniform across all
elements.

Since we are concerned with the error of all eigenpairs simultaneously, we
define the local estimator on an element $K$ by

\begin{equation}
  \eta_K^2 = \sum_{i=1}^N \eta_{i,K}^2.
\end{equation}

The refinement strategy is straightforward. Minimum and maximum local error
thresholds $\epsilon_{min}$ and $\epsilon_{max}$ are chosen based on the
desired global error, along with a basis refinement step size
$b_{\text{step}}$ and a number of refinement steps $n$.  An initial
distribution of basis functions $\mc{J}_1$, which may be uniform or, given
prior knowledge of the problem, non-uniform, is also specified. A trial
solution is computed using the initial distribution $\mc{J}_1$. The local
estimators $\eta_K^2$ are evaluated. For each element $K$, if $\eta_K^2 <
\epsilon_{min}$, $b_{\text{step}}$ local basis functions are removed from
$K$. If $\eta_K^2 > \epsilon_{max}$, $b_{\text{step}}$ local basis
functions are added to $K$. Thus, local basis functions are removed from
over-resolved elements and added to under-resolved elements. 
At each iteration $j$, Eq.~\eqref{eqn:localproblem} is solved
for the lowest $J_{K,j}$ eigenfunctions to obtain the adaptive local
basis functions.
A new solution is computed, and the process is repeated. On the
$(n-1)$th step, a final distribution of basis function $\mc{J}_n$ is
determined, and the electron density $\rho$ is computed using this
distribution. This procedure is outlined in Algorithm \ref{alg:bref}.

This strategy may not yield an efficient distribution of basis functions if
the initial distribution $\mc{J}_1$ contains too few basis functions to
reasonably resolve the system, nor is it intended to do so. In this case,
the local estimator may not be effective at the first step, which could
lead to erratic estimates of local error. It is not necessarily the case
that the total number of basis functions will increase as the basis is
refined - as we will demonstrate in Section \ref{sec:numerical3D}, the
total number of basis functions may very well decrease through the
refinement process - so all available resources may be used to make the
initial calculation. Rather, the refinement process redistributes basis
functions to achieve higher accuracy at a minimum cost by adding basis
functions to under-resolved elements and removing basis functions from
over-resolved elements.

\begin{algorithm}
  \caption{Non-uniform basis refinement procedure.}
  \label{alg:bref}
    \begin{algorithmic}
    \STATE Input: $\epsilon_{\min}$, $\epsilon_{\max}$, initial
		distribution of basis functions $\mc{J}_1
		=\{J_{K,1}\}_{K \in \mc{T}}$, $n$,
    $b_{\text{step}}$
		\STATE Output: Refined distribution of basis functions
		$\mc{J}_{n}$, and physical quantities such as electron density $\rho$
    \FOR{$1 \le j \le n-1$}
      \STATE Compute a set of eigenpairs
      $\{(\varepsilon_{i,\mc{J}_j},u_{i,\mc{J}_j})\}_{i=1}^N$ based on
			the ALB approach
      \FOR{$K \in \mc{T}$}
        \STATE Compute $\eta_K^2$
        \IF{$\eta_K^2 < \epsilon_{\min}$}
	  \STATE Update $J_{K,j+1}\gets J_{K,j} - b_{\text{step}}$
        \ELSIF{$\eta_K^2 > \epsilon_{\max}$}
          \STATE Update $J_{K,j+1}\gets J_{K,j} + b_{\text{step}}$
        \ELSE
          \STATE Update $J_{K,j+1}\gets J_{K,j}$
        \ENDIF
      \ENDFOR
    \ENDFOR
    \STATE Compute and output physical quantities (such as $\rho$) using $\mc{J}_n$
  \end{algorithmic}
\end{algorithm}

Let us contrast the non-uniform refinement procedure with a uniform
refinement procedure, in which some number $b_{\text{step}}$ of basis
functions are added to every element for each of $n$ steps. If we carry out
uniform and non-uniform refinement with the same choices of
$b_{\text{step}}$ and $n$, and the same initial distribution of basis
functions $\mc{J}_1$, we can compare the accuracy and total number of basis
functions used at each step of the two refinement processes. The maximum
possible number of basis functions assigned to any element on a given step of the
non-uniform refinement scheme will always be equal to the number of basis
functions in all elements on the same step of the uniform refinement
scheme, but the minimum number of basis functions assigned to any element may be as
little as zero. Therefore, the error on a given step of the uniform refinement
scheme will always be less than or equal to that on the corresponding step
of the non-uniform refinement scheme, up to small numerical errors incurred
during the course of the solution process. We may then ask how much less is
the error obtained from uniform refinement than from non-uniform refinement,
and how many more basis functions are used to make it so. This is the
criterion by which the non-uniform refinement scheme will be judged against a
uniform refinement scheme for the calculations presented in Section
\ref{sec:numerical3D}.

\section{Numerical results for Kohn-Sham density functional theory
calculations}\label{sec:numerical3D}

In this section we present numerical results which demonstrate that the global
estimator is effective in predicting the error of the total energy, and that
the local estimator may be used to guide efficient non-uniform local basis
refinement. Although the theoretical part of the paper only
aims at developing an a posteriori error estimator for linear eigenvalue problems, we
apply the estimator to solve the nonlinear Kohn-Sham equations.  All
numerical results are performed with converged self-consistent field
iterations. They properly take into account the nonlinearity of the Kohn-Sham
equations.
We test the refinement strategy for Kohn-Sham density
functional theory calculations on two example systems. The first is an aluminum
surface with a large vacuum region, and the second is a single
layer of graphene oxide in water. These systems are highly inhomogeneous
and therefore serve as good benchmark examples of non-uniform adaptive
local basis refinement.   

The evaluation of the a posteriori error estimator and the refinement
strategy are implemented in the Discontinuous Galerkin Density
Functional Theory (DGDFT) software.  The energies and atomic forces are
directly compared with the results obtained from ABINIT~\cite{abinit1}
using the same atomic configuration.  We note that although the analysis in
Section~\ref{sec:apost} is given for the linear eigenvalue problem, in all
the numerical examples we solve the nonlinear eigenvalue problem as is
required for KSDFT calculations. The results presented for both DGDFT
and for ABINIT calculations have achieved convergence in the self-consistent
field (SCF) iteration.
In both the DGDFT and the ABINIT
calculations, we use the local density approximation
(LDA)~\cite{CeperleyAlder1980,PerdewZunger1981} for the
exchange-correlation functional, and the Hartwigsen-Goedecker-Hutter (HGH)
pseudopotential~\cite{HartwigsenGoedeckerHutter1998} with the local and
non-local pseudopotential fully implemented in the real
space~\cite{PaskSterne2005}. All quantities are reported in atomic units
(au).  All calculations are carried out on the Edison system maintained
at the National Energy Research Scientific Computing Center (NERSC). Each
compute node on Edison has $24$ Intel ``Ivy Bridge'' cores ($2.4$ GHz)
with $64$ gigabytes (GB) of memory. 

As a first attempt we use the choice of parameters in
Eq.~\eqref{eqn:constantHP} to define the a posteriori error estimator.
We define $h_{K}$ to be the diameter of the element $K$, 
and we reinterpret $p_{K}$ to be the number of ALBs rather than the
polynomial degree in each element $K$.  Our results indicate
that for the examples studied, the global estimator $\eta^2 = \sum_{K
\in \mc{T}} \eta_K^2$ is numerically effective in predicting the error
of the total energy within a relatively small constant factor, and that
the local estimator $\eta_K^2$ is numerically effective in predicting
the error in each element $K$ of the computed electron density $\rho$.
As a result, the non-uniform basis refinement strategy yields efficient
distributions of local basis functions across the elements for these
examples. Replacing a uniform refinement scheme with a non-uniform
refinement scheme incurs very little loss in accuracy, but allows the
use of a much smaller basis and hence gives significant computational
savings.

\subsection{Quasi-2D aluminum surface}

The first example is a quasi-2D aluminum surface with a large vacuum
region. The periodic domain contains 16 aluminum atoms arranged in a
row, and a large vacuum region. 
The size of the supercell is approximately $7.65$ au, $30.61$ au and $45.92$ au along
the $x,y,$ and $z$ directions, respectively. The computational domain is
partitioned into a $1 \times 6 \times 9$ grid of elements. The size of the
extended element is always $3$ times the size of the element along each
direction, unless there is only one element and in this case the size of
the extended element is the same as the size of the element.
One slice of the electron density $\rho$ in the $y-z$ plane is shown in
Fig.~\ref{fig:Al6x9yzslice}, and another
in the $x-y$ plane in Fig.~\ref{fig:Al6x9xyslice}. The
black dashed lines indicate the partition of the elements.
In DGDFT, a uniform grid is used to represent quantities such as the electron
density and potential in the global domain, with $70$, $276$, and
$414$ grid points along
the $x$, $y$, and $z$ directions, respectively. A Legendre-Gauss-Lobatto (LGL) grid
is used inside each element to construct the DG Hamiltonian matrix, and
the number of LGL grid points inside each element is $140$, $92$, and $92$ along
the $x$, $y$, and $z$ directions, respectively.

\begin{figure}[ht]
  \centering
  \subfloat[Fixed x = 3.8267 \label{fig:Al6x9yzslice}]{%
    \raisebox{-0.5\height}{\includegraphics[width=0.47\textwidth]{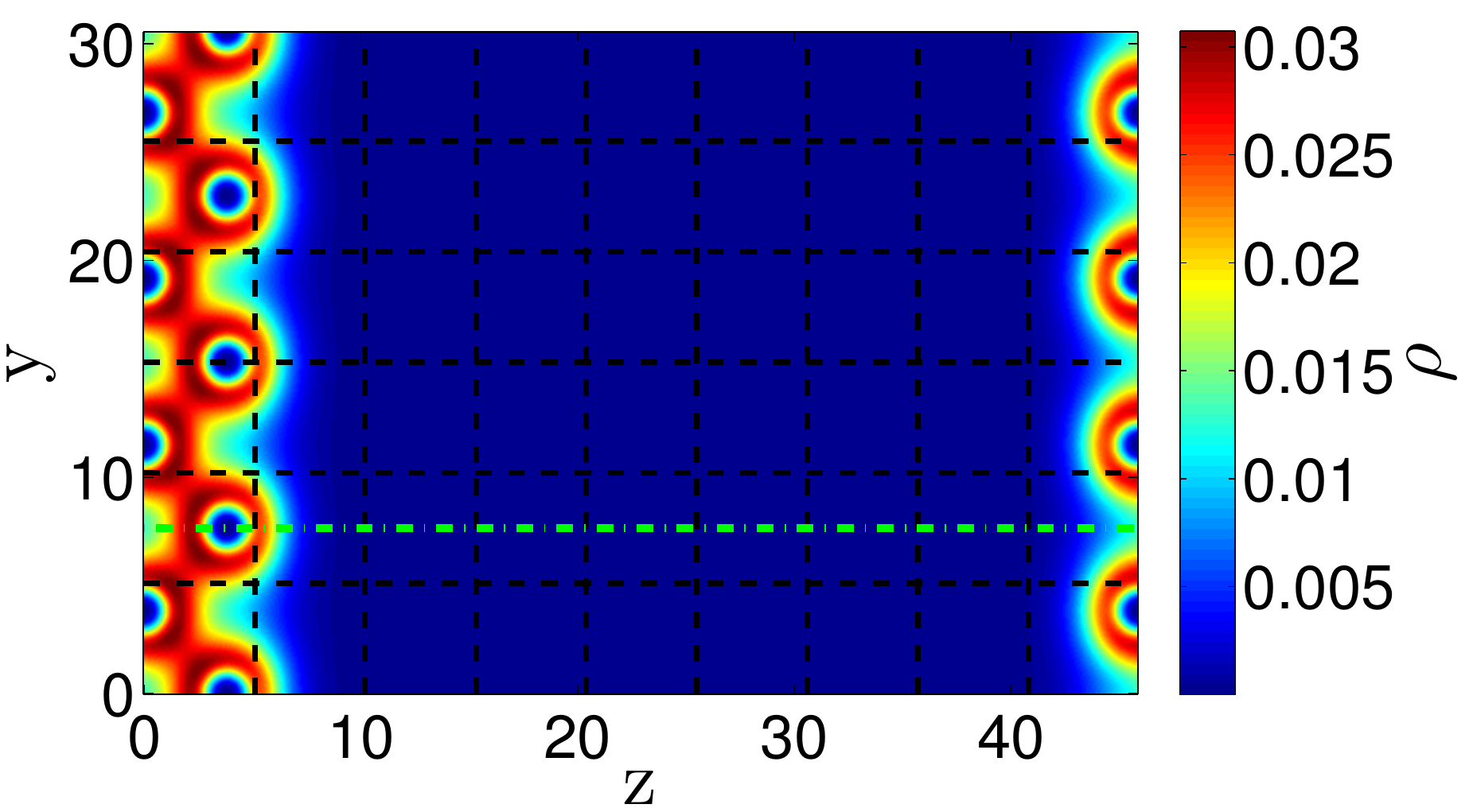}}
  }
  \hfill
  \subfloat[Fixed z = 0 \label{fig:Al6x9xyslice}]{%
    \raisebox{-0.5\height}{\includegraphics[width=0.47\textwidth]{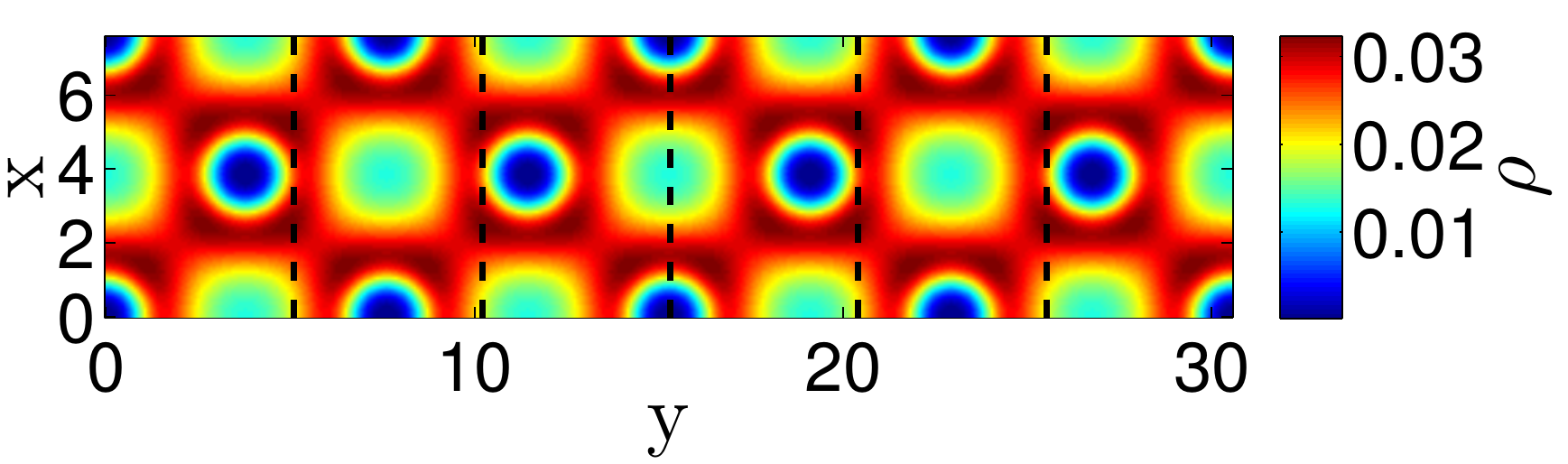}}
  }
  \caption{Electron density $\rho$ in the (a) $y-z$ plane and in the (b)
	$x-y$ plane. The partition of the global domain into $1 \times 6 \times 9$ 
	elements is indicated by black dashed lines.}
  \label{fig:Al6x9slices}
\end{figure}

Fig.~\ref{fig:Al6x9yzsliceerr} shows the pointwise error of the electron
density corresponding to the slice shown in Fig.~\ref{fig:Al6x9yzslice}, with 50 adaptive
local basis functions used in each element.
The error of the
density $\rho$ is highly non-uniform across different elements; the error
in the elements which contain atoms can be more than $5$ orders of
magnitude larger than that in elements in the vacuum region.  

Fig.~\ref{fig:Al6x91Dslice} shows the slice of electron density
along $y=7.6534$ as indicated by the green dash-dotted line in
Fig.~\ref{fig:Al6x9yzslice}.  Fig.~\ref{fig:Al6x91Dsliceerr} shows the
error of the electron density, along with the values of the local estimator
$\eta_K^2$ in each of the elements $K$ through which the slice passes. 
We observe that the local estimator provides an upper bound for the local
error of the electron density, and can be used to indicate the relative
contributions of the various elements to this error in the adaptive
refinement procedure.

\begin{figure}[ht]
  \centering
  \includegraphics[width=0.47\textwidth]{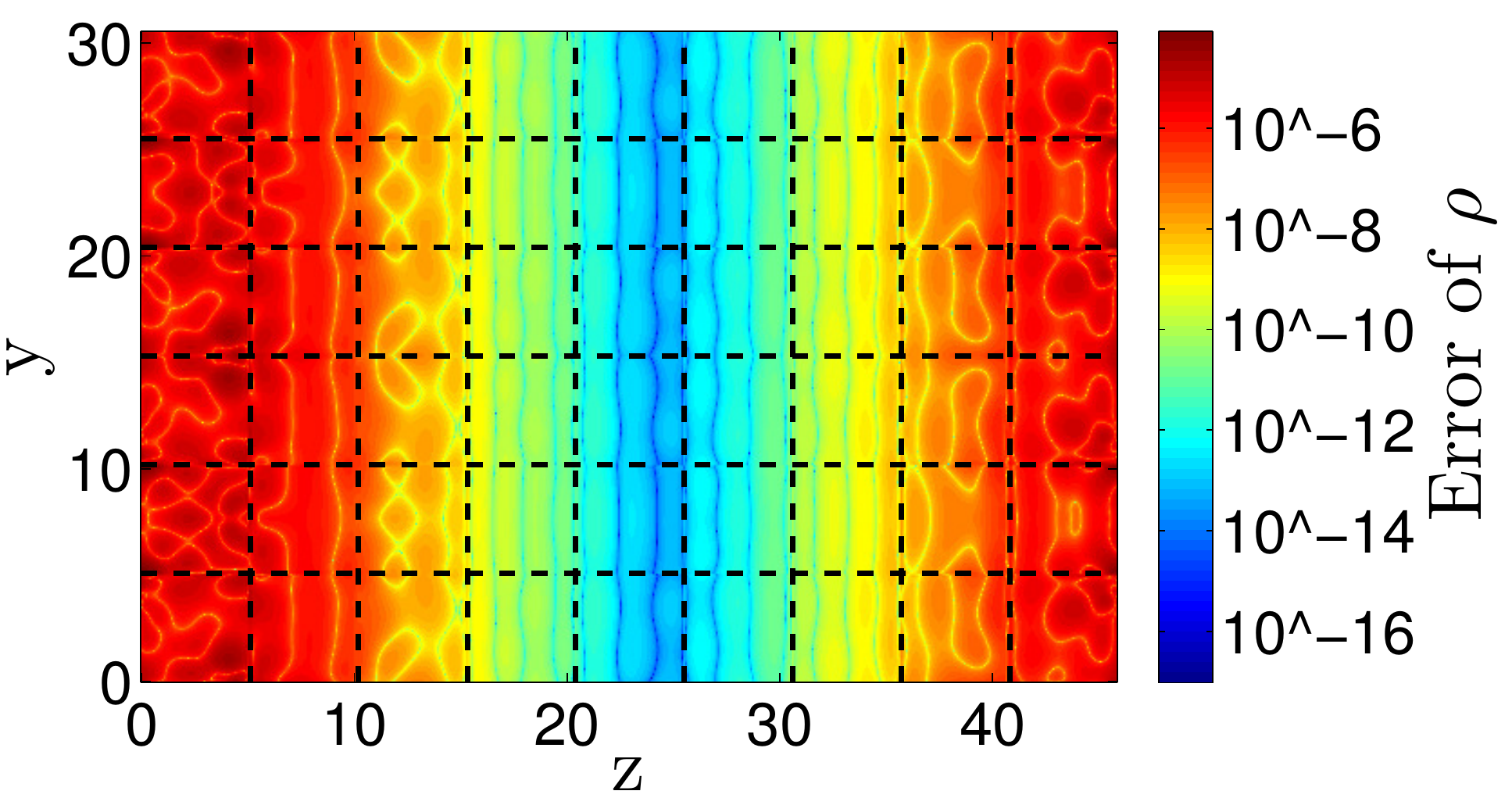}
  \caption{The pointwise error of the electron density $\rho$ (in log
    scale) for the slice shown in Figure \ref{fig:Al6x9yzslice} when 50
    adaptive local basis functions are used in each element.  The error in
    different elements differ by several orders of magnitude, which
    indicates that non-uniform distribution of basis functions may be used
    to generate more efficient electronic structure calculations for this
    system. The error is computed by comparing with the DGDFT result an
  accurate calculation obtained using the software package ABINIT.}
  \label{fig:Al6x9yzsliceerr}
\end{figure}

\begin{figure}[ht]
  \centering
  \subfloat[Electron density \label{fig:Al6x91Dslice}]{%
    \raisebox{-0.5\height}{\includegraphics[width=0.47\textwidth]{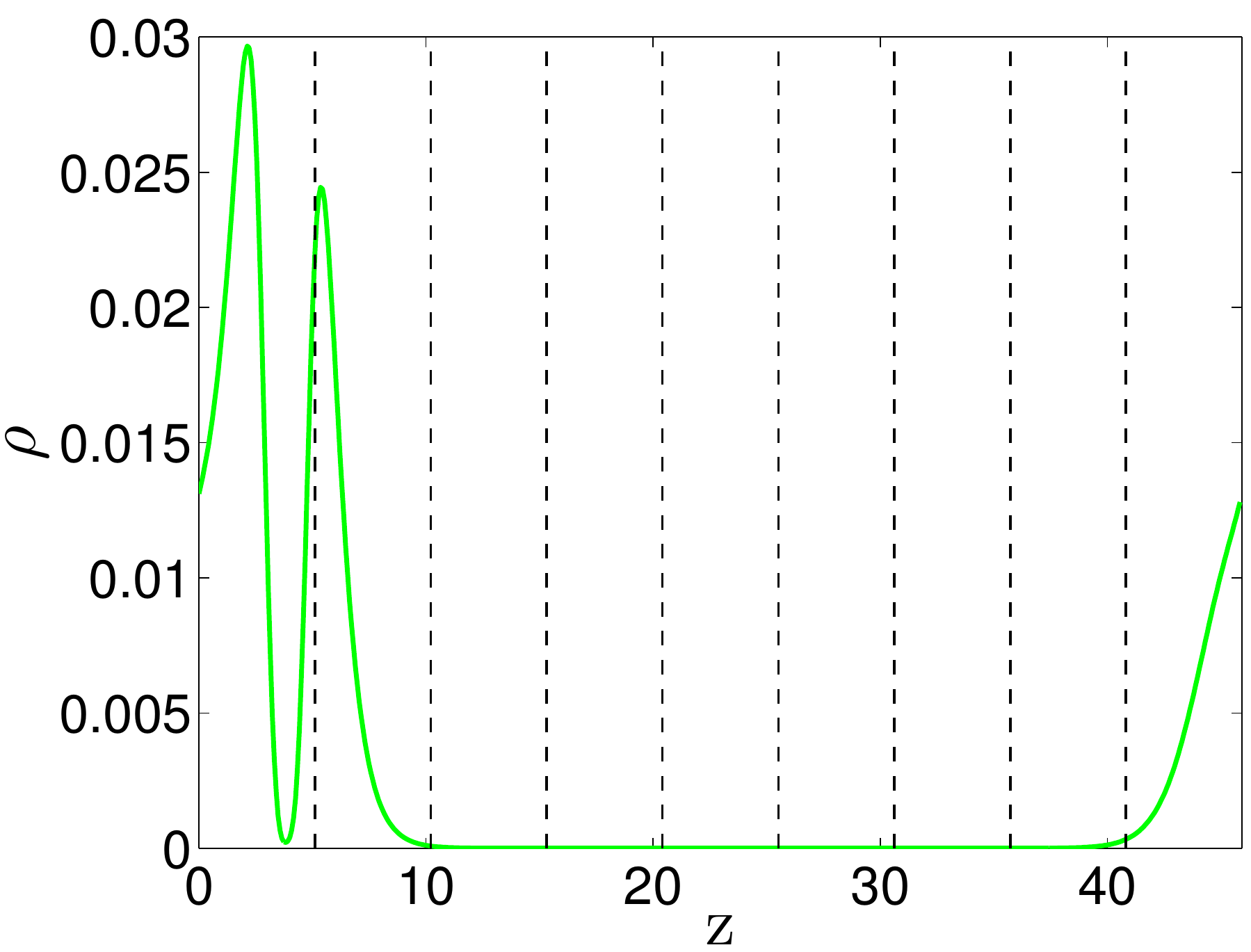}}
  }
  \hfill
  \subfloat[Pointwise error of slice and local estimator \label{fig:Al6x91Dsliceerr}]{%
    \raisebox{-0.5\height}{\includegraphics[width=0.47\textwidth]{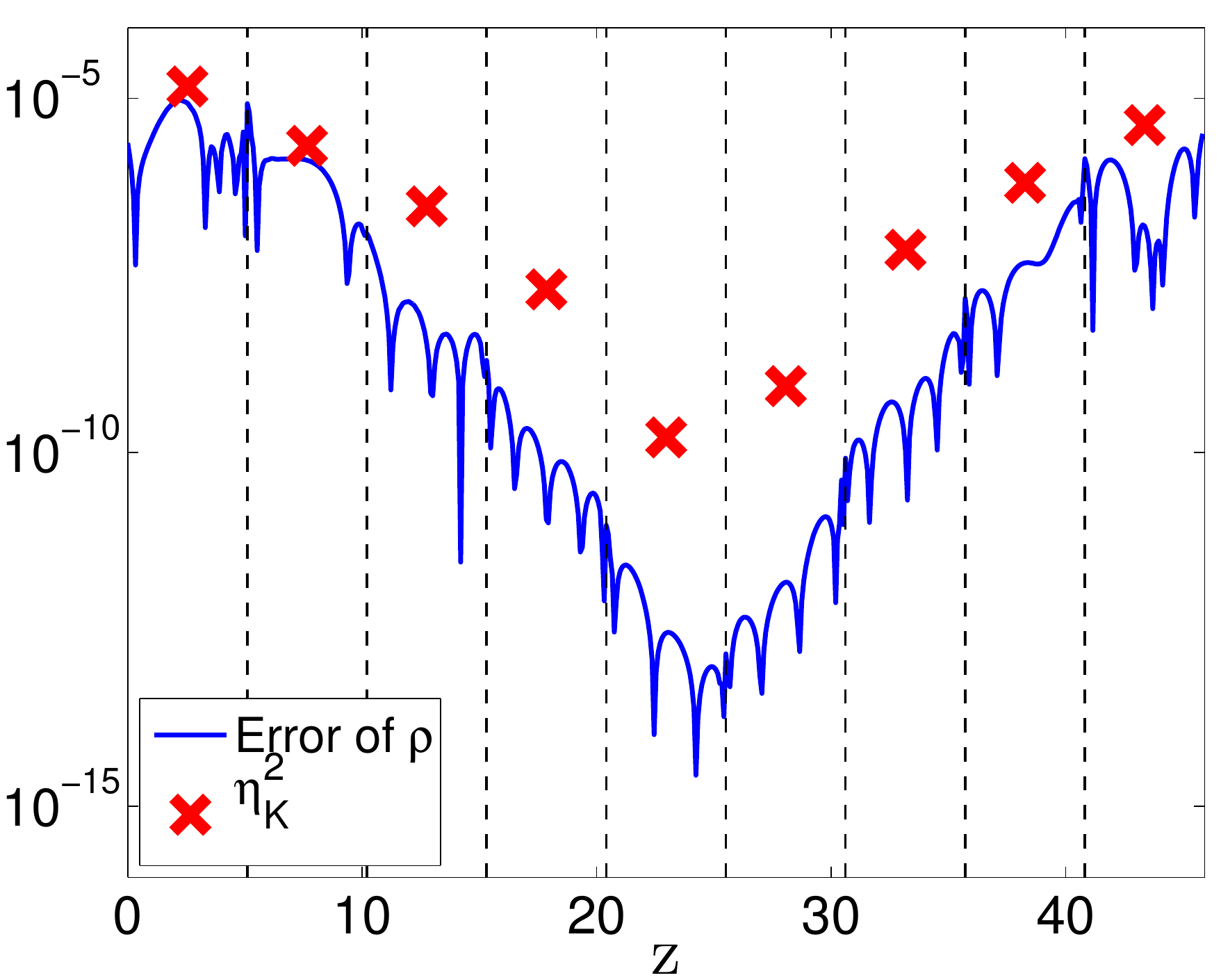}}
  }
	\caption{(a) One-dimensional slice of the electron density indicated
	by the green dash-dotted line in Fig.~\ref{fig:Al6x9yzslice} at
	$y=7.6534$.  (b) The
	error of the electron density along with the local estimator
	$\eta_K^2$ along the one dimensional slice, demonstrating the
      effectiveness of the local estimator at predicting the local error of
    $\rho$.}
  \label{fig:Al6x91Dslicefigs}
\end{figure}

Let us next compare the accuracy and efficiency of uniform and non-uniform
basis refinement schemes for this system.  In both schemes, we begin with
25 adaptive local basis functions in all elements on the first step. We
specify the basis refinement step size $b_{\text{step}} = 5$ and the number
of steps $n = 6$. In the uniform refinement scheme, we simply add
$b_{\text{step}} = 5$ basis functions to each element on each step. The
non-uniform refinement scheme follows Algorithm \ref{alg:bref} with
$\epsilon_{\max} = 5 \times 10^{-6}$ and $\epsilon_{\min} = 5 \times
10^{-7}$. 

We first examine the distribution of basis functions produced by
the non-uniform refinement scheme, shown for all steps in Figure
\ref{fig:Al6x9yzbasisnu}. The scheme eventually places more basis functions
in elements for which a uniform distribution of basis functions produces
larger pointwise errors. This becomes clear when comparing Figures
\ref{fig:Al6x9yzbasisnu} and \ref{fig:Al6x9yzsliceerr}, the latter of which
gives the pointwise error in the density for the last step of the uniform
refinement scheme.  Indeed, the non-uniform scheme allocates \textit{zero
basis functions} to the elements which contain the smallest pointwise
errors in the uniform scheme, and the maximum possible number of basis
functions to the elements which contain the largest pointwise errors in the
uniform scheme.

\begin{figure}[ht]
  \centering
  \subfloat[Step 1 \label{fig:Al6x9yzbasisnu1}]{%
    \includegraphics[width=0.3\textwidth]{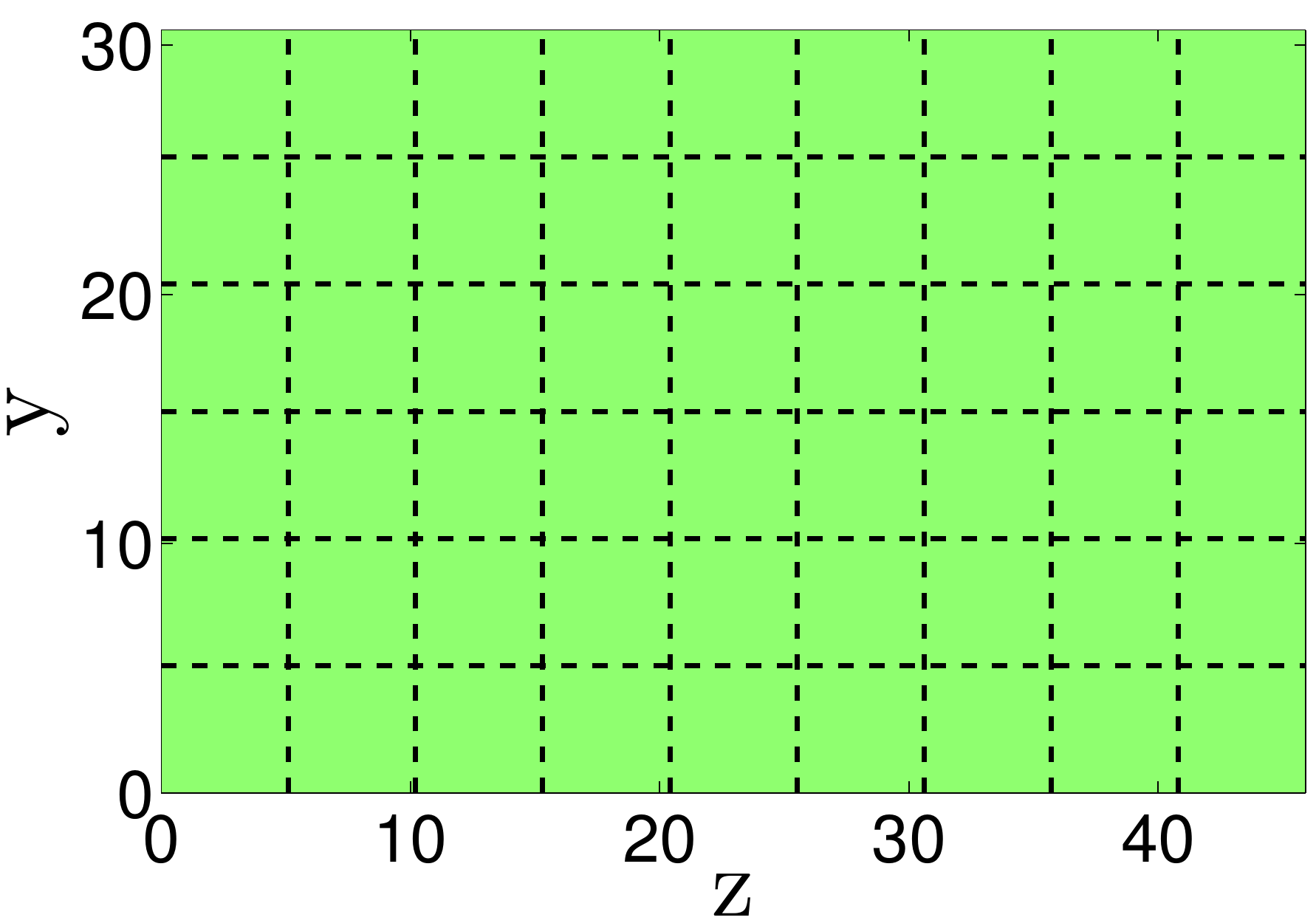}
  }
  \hfill
  \subfloat[Step 2 \label{fig:Al6x9yzbasisnu2}]{%
    \includegraphics[width=0.3\textwidth]{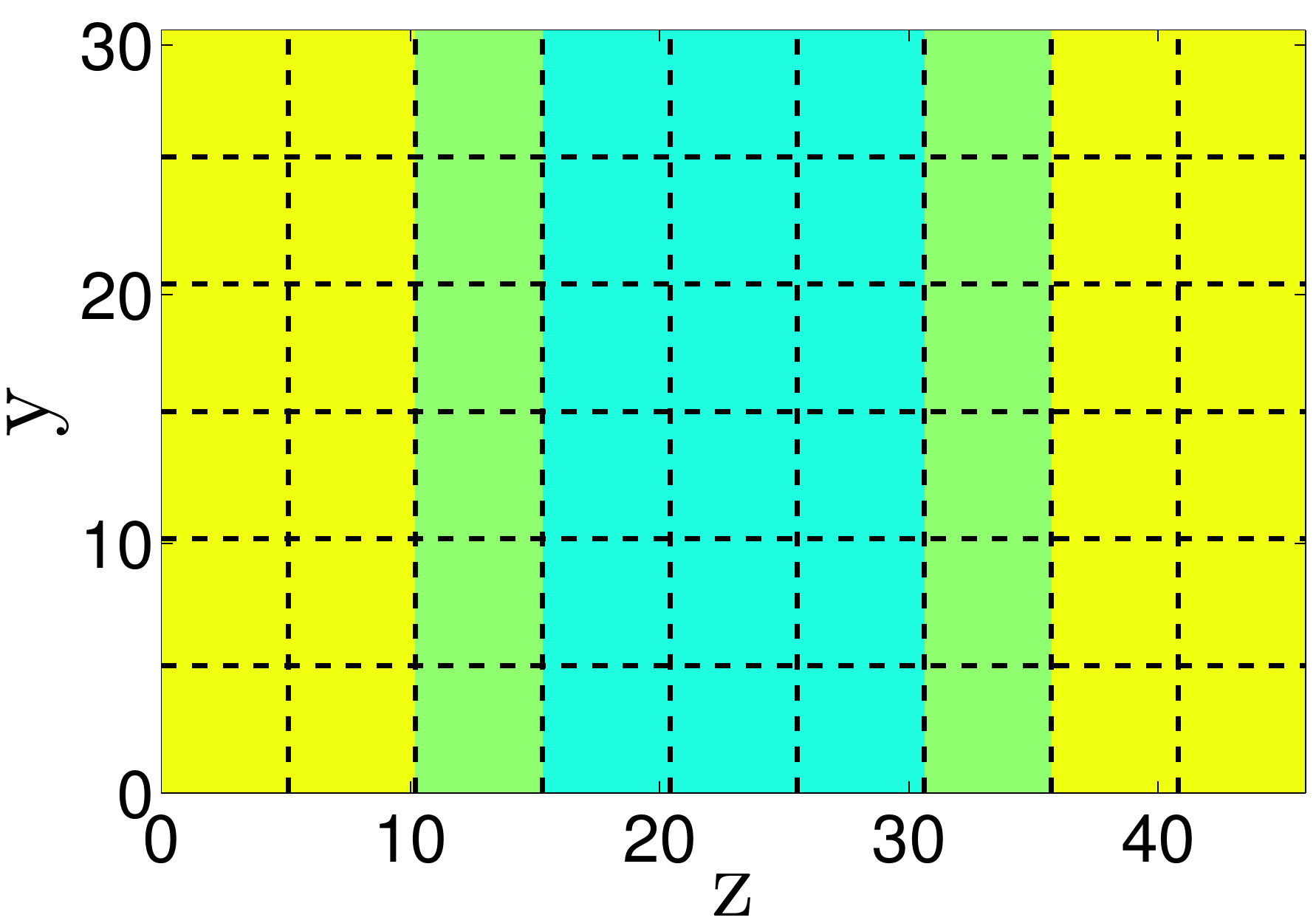}
  }
  \hfill
  \subfloat[Step 3 \label{fig:Al6x9yzbasisnu3}]{%
    \includegraphics[width=0.3\textwidth]{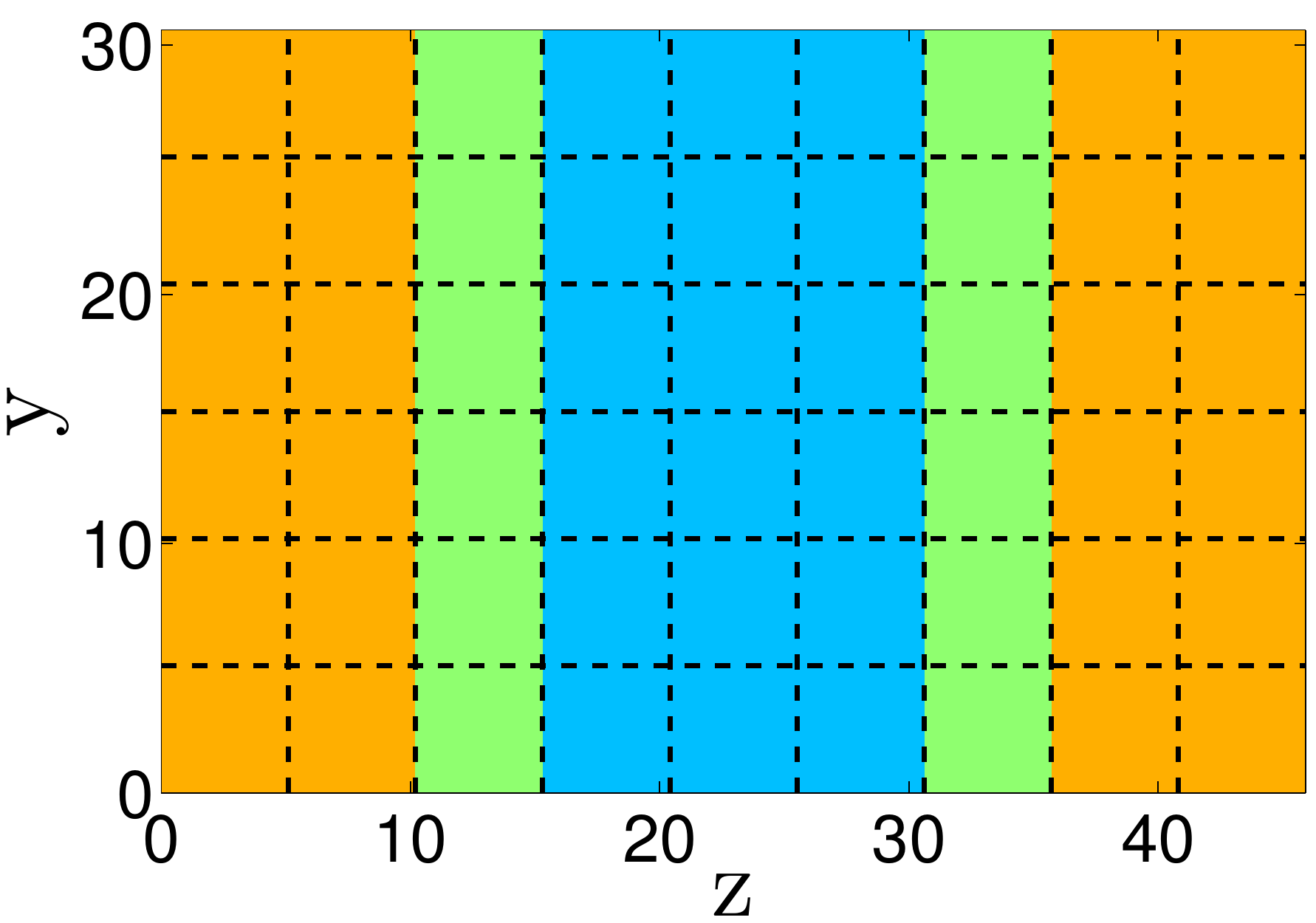}
  }

  \subfloat[Step 4 \label{fig:Al6x9yzbasisnu4}]{%
    \includegraphics[width=0.3\textwidth]{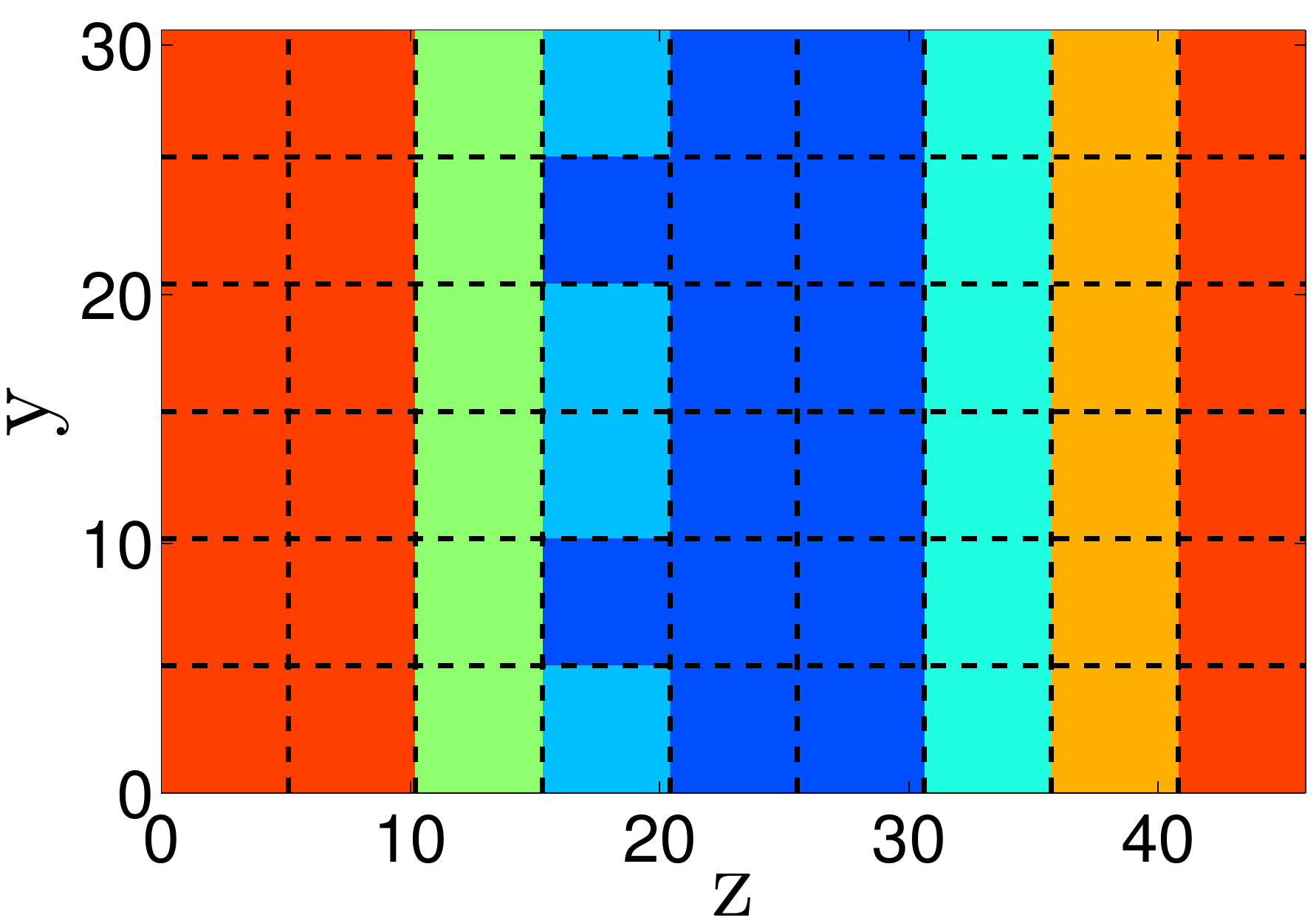}
  }
  \hfill
  \subfloat[Step 5 \label{fig:Al6x9yzbasisnu5}]{%
    \includegraphics[width=0.3\textwidth]{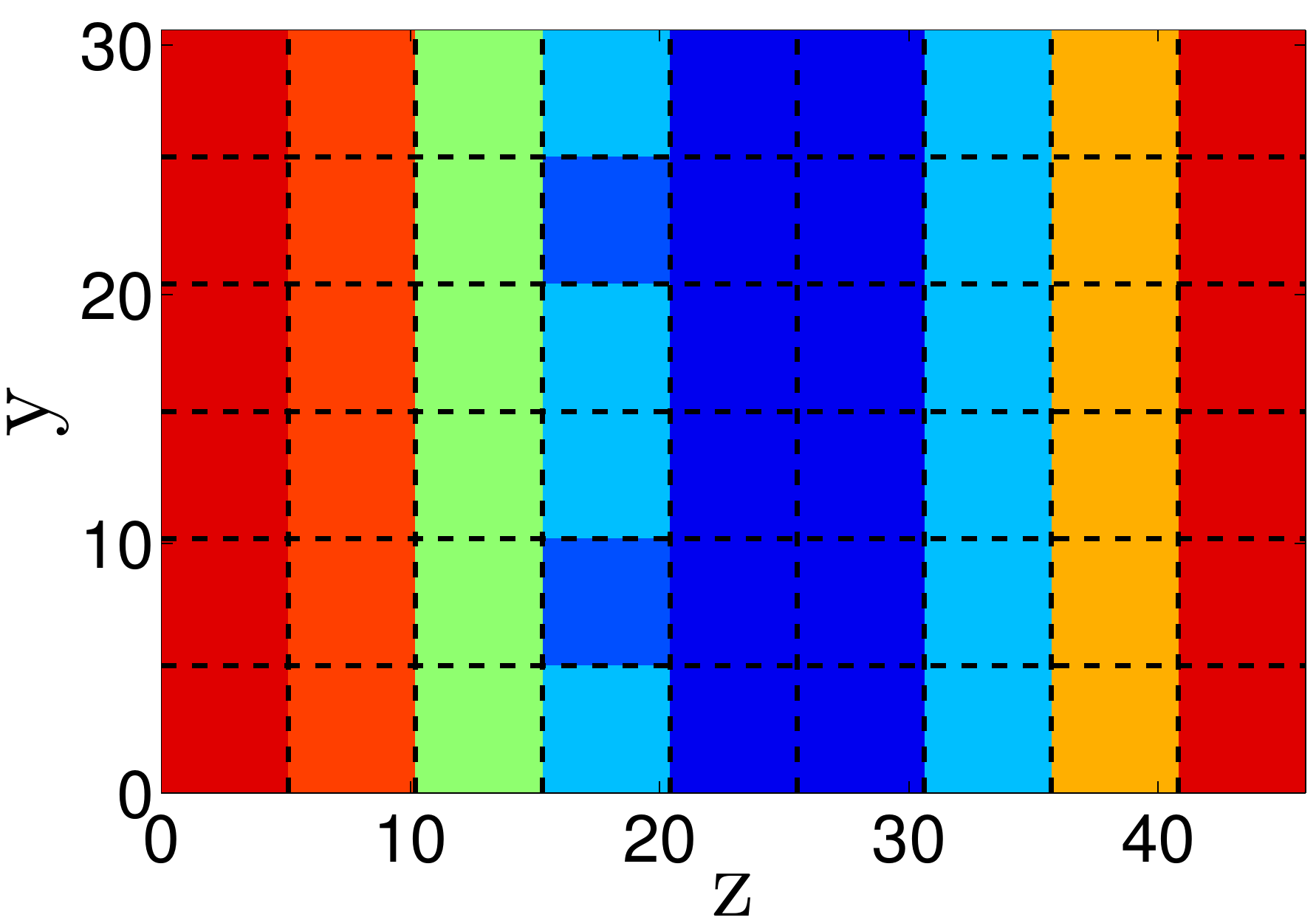}
  }
  \hfill
  \subfloat[Step 6 \label{fig:Al6x9yzbasisnu6}]{%
    \includegraphics[width=0.3\textwidth]{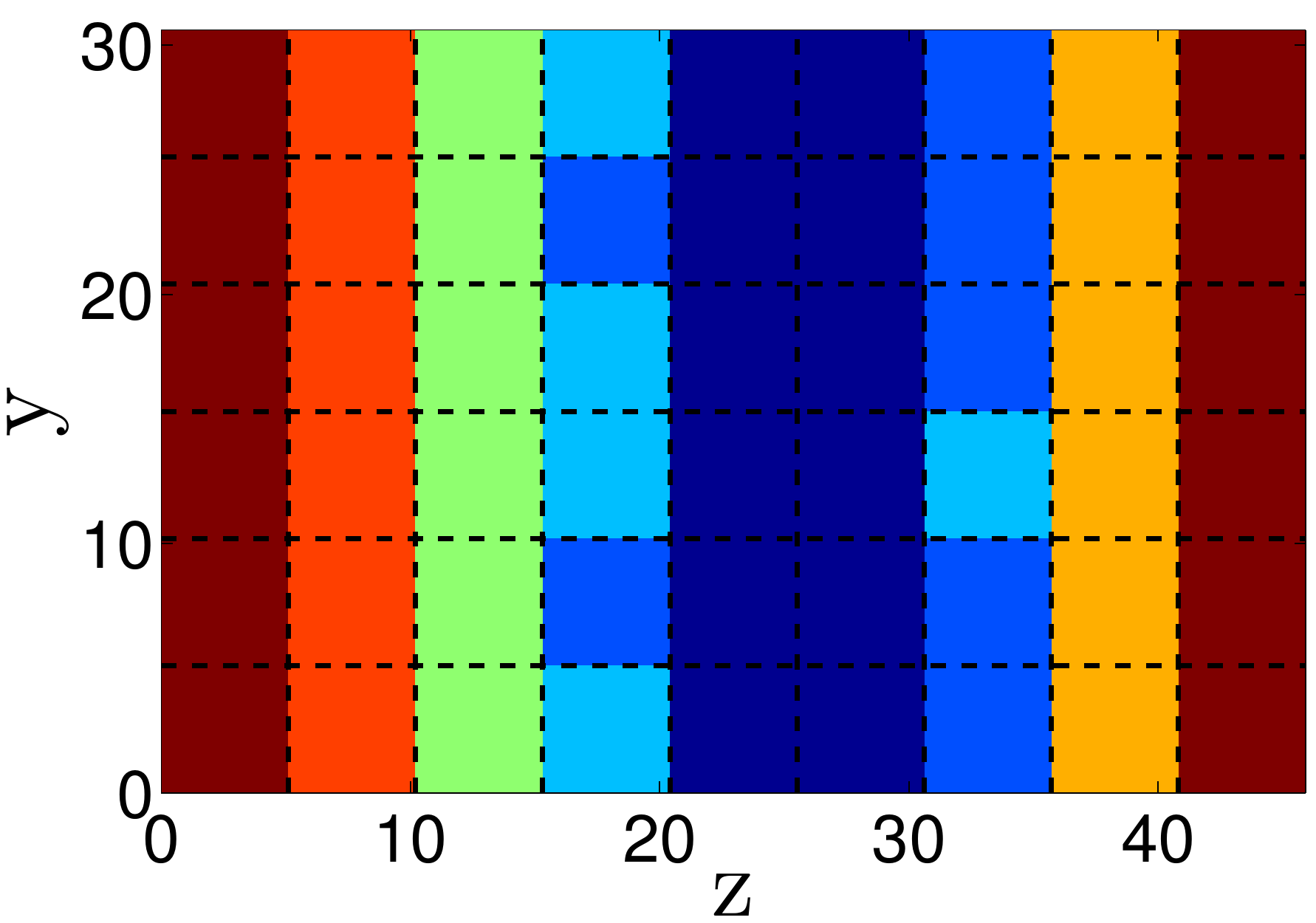}
  }

  \vspace{-4mm}
  \subfloat{%
    {\includegraphics[width=0.5\textwidth]{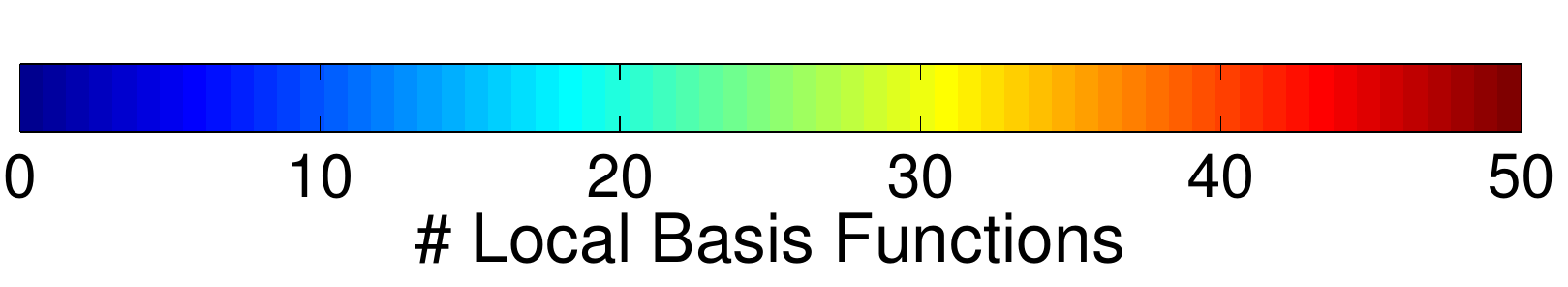}}
  }

  \caption{Distribution of adaptive local basis functions in six steps of
  non-uniform refinement for aluminum system. The final distribution mirrors
  the pointwise errors shown in Fig.~\ref{fig:Al6x9yzsliceerr}.} 
  \label{fig:Al6x9yzbasisnu}
\end{figure}

At the last step of refinement, the non-uniform refinement scheme uses half
as many basis functions per atom as the uniform refinement scheme, as shown
in Fig.~\ref{fig:Al6x9totalbasis}.  These savings are obtained without
sacrificing accuracy.  Fig.~\ref{fig:Al6x9energyerr} show that the error
per atom of the total energy is nearly equal using both methods, despite
the large difference in the number of basis functions used. We also measure
the accuracy of the atomic force, which is given by the derivative of the
energy with respect to the atomic positions. The atomic force is used in
various applications including geometry optimization and molecular
dynamics.  We measure both the maximum and the average Euclidean error of
the atomic forces among all atoms.  Fig.~\ref{fig:Al6x9forceerr} shows that
both accuracy measurements are nearly the same using uniform and
non-uniform refinement.  We therefore see that non-uniform refinement
significantly improves the efficiency of the adaptive local basis functions
by simply redistributing basis functions according to the predictions of
the local estimator, without sacrificing accuracy.

\begin{figure}[ht]
  \centering
  \subfloat[Error per atom of total energy (Hartree/atom) for uniform refinement (U.
	Ref.) and non-uniform
	refinement (N.U. Ref), along with the global estimator. We also
        show the individual terms of the estimator summed over all elements
        $K \in \mc{T}$; the residual $\eta_{R}^{2}$, the discontinuity
	of the gradient $\eta_{G}^2$, and the discontinuity of the function values
	$\eta_{V}^2$.
	\label{fig:Al6x9energyerr}]{%
    \includegraphics[width=0.47\textwidth]{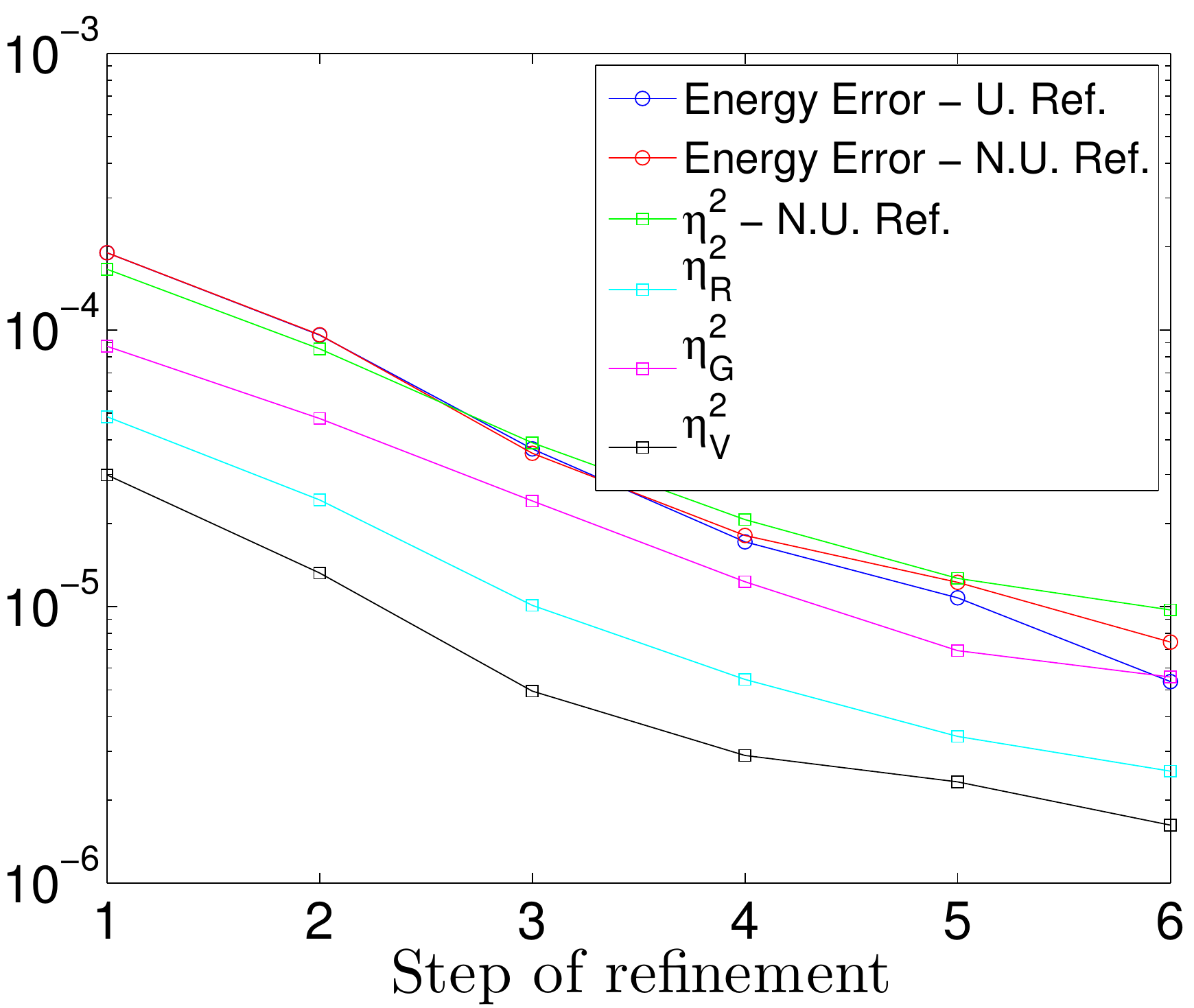}
  }
  \hfill
  \subfloat[Error of atomic force (au) for uniform and non-uniform refinement. \label{fig:Al6x9forceerr}]{%
    \includegraphics[width=0.47\textwidth]{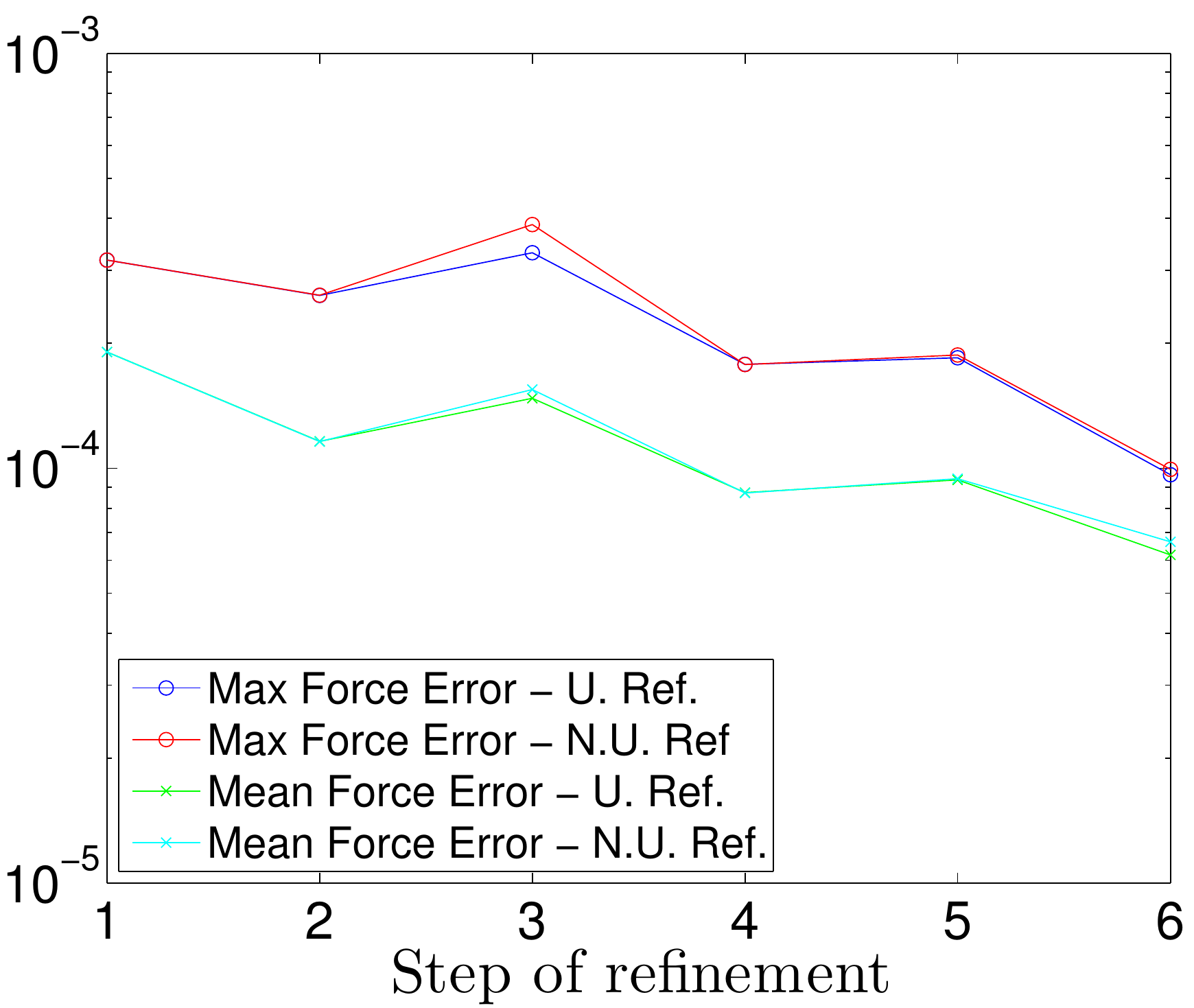}
  }

  \subfloat[Number of basis functions per atom for uniform and
  non-uniform refinement. \label{fig:Al6x9totalbasis}]{%
    \includegraphics[width=0.47\textwidth]{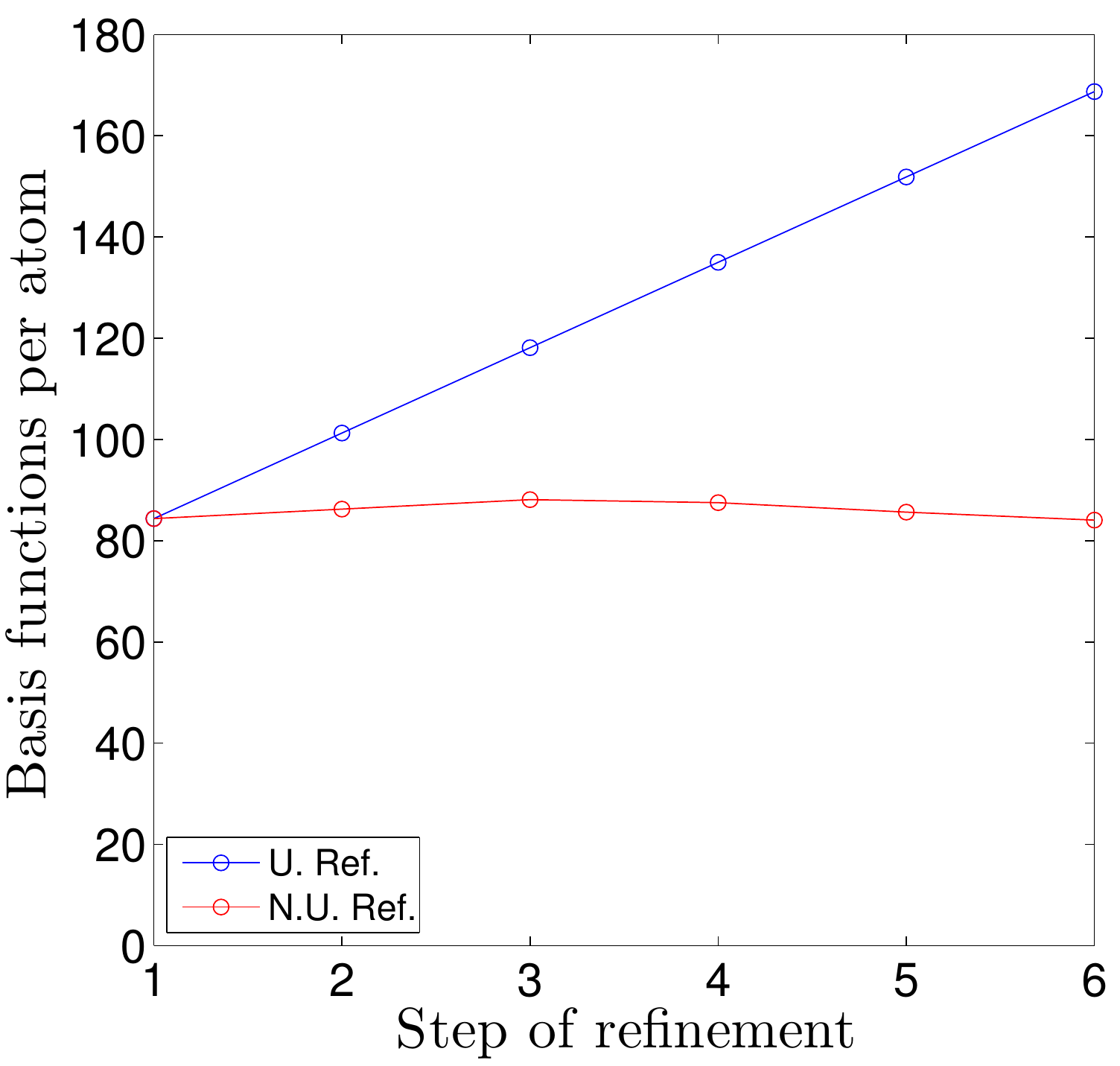}
  }

  \caption{Accuracy and savings for aluminum system calculations. The
	non-uniform refinement scheme produces results which are nearly as
	accurate as those produced by the uniform refinement scheme, using
	only half as many basis functions.}
  \label{fig:Al6x9errorandtotbasis}
\end{figure}

We lastly demonstrate that the use of non-uniformly distributed adaptive
local basis functions gives significant savings over the use of
planewave basis functions as in ABINIT. The number of basis functions in
planewave-based electronic structure calculations is usually denoted by
the kinetic energy cutoff ($\Ecut$). In atomic units, the relation
between the total number $N_{\mathrm{pw}}$ of basis functions used in a
planewave calculation and $\Ecut$ is approximately (depending on
adjustments to the number of grid points for performance
improvements)~\cite{Martin2004}
\begin{equation}
	N_{\mathrm{pw}} = \left( \frac{\sqrt{2\Ecut}}{\pi} \right)^{3}
	\mathrm{Vol},
	\label{eqn:pwnum}
\end{equation}
where $\mathrm{Vol}$ is the volume of the computational domain.

Fig.~\ref{fig:Al6x9PWComparison} plots the number of basis
functions per atom against the error per atom of the total energy using the planewave basis
set obtained from ABINIT, a uniform distribution of adaptive local basis
functions obtained from DGDFT, and a non-uniform distribution of adaptive
local basis functions obtained from DGDFT. We observe that DGDFT
calculations use many fewer basis functions than ABINIT calculations to
achieve similar error. For example, to achieve an error of the total energy
on the order of $10^{-6}$ Hartree per atom, ABINIT requires around $10,000$
planewave basis functions per atom, whereas DGDFT requires approximately
$84$ adaptive local basis functions per atom when non-uniform basis
refinement is used. 

\begin{figure}[ht]
  \centering
  \includegraphics[width=0.47\textwidth]{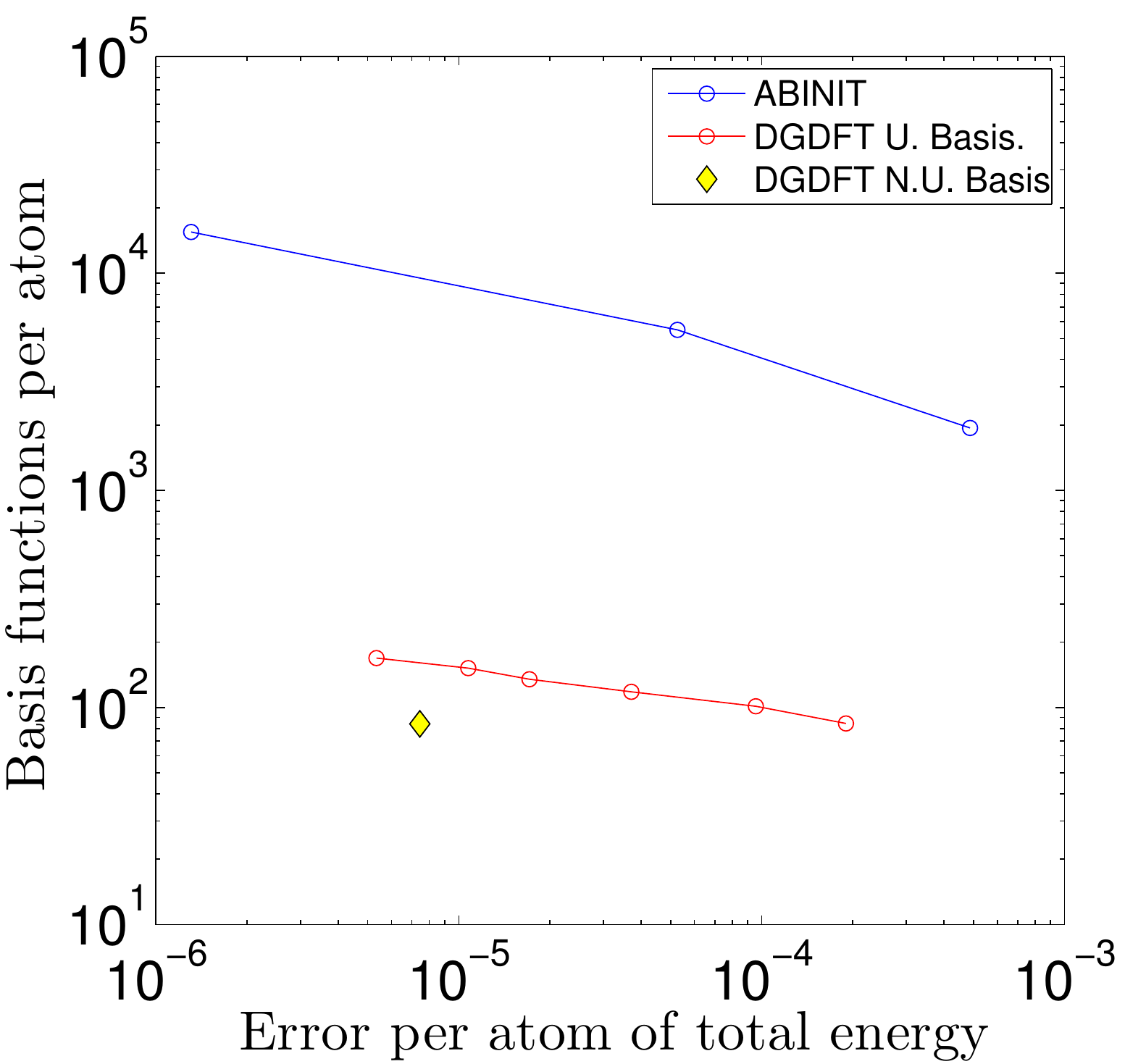}
  \caption{Number of basis functions used to achieve a given error of
    the total energy for the aluminum example. Non-uniform adaptive
    local basis sets obtained from DGDFT are approximately two orders of
    magnitude smaller than planewave basis sets obtained from ABINIT for
    similar error.}
  \label{fig:Al6x9PWComparison}
\end{figure}

\subsection{Graphene oxide}

We next consider a layer of graphene oxide in water, consisting of $160$
atoms, with the atomic configuration shown in Fig.~\ref{fig:GOconfig}.
The size of the supercell is $19.63$ au, $56.69$ au and $28.35$ au along
the $x, y,$ and $z$ directions, respectively. The computational domain is
partitioned into a $4 \times 12 \times 6$ grid of elements. As in the
case of the aluminum system, 
a uniform grid is used to represent quantities such as the electron density
and potential in the global domain in DGDFT, with $160$, $456$, and $228$
grid points along the
$x$, $y$, and $z$ directions, respectively. A Legendre-Gauss-Lobatto (LGL) grid is
used inside each element to construct the DG Hamiltonian matrix, and
the number of LGL grid points inside each element is $80$, $76$, and $76$
along the $x$, $y$, and $z$ directions, respectively.
Fig.~\ref{fig:GOdenslices} shows slices of the density
across various planes. 

We perform both uniform and non-uniform refinement of the basis
functions. We begin with a uniform distribution of 40 basis functions in
each element for both schemes. The basis refinement step size is set at
$b_{\text{step}} = 10$ and the number of refinement steps at $n = 6$.
For the non-uniform refinement scheme, we use $\epsilon_{max} = 10^{-4}$
and $\epsilon_{min} = 10^{-6}$. 
Fig.~\ref{fig:GObasisnu6} shows the final distributions of
adaptive local basis functions generated by the non-uniform refinement
process in the elements containing the slices shown in
Fig.~\ref{fig:GOdenslices}. The distribution of
adaptive local basis functions corresponds well to the highly
inhomogeneous electron density of the system. 

\begin{figure}[ht]
	\begin{center}
    \includegraphics[width=0.60\textwidth]{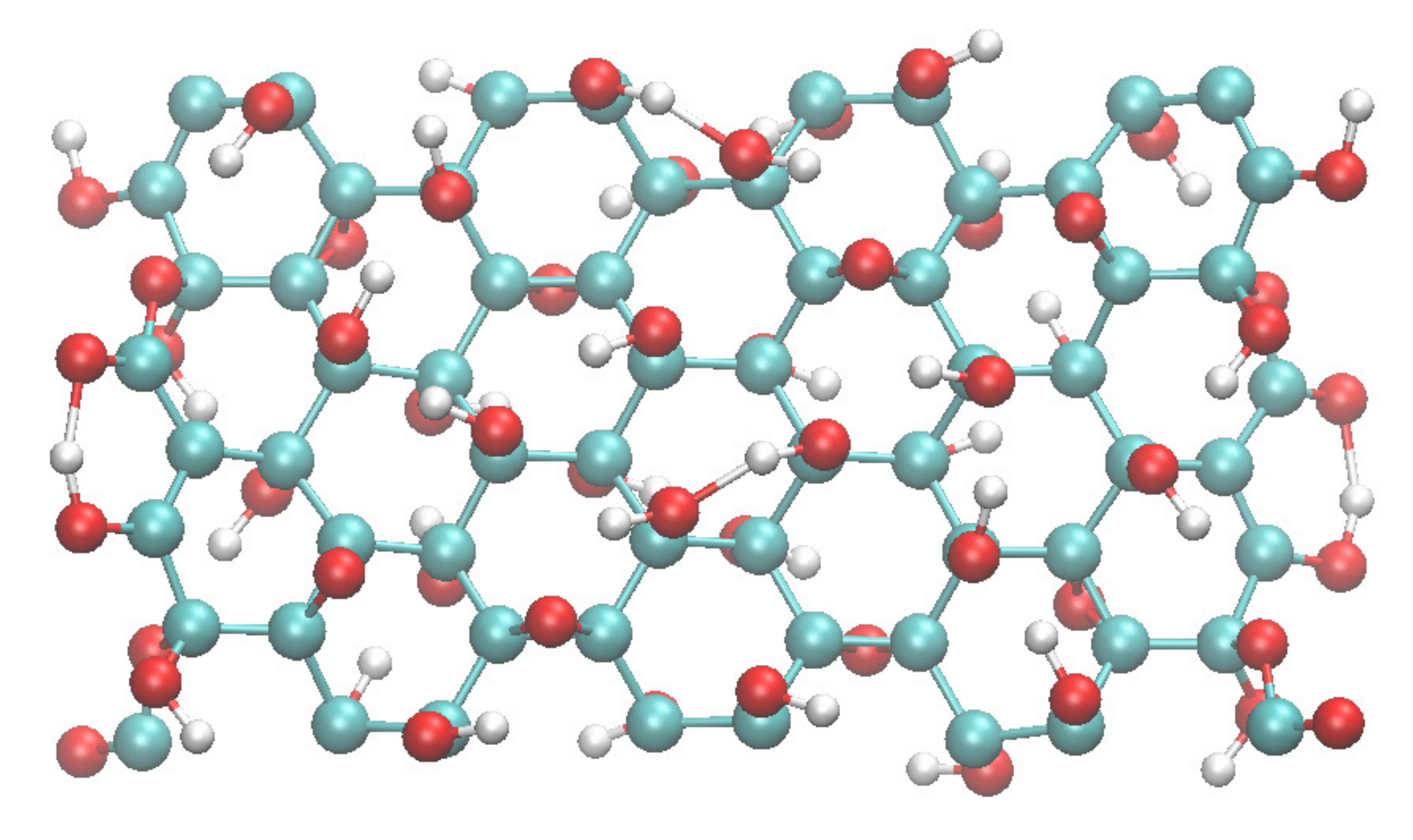}
	\end{center}
	\caption{Atomic configuration of a single layer of graphene oxide in
	water with $160$ atoms.}
	\label{fig:GOconfig}
\end{figure}

\begin{figure}[ht]
  \centering
  \subfloat[Fixed z = 9.2975 \label{fig:GOxyslice1}]{%
    \includegraphics[width=0.47\textwidth]{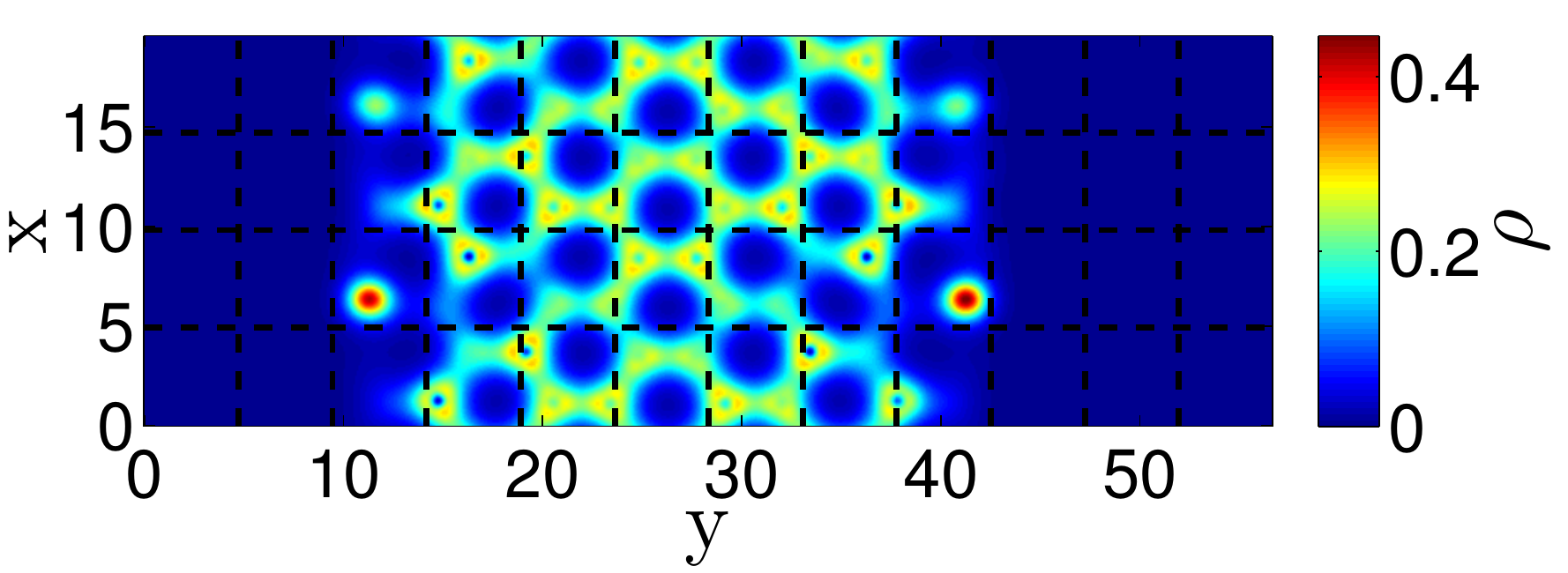}
  }
  \hfill
  \subfloat[Fixed z = 12.3966 \label{fig:GOxyslice2}]{%
    \includegraphics[width=0.47\textwidth]{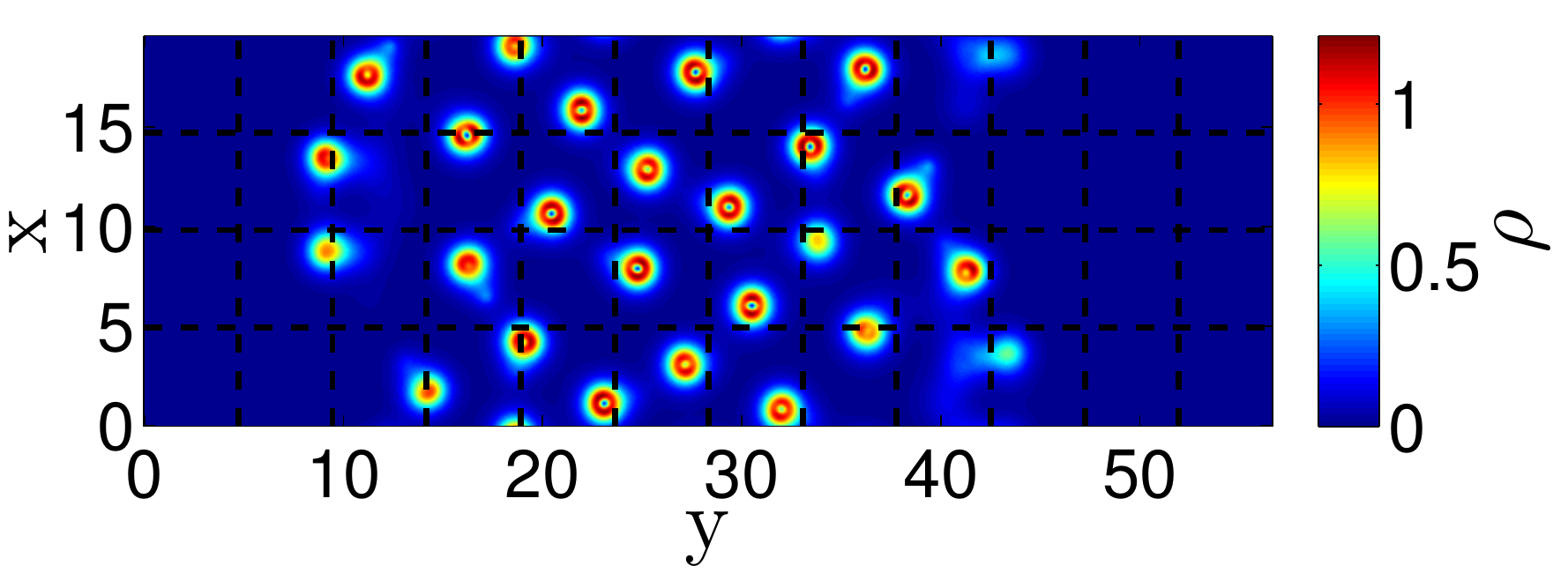}
  }

  \subfloat[Fixed x = 10.9079 \label{fig:GOyzslice}]{%
    \raisebox{-0.5\height}{\includegraphics[width=0.47\textwidth]{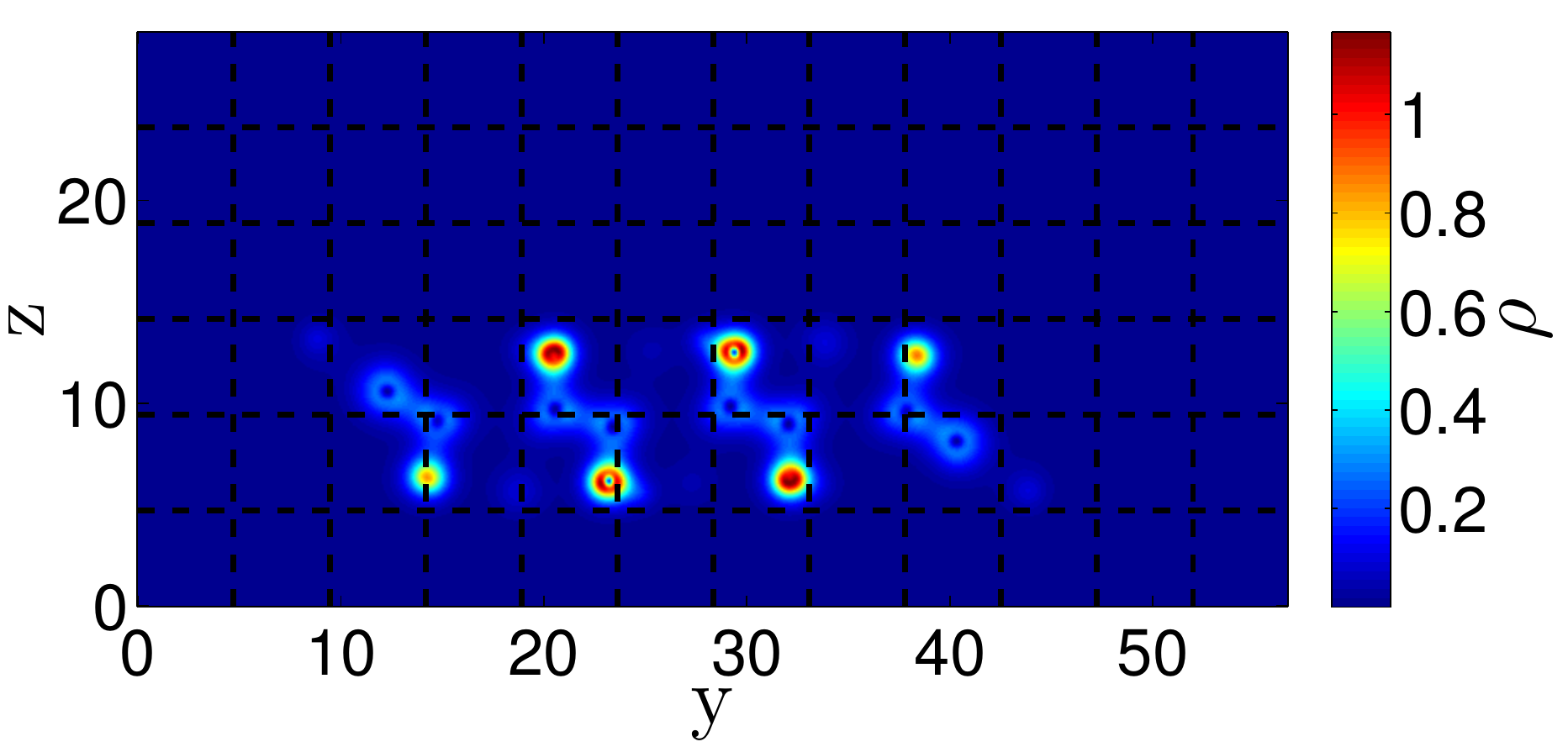}}
  }
  \hfill
 \subfloat[Fixed y = 24.0205 \label{fig:GOxzslice}]{%
   \raisebox{-0.5\height}{\includegraphics[width=0.47\textwidth]{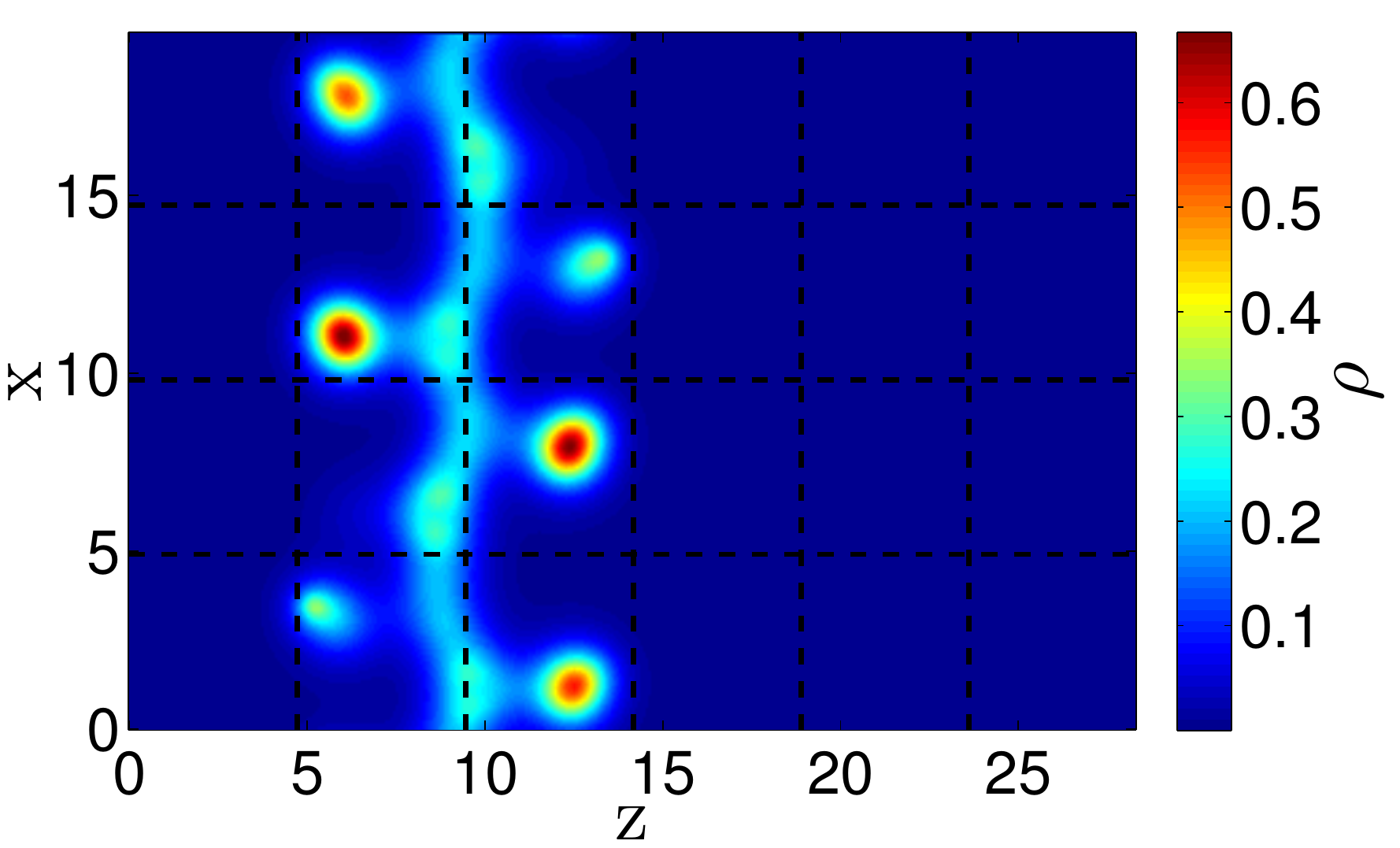}}
  }

  \caption{Two-dimensional slices of electron density for graphene oxide
  system, with $4 \times 12 \times 6$ grid used in DGDFT calculations indicated by black dashed
lines.}
  \label{fig:GOdenslices}
\end{figure}

\begin{figure}[ht]
  \centering
  \subfloat[Fixed z = 9.2975
  \label{fig:GOxybasisnu6}]{%
    \includegraphics[width=0.47\textwidth]{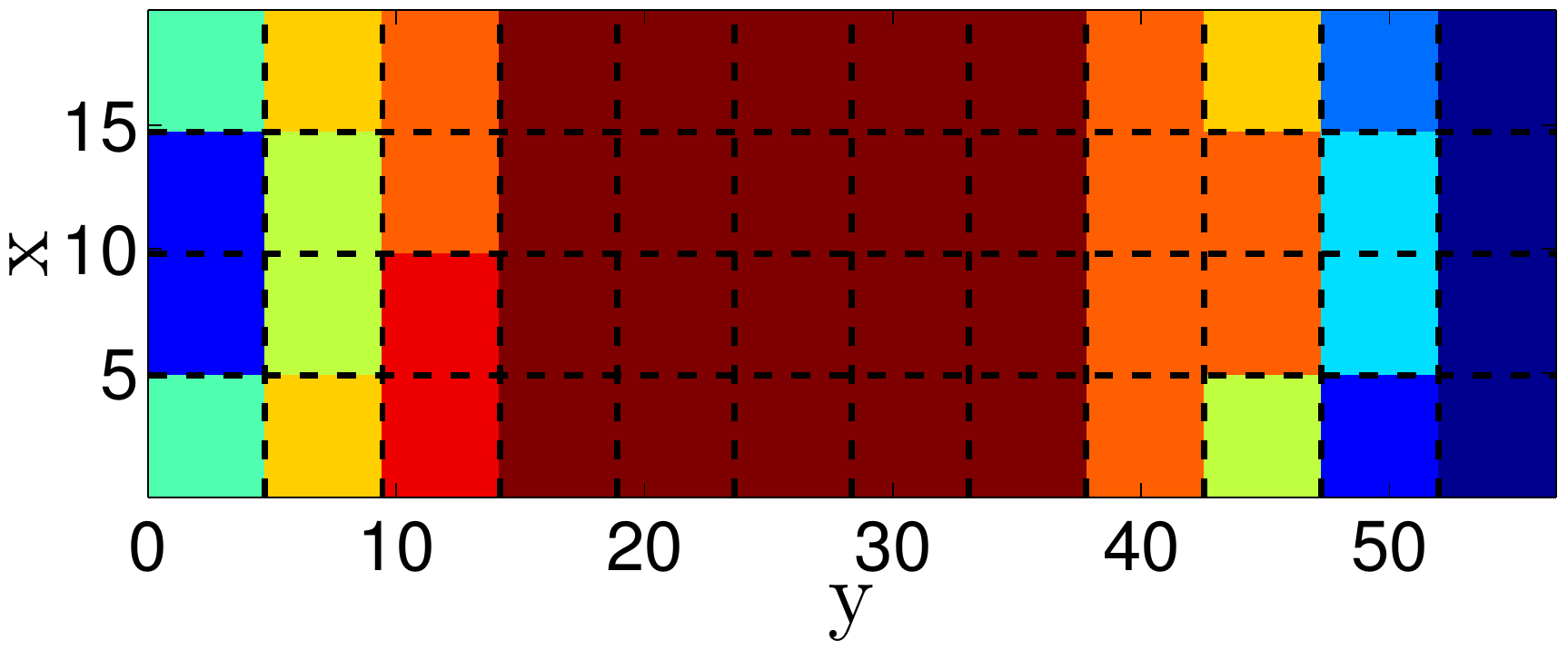}
  }
  \hfill
  \subfloat[Fixed z = 12.3966
  \label{fig:GOxybasisnu62}]{%
    \includegraphics[width=0.47\textwidth]{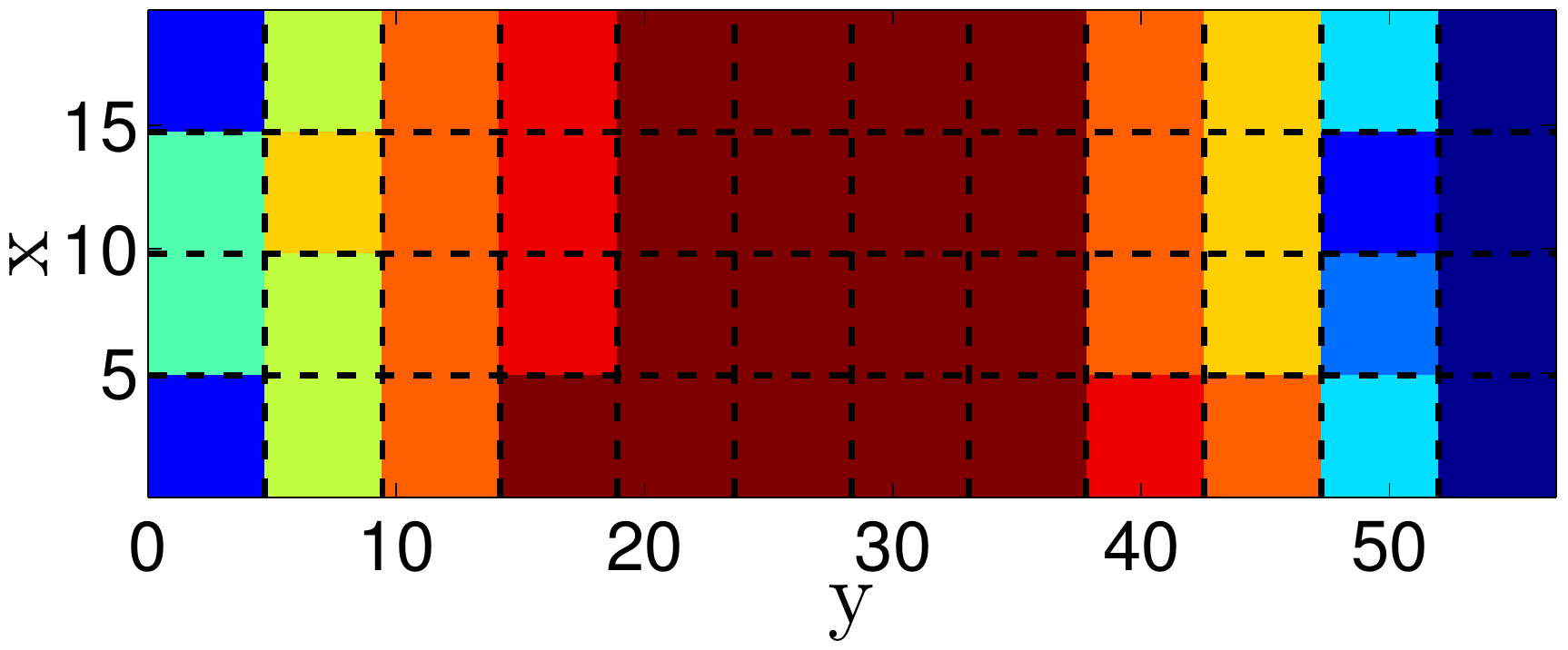}
  }

  \subfloat[Fixed x = 10.9079 \label{fig:GOyzbasisnu6}]{%
    \raisebox{-0.5\height}{\includegraphics[width=0.47\textwidth]{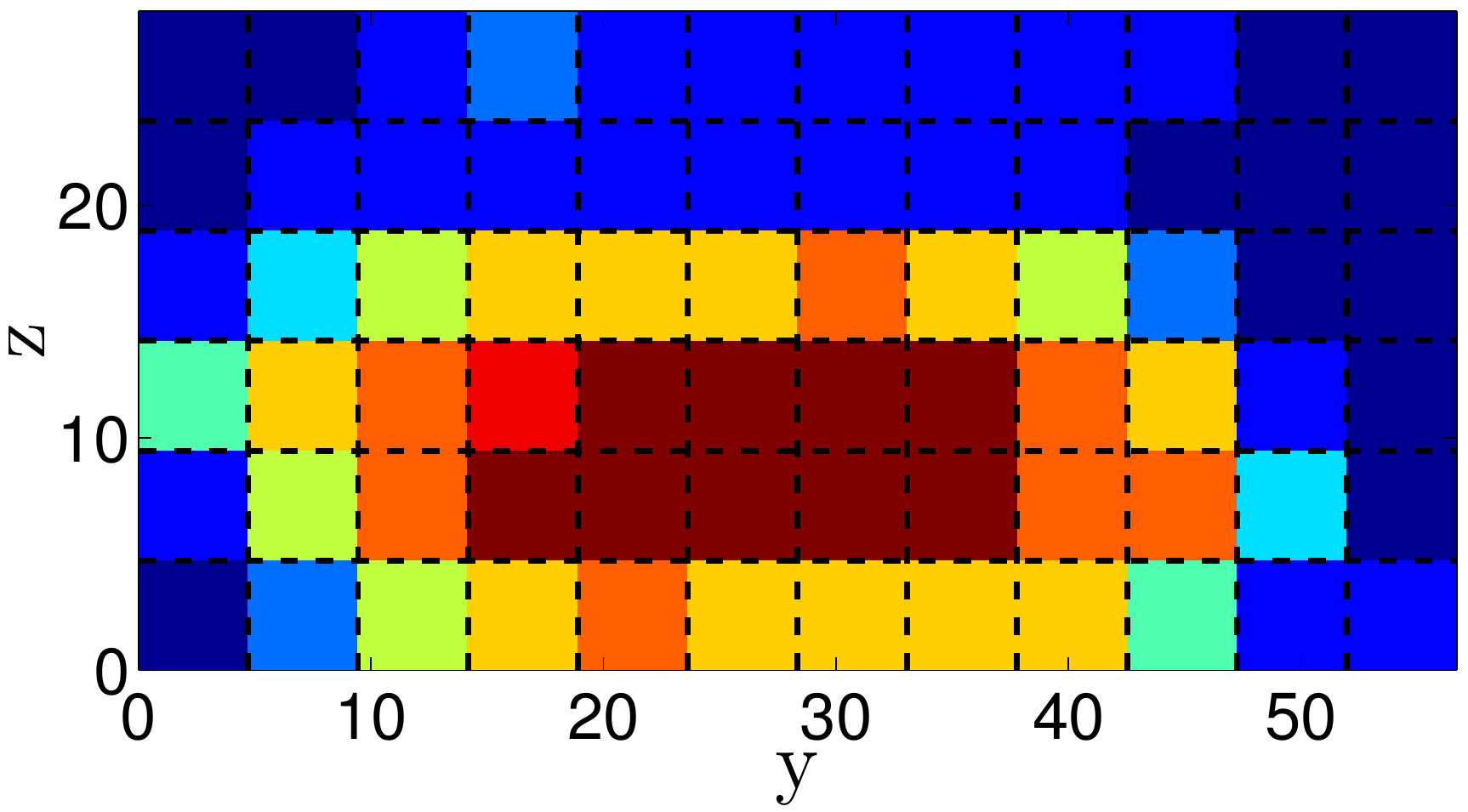}}
  }
  \hfill
  \subfloat[Fixed y = 24.0205 \label{fig:GOxzbasisnu6}]{%
    \raisebox{-0.5\height}{\includegraphics[width=0.47\textwidth]{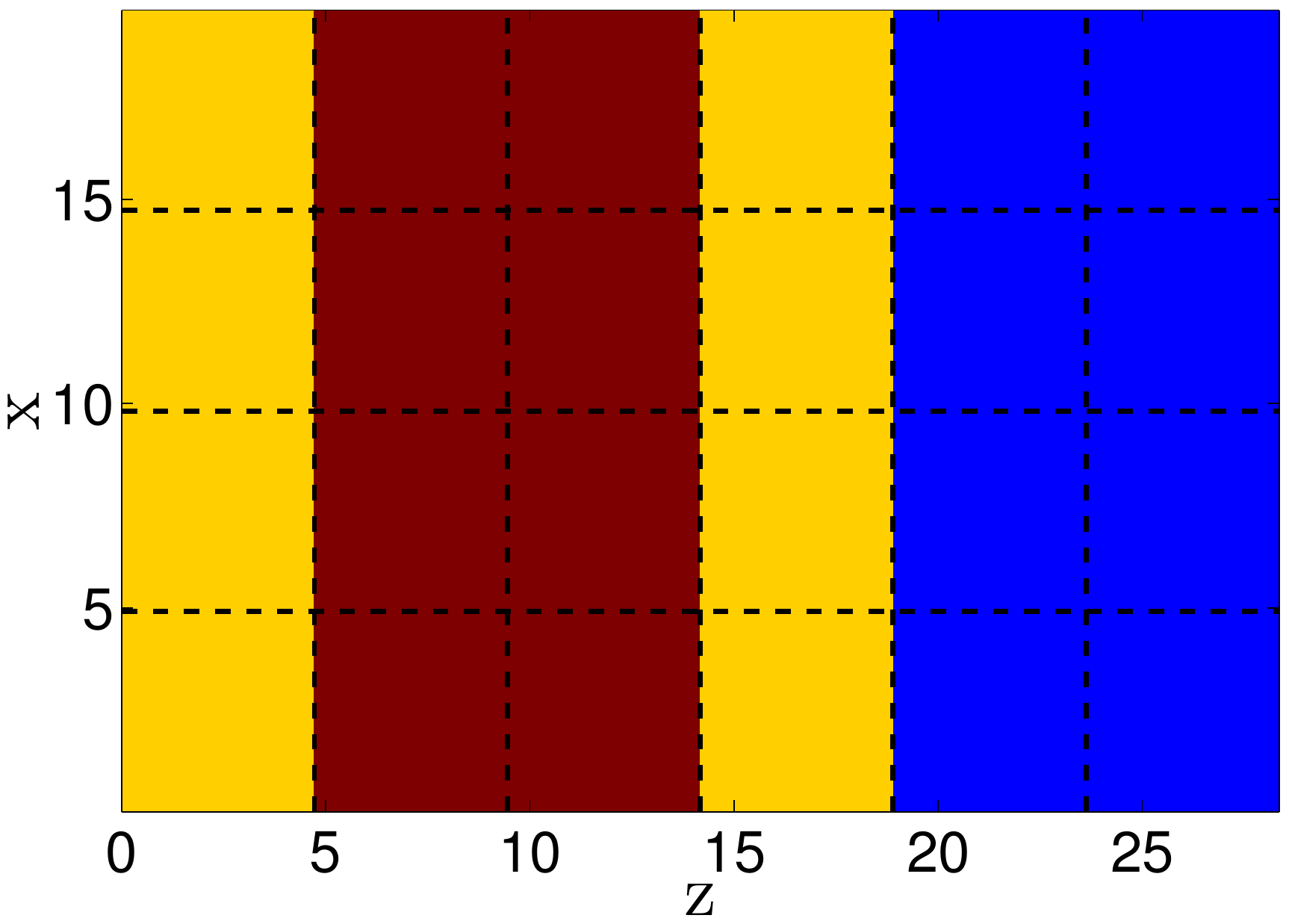}}
  }

  \vspace{-4mm}
  \subfloat{%
    {\includegraphics[width=0.5\textwidth]{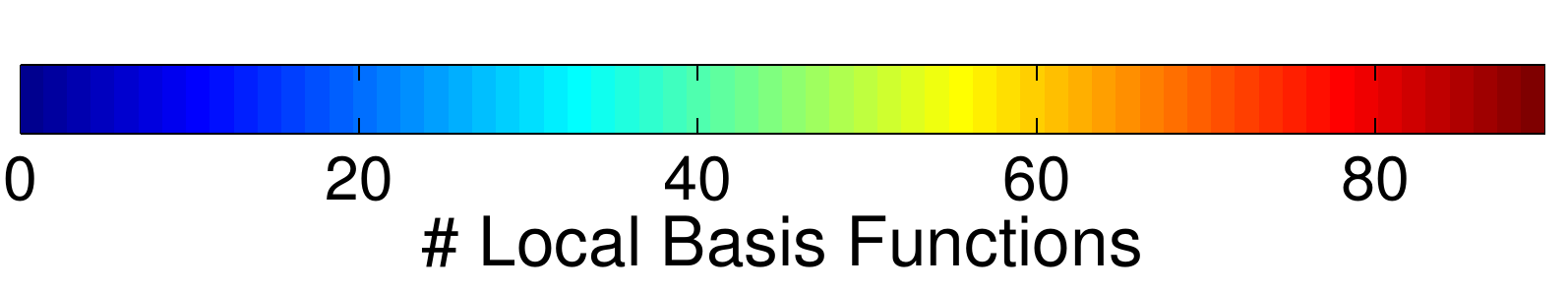}}
  }

  \caption{Distributions of adaptive local basis functions at the last step of
  non-uniform refinement in the elements containing the slices shown in Figure
\ref{fig:GOdenslices}. The distribution of basis functions closely reflects
the density.}
  \label{fig:GObasisnu6}
\end{figure}

Fig.~\ref{fig:GOerrorandtotbasis} demonstrates the accuracy and
efficiency of the non-uniform refinement procedure for the graphene oxide
calculation. Fig.~\ref{fig:GOenergyerr} shows that the difference in the
error of the total energy per atom for the two refinement schemes is
very small, and that the global estimator is highly effective in
capturing the error of the total energy. The maximum and average errors of the atomic
forces among all the atoms are given in Fig.~\ref{fig:GOforceerr} for the
uniform and non-uniform refinement schemes, which again give similar
accuracy. Fig.~\ref{fig:GOtotalbasis} shows
that by the final step of refinement,
the uniform refinement scheme uses nearly 2.5 times as many basis
functions as the non-uniform refinement scheme. Furthermore, the
non-uniform refinement scheme uses fewer basis functions at the final
step than it does at the first step, though the error is four orders of
magnitude lower. 

\begin{figure}[ht]
  \centering
	\subfloat[Error per atom of total energy (Hartree/atom) for uniform refinement (U.
	Ref.) and non-uniform
	refinement (N.U. Ref), along with the global estimator. We also
        show the individual terms of the estimator summed over all elements
        $K \in \mc{T}$; the residual $\eta_{R}^{2}$, the discontinuity
	of the gradient $\eta_{G}^2$, and the discontinuity of the function values
	$\eta_{V}^2$.
      \label{fig:GOenergyerr}]{%
    \includegraphics[width=0.47\textwidth]{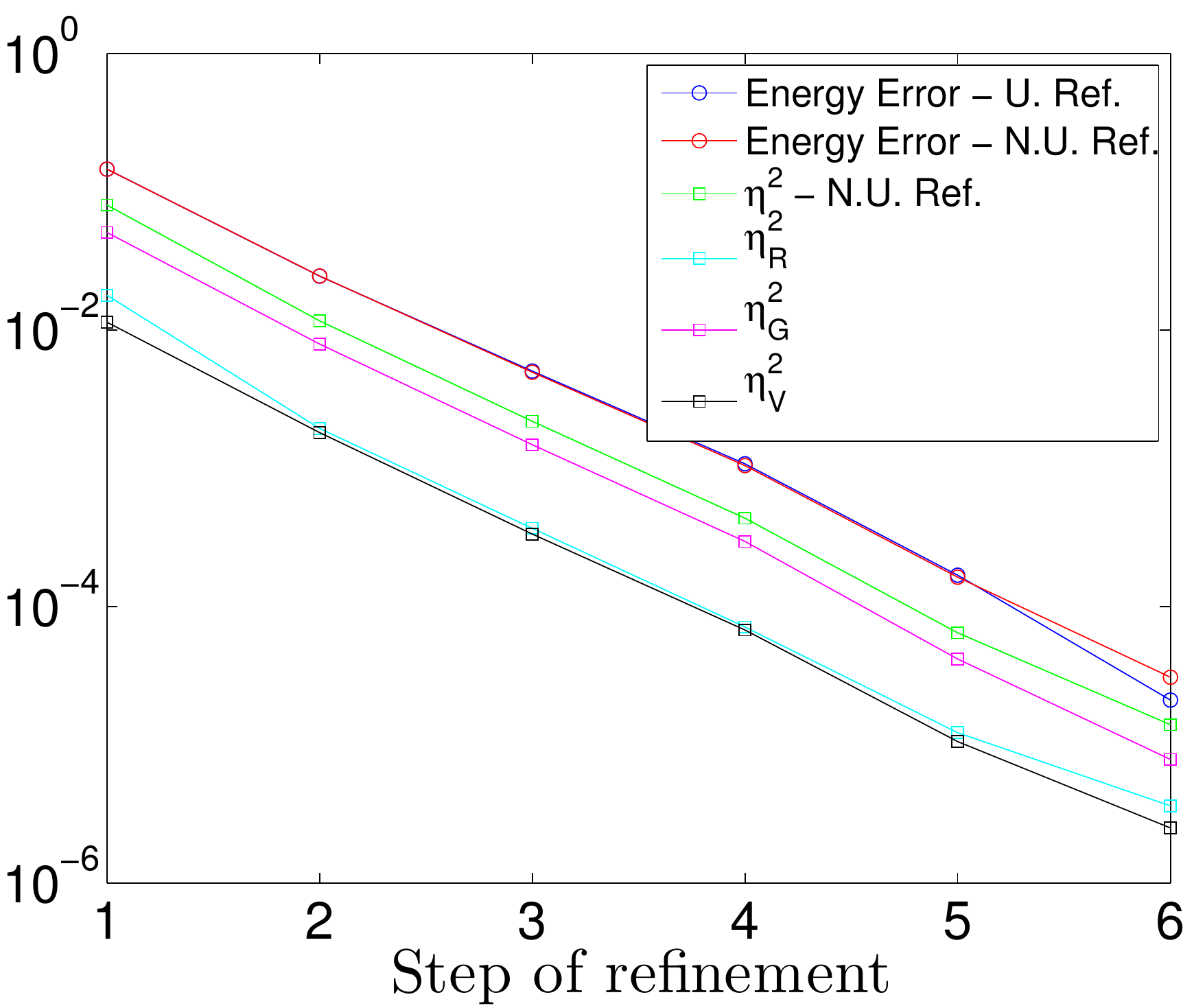}
  }
  \hfill
  \subfloat[Error of atomic force (au) for uniform and non-uniform refinement.
  \label{fig:GOforceerr}]{%
    \includegraphics[width=0.47\textwidth]{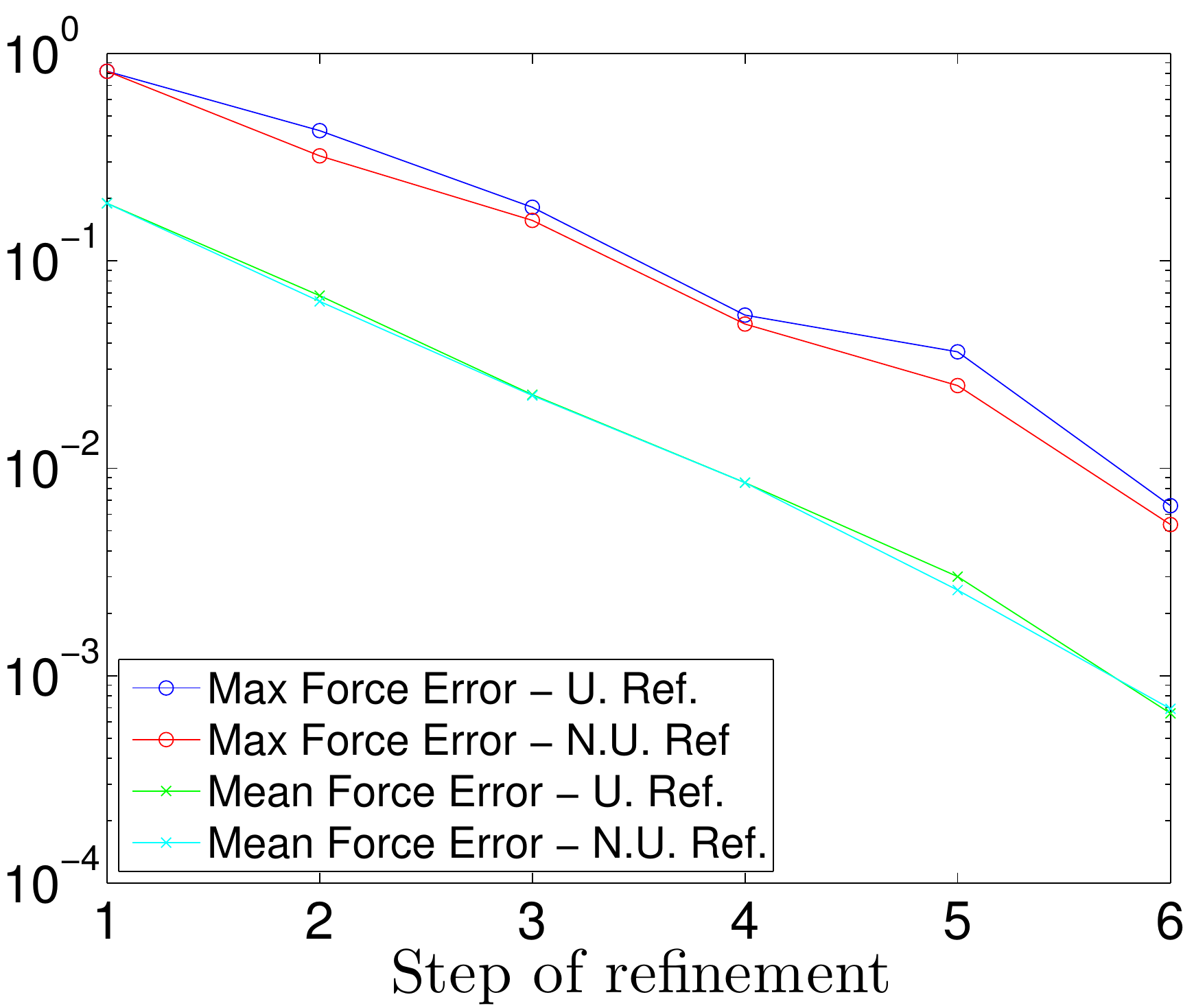}
  }

  \subfloat[Number of basis functions per atom for uniform and
  non-uniform refinement. \label{fig:GOtotalbasis}]{%
    \includegraphics[width=0.47\textwidth]{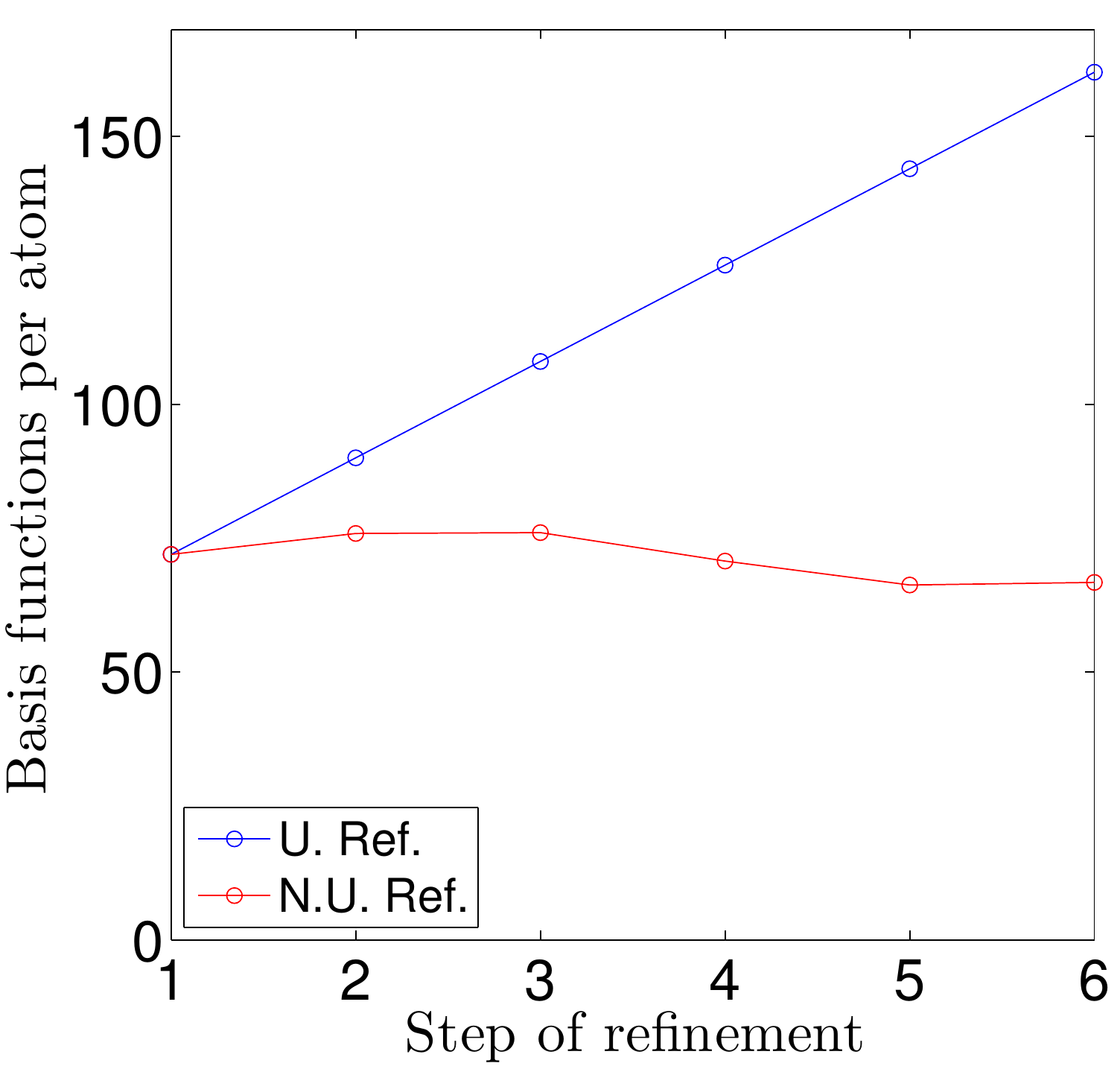}
  }

  \caption{Accuracy and savings for graphene oxide system calculations. As
  in the aluminum case, the non-uniform refinement scheme is nearly as
accurate and more efficient than the uniform scheme, and the global
estimator is asymptotically effective.}
  \label{fig:GOerrorandtotbasis}
\end{figure}

Another perspective on the effectiveness of the non-uniform refinement
strategy is given in Fig.~\ref{fig:GOlocalestimators}, in which we plot the
quintile values of the local estimator across all elements at each step
of uniform and non-uniform refinement, along with the error thresholds
$\epsilon_{\max}$ and $\epsilon_{\min}$. In the uniform refinement
scheme, both under and over-resolved elements become further refined at
each step.  In the non-uniform refinement scheme, only the under-resolved
elements are further refined; the over-resolved elements are instead
made less refined at each step, thus recovering inefficiently-allocated
resources. The dominant sources of error are similar in both cases, so
very little loss in accuracy is incurred by using a non-uniform basis
refinement scheme. 

\begin{figure}[ht]
  \centering
  \subfloat[Uniform refinement \label{fig:GOlocalestimatorsu}]{%
    \includegraphics[width=0.47\textwidth]{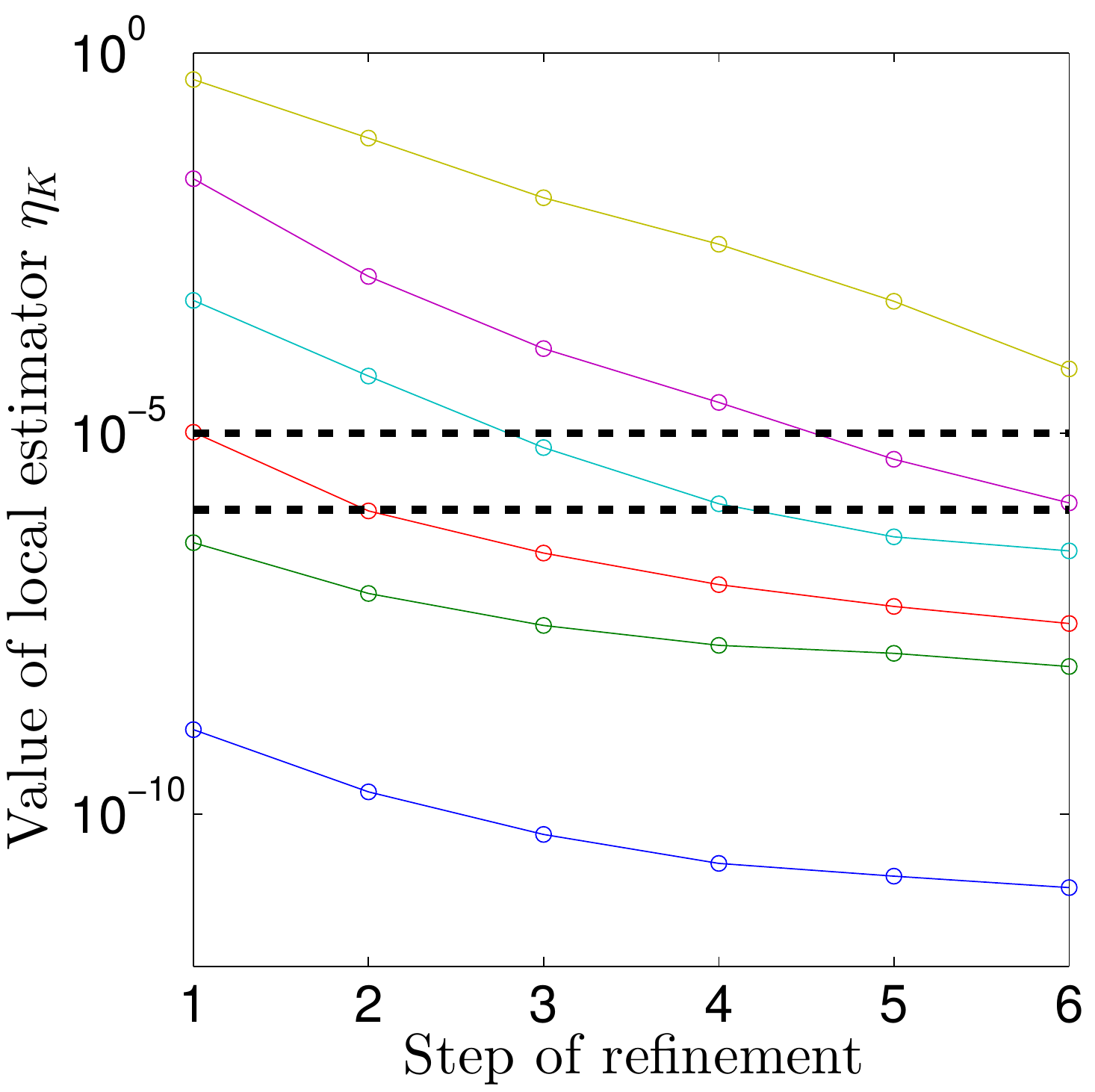}
  }
  \hfill
  \subfloat[Non-uniform refinement \label{fig:GOlocalestimatorsnu}]{%
    \includegraphics[width=0.47\textwidth]{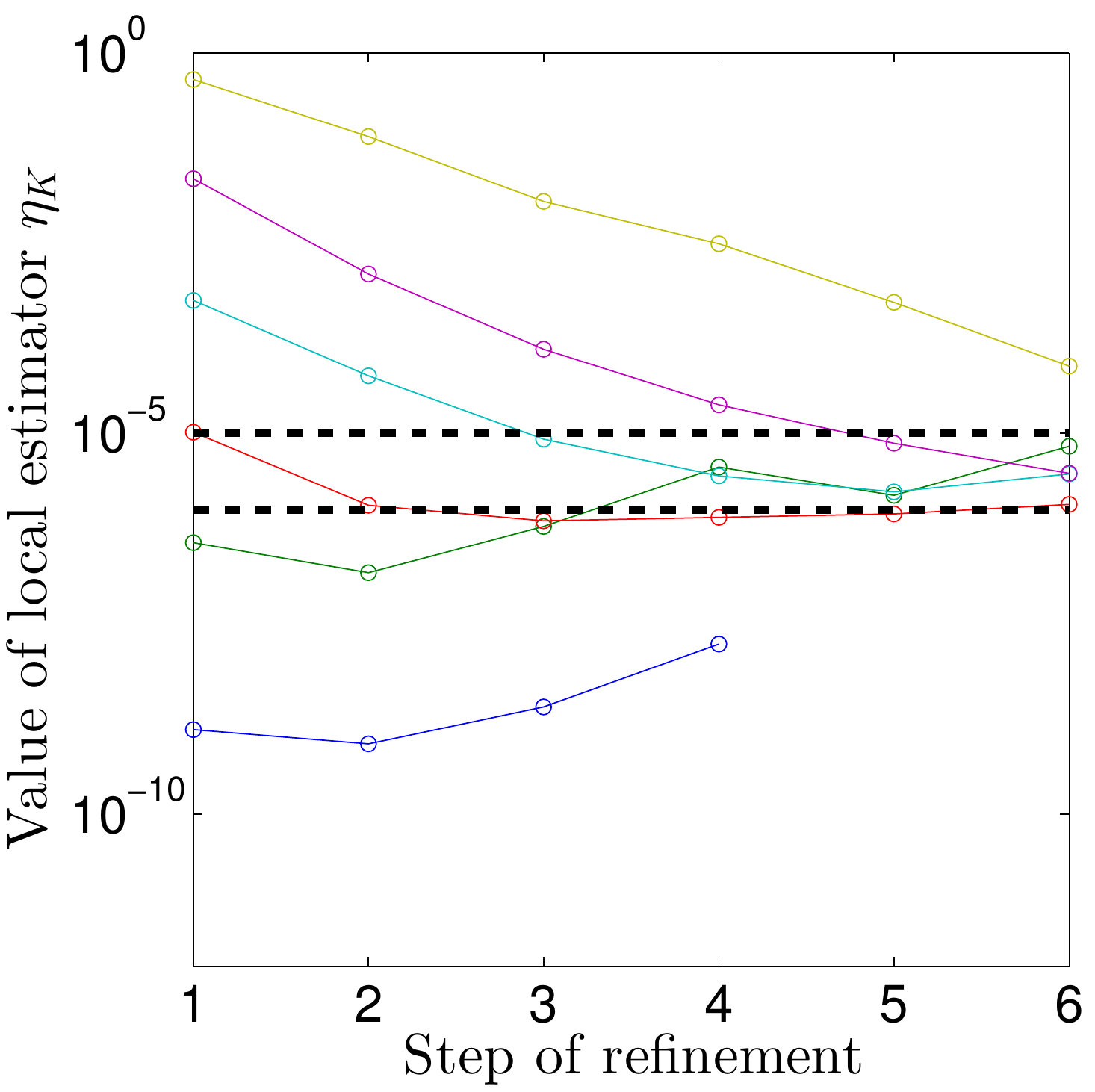}
  }

  \caption{Quintile values of local estimator for uniform and non-uniform basis refinement, along with
$\epsilon_{\max}$ and $\epsilon_{\min}$. The dominant local errors are roughly the
same in both elements, but the non-uniform refinement scheme produces
savings in the elements with small local error. Note that the minimum value of the local
estimator across all elements becomes $0$ by the fifth step of non-uniform refinement.}
  \label{fig:GOlocalestimators}
\end{figure}

Fig.~\ref{fig:GOPWComparison} compares the number of planewave basis functions
used by ABINIT with the number of adaptive local basis functions used by
DGDFT to achieve similar error for the graphene oxide example. We observe,
for example, that to achieve an error of the total energy on the order of
$10^{-5}$ Hartree per atom, ABINIT requires around $35,000$ planewave basis
functions per atom, whereas DGDFT requires approximately $67$ adaptive
local basis functions per atom when non-uniform basis refinement is used.

\begin{figure}[ht]
  \centering
  \includegraphics[width=0.47\textwidth]{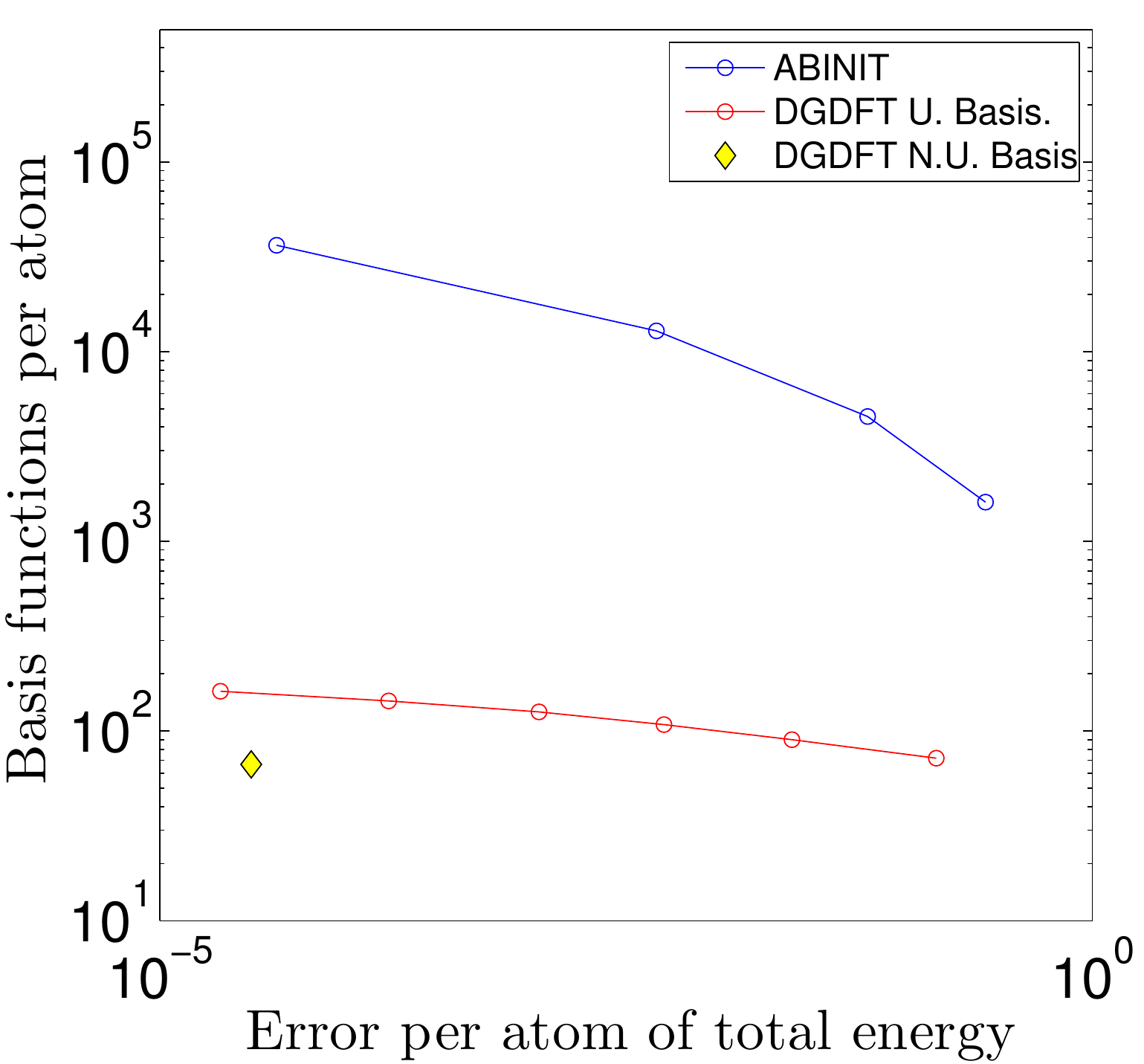}
  \caption{Number of basis functions used to achieve a given error of
    the total energy for the GO example. Non-uniform adaptive
    local basis sets obtained from DGDFT are approximately three orders of
    magnitude smaller than planewave basis sets obtained from ABINIT for
    similar error.}
  \label{fig:GOPWComparison}
\end{figure}

\section{Conclusion and future work}\label{sec:conclusion}
In this paper we develop residual-based a posteriori local error
estimates for solving KSDFT using adaptive local basis functions in a
discontinuous Galerkin framework.  Compared with standard $hp$-refinement,
the major difficulty in carrying out the analysis is that little is known
about the analytical properties of the adaptive local basis functions.
We therefore effectively need to perform an a posteriori error analysis for 
general non-polynomial basis functions. In order to proceed we postulate
a set of assumptions on the basis set under which we are
able to prove that the a posteriori error estimator is an
upper bound for the errors of both eigenvalues and eigenvectors, up to
terms which are of higher order in the context of standard $hp$-refinement.
We acknowledge that there is still difficulty in verifying the
assumptions for a specific non-polynomial basis set such as
the ALBs. We plan to carry out a numerical verification of the
assumptions for the adaptive local basis functions in the future. The
hope is that such a study can illuminate the approximation properties of
the ALB sets from an analytic point of view, and also inform the a
priori error analysis.

From a practical point of view, the results given by the local error
estimator used in DGDFT mimic the $hp$-refinement results. As a first
attempt we reinterpret $p$ as the number of adaptive local basis
functions used in each element. We consider this only as a starting
point for the study, but nevertheless, the numerical results show that
the approach of a posteriori error estimates is promising for KSDFT
studies. We demonstrate the practical use of the a posteriori error
estimator for three-dimensional KSDFT calculations for quasi-2D aluminum
surfaces and a single layer graphene oxide system in water. Besides the systems
with a vacuum region, we also plan to use the non-uniform refinement
strategy to study systems with defects and dislocations. In these cases,
the electron density and the magnitude of the estimator may not vary as
much as in the case of a large vacuum region. 

Besides the adaptive refinement, another important application of a
posteriori error estimation is to capture the error due to the finite
dimensional approximation of physical quantities such as
total energies without performing a more refined calculation, which can
be prohibitively expensive in practical KSDFT calculations.  The error
of the total energy is directly related to the error of eigenvalues.
However, this objective requires a much more precise estimator than is used
in the current formulation, and this will be our future work. 

\section*{Acknowledgments}

This work was partially supported by the Science Undergraduate Laboratory
Internship (SULI) program of Lawrence Berkeley National Laboratory (J. K.),
by the Laboratory Directed Research and Development Program of Lawrence
Berkeley National Laboratory under the U.S.  Department of Energy contract
number DE-AC02-05CH11231, and by the Scientific Discovery through Advanced
Computing (SciDAC) program funded by U.S. Department of Energy, Office of
Science, Advanced Scientific Computing Research and Basic Energy Sciences
(L. L.  and C. Y.).  We are grateful to Roberto Car and Limin Liu for
providing the atomic configuration of the graphene oxide in water. We
would also like to thank Eric Canc\`{e}s, Yvon Maday and Benjamin Stamm for
helpful discussions.

\FloatBarrier

\appendix
\section{Details for proving Theorem~\ref{thm:eigfunc} and~\ref{thm:eigval}}

\begin{lemma}
	Let $u_{i,\mc{J}}^{r}$ be defined as in Eq.~\eqref{eqn:splitu}, then
	\begin{equation}
		\normET{u_{i,\mc{J}}^{r}} \lesssim \eta_{i}.
		\label{}
	\end{equation}
	\label{lem:urbound}
\end{lemma}
\begin{proof}
	Using Eq.~\eqref{eqn:IJcondition1}
	\begin{equation}
		\sum_{K\in \mc{T}}\norm{\nabla u_{i,\mc{J}}^{r} }^2_{K} 
		\lesssim \sum_{F\in \mc{S}}\gamma_{2}(J_{F}) \alpha^{2}(J_F) 
		\norm{\jump{u_{i,\mc{J}}}}^{2}_{F}
		\lesssim \sum_{K\in \mc{T}} \eta_{i,V_K}^2 \le \eta_{i}^2.
		\label{}
	\end{equation}
	Using Eq.~\eqref{eqn:constantCondition2},
	\begin{equation}
		\sum_{F\in \mc{S}} \alpha(J_F) 
		\norm{\jump{u_{i,\mc{J}}}}^{2}_{F}
		\lesssim \sum_{F\in \mc{S}} \gamma_{2}(J_F) \alpha^2(J_F) 
		\norm{\jump{u_{i,\mc{J}}}}^{2}_{F}
		\lesssim \sum_{K\in \mc{T}} \eta_{i,V_K}^2 \le \eta_{i}^2.
		\label{}
	\end{equation}
	Then Lemma~\ref{lem:urbound} follows directly from the definition of
	the energy norm~\eqref{eqn:energynorm}.
\end{proof}

\begin{definition}[Split of $A_{\mc{J}}$]
	The bilinear form $A_{\mc{J}}$ in Eq.~\eqref{eqn:bilinearDG} 
	can be split into two parts
	\begin{equation}
		A_{\mc{J}}(u,v) = D_{\mc{J}}(u,v) + K_{\mc{J}}(u,v).
		\label{}
	\end{equation}
	Here
	\begin{equation}
		D_{\mc{J}}(u,v) = \frac{1}{2}\average{\nabla u,\nabla v}_{\mc{T}}
                + \sum_{F\in \mc{S}} \alpha(J_F) \average{\jump{u},\jump{v}}_{F},
		\label{}
	\end{equation}
	and
	\begin{equation}
		K_{\mc{J}}(u,v) = - \frac{1}{2} \average{\mean{\nabla u},\jump{v}}_{\mc{S}}
		- \frac{1}{2} \average{\mean{\nabla v},\jump{u}}_{\mc{S}}.
		\label{}
	\end{equation}
	Formally, it is clear that
	\begin{equation}
		A(u,v) = D_{\mc{J}}(u,v),\quad \forall u,v\in H^{1}_{\pi}(\Omega).
		\label{}
	\end{equation}
\end{definition}

\begin{lemma}
  If Assumption~\ref{asp:regularity} holds, 
	let $(\varepsilon_{i,\mc{J}},u_{i,\mc{J}})$ be a computed eigenpair
	corresponding to~\eqref{eqn:weakDGeig}, and $(\varepsilon_{i},u_{i})$
	an eigenpair corresponding to~\eqref{eqn:DGeig}.
	For any $v\in H^{1}_{\pi}(\Omega)$, let $v_{\mc{J}}\in
	\mc{V}_{\mc{J}}(\mc{T})$ satisfy~\eqref{eqn:gamma1}
	and~\eqref{eqn:gamma2}. Then we have
	\begin{equation}
		\begin{split}
                  & \average{\varepsilon_{i} u_{i}-V_{\eff}u_i,v-v_{\mc{J}}}_{\mc{T}} -
		D_{\mc{J}}(u_{i,\mc{J}},v-v_{\mc{J}}) +
		K_{\mc{J}}(u_{i,\mc{J}}, v_{\mc{J}})  \\
    &~~~\lesssim \left( \eta_{i} + \sqrt{\gamma_{1,\mc{J}}}\xi_{i} \right) \normET{v}.
		\end{split}
		\label{eqn:ucboundaux}
	\end{equation}
	\label{lem:ucboundaux}
\end{lemma}

\begin{proof}
For brevity, set 
\begin{equation}
	T = \average{\varepsilon_{i} u_{i} - V_{\eff} u_{i} ,v-v_{\mc{J}}}_{\mc{T}} -
		D_{\mc{J}}(u_{i,\mc{J}},v-v_{\mc{J}}) +
		K_{\mc{J}}(u_{i,\mc{J}}, v_{\mc{J}})  
		\label{}
\end{equation}
Integrating by parts gives
\begin{equation}
	\begin{split}
          D_{\mc{J}}(u_{i,\mc{J}}, v-v_{\mc{J}}) = &-\sum_{K\in \mc{T}} \frac12 \average{\Delta
	u_{i,\mc{J}},v-v_{\mc{J}}}_{K}
        + \sum_{F\in \mc{S}} \alpha(J_{F}) 
	\average{\jump{u_{i,\mc{J}}},\jump{v-v_{\mc{J}}}}_{F} \\ 
        &+ \sum_{K\in \mc{T}} \frac12 
        \average{\nabla u_{i,\mc{J}}\cdot \vn_{K}, v-v_{\mc{J}}}_{\partial K}
	\end{split}
	\label{}
\end{equation}
Using the identity
\begin{equation}
  \begin{split}
	\sum_{K\in \mc{T}} 
  \average{\nabla u_{i,\mc{J}}\cdot \vn_{K}, v-v_{\mc{J}}}_{\partial K}
        &= \sum_{F\in \mc{S}} \average{\mean{\nabla
	u_{i,\mc{J}}}, \jump{v-v_{\mc{J}}}}_{F} \\
        &+ \sum_{F\in \mc{S}} \average{\jump{\nabla
	u_{i,\mc{J}}}, \mean{v-v_{\mc{J}}}}_{F},
	\label{}
      \end{split}
\end{equation}

we have
\begin{equation}
	\begin{split}
		T =& \sum_{K\in\mc{T}} \average{\varepsilon_{i}u_i + \frac12 \Delta
		u_{i,\mc{J}} - V_{\eff} u_{i}, v-v_{\mc{J}}}_{K} 
    -\sum_{F\in \mc{S}}\alpha(J_{F})
		\average{\jump{u_{i,\mc{J}}},\jump{v-v_{\mc{J}}}}_{F}\\
		&-\sum_{F\in \mc{S}} \frac12 \average{\mean{\nabla
		u_{i,\mc{J}}}, \jump{v-v_{\mc{J}}}}_{F}
		- \sum_{F\in \mc{S}} \frac12 \average{\jump{\nabla
		u_{i,\mc{J}}}, \mean{v-v_{\mc{J}}}}_{F}\\
		&-\sum_{F\in \mc{S}} \frac12 \average{\mean{\nabla
		u_{i,\mc{J}}}, \jump{v_{\mc{J}}}}_{F}
		- \sum_{F\in \mc{S}} \frac12 \average{\mean{\nabla
		v_{\mc{J}}}, \jump{u_{i,\mc{J}}}}_{F}\\
		=& \sum_{K\in\mc{T}} \average{\varepsilon_{i}u_i + \frac12 \Delta
		u_{i,\mc{J}} - V_{\eff} u_{i}, v-v_{\mc{J}}}_{K} 
    -\sum_{F\in \mc{S}}\alpha(J_{F})
		\average{\jump{u_{i,\mc{J}}},\jump{v-v_{\mc{J}}}}_{F}\\
		&- \sum_{F\in \mc{S}} \frac12 \average{\jump{\nabla
		u_{i,\mc{J}}}, \mean{v-v_{\mc{J}}}}_{F}
		-\sum_{F\in \mc{S}} \frac12 \average{\mean{\nabla
		v_{\mc{J}}}, \jump{u_{i,\mc{J}}}}_{F}\\
		\equiv & T_1 + T_2 + T_3 + T_4.
	\end{split}
	\label{eqn:Taux}
\end{equation}
Here the definition of $T_{1}$, $T_{2}$, $T_{3}$, and $T_{4}$ respects the order of the
four terms in the second to last equality of Eq.~\eqref{eqn:Taux}.  We shall proceed to prove that each
$T_{1},T_{2},T_{3},T_{4}$ is bounded by the right hand side of
Eq.~\eqref{eqn:ucboundaux}.

For $T_1$, using the Cauchy-Schwarz inequality, we have
\begin{equation}
	\begin{split}
		T_{1} =& \sum_{K\in\mc{T}}
		\average{\varepsilon_{i,\mc{J}}u_{i,\mc{J}} + \frac12 \Delta
		u_{i,\mc{J}} - V_{\eff} u_{i,\mc{J}}, v-v_{\mc{J}}}_{K} \\
		&+ \sum_{K\in\mc{T}} \average{\varepsilon_{i}u_i -
		\varepsilon_{i,\mc{J}} u_{i,\mc{J}} - V_{\eff} u_{i} + V_{\eff}
		u_{i,\mc{J}}, v-v_{\mc{J}}}_{K} \\
		\le & \sum_{K\in\mc{T}} \norm{\varepsilon_{i,\mc{J}}u_{i,\mc{J}} + \frac12 \Delta
		u_{i,\mc{J}} - V_{\eff} u_{i,\mc{J}}}_{K}
		\norm{v-v_{\mc{J}}}_{K}\\
		&+ \sum_{K\in\mc{T}} \norm{\varepsilon_{i}u_i -
		\varepsilon_{i,\mc{J}} u_{i,\mc{J}} - V_{\eff} u_{i} + V_{\eff}
		u_{i,\mc{J}}}_K \norm{v-v_{\mc{J}}}_{K}\\
    = & \sum_{K\in\mc{T}} \left( \sqrt{\gamma_{1}(J_K)}\norm{\varepsilon_{i,\mc{J}}u_{i,\mc{J}} + \frac12 \Delta u_{i,\mc{J}} - V_{\eff} u_{i,\mc{J}}}_{K}\right)
    \left(\frac{1}{\sqrt{\gamma_{1}(J_K)}}\norm{v-v_{\mc{J}}}_{K}\right)\\
    &+ \sum_{K\in\mc{T}} \left( \sqrt{\gamma_{1}(J_K)} \norm{\varepsilon_{i}u_i -
		\varepsilon_{i,\mc{J}} u_{i,\mc{J}} - V_{\eff} u_{i} + V_{\eff}
    u_{i,\mc{J}}}_K \right) \left(\frac{1}{\sqrt{\gamma_{1}(J_K)}}\norm{v-v_{\mc{J}}}_{K}\right)
	\end{split}
	\label{}
\end{equation}

Using Eq.~\eqref{eqn:gamma1}, we obtain
\begin{equation}
	\begin{split}
		T_{1} \lesssim & \sum_{K\in \mc{T}} \eta_{i,R_K} \norm{\nabla
    v}_{K} + \sqrt{\gamma_{1,\mc{J}}} \sum_{K\in \mc{T}} \xi_{i,K} \norm{\nabla
		v}_{K} \\
    \lesssim & \left(\eta_{i,R} + \sqrt{\gamma_{1,\mc{J}}} \xi_{i}\right) \normET{v}.
	\end{split}
	\label{eqn:T1}
\end{equation}

For $T_2$ using Eq.~\eqref{eqn:gamma2} we obtain
\begin{equation}
	\begin{split}
	T_{2}\lesssim& \left( \sum_{F\in\mc{S}} \gamma_{2}(J_F) \alpha^{2}(J_F) 
	\norm{\jump{u_{i,\mc{J}}}}^{2}_{F}\right)^{\frac12}
	\left( \sum_{F\in\mc{S}} \frac{1}{\gamma_{2}(J_F)} 
	\norm{\jump{v-v_{\mc{J}}}}^2_{F}\right)^{\frac12}\\
	\lesssim & \eta_{i,V} \left( \sum_{K\in\mc{T}} 
	\frac{1}{\gamma_{2}(J_K)} 
	\norm{v-v_{\mc{J}}}^2_{\partial K} \right)^{\frac12}\\
	\lesssim & \eta_{i,V} \normET{v}.
	\end{split}
	\label{eqn:T2}
\end{equation}
Similarly for $T_{3}$ we have
\begin{equation}
	\begin{split}
		T_{3} \lesssim & \left( \sum_{F\in\mc{S}} \gamma_{2}(J_F)
		\norm{\jump{\nabla u_{i,\mc{J}}}}^2_{F}\right)^{\frac12}
		\left( \sum_{F\in\mc{S}} \frac{1}{\gamma_{2}(J_F)} 
		\norm{\mean{v-v_{\mc{J}}}}^2_{F}\right)^{\frac12}\\
		\lesssim & \eta_{i,G} \left( \sum_{K\in\mc{T}} 
		\frac{1}{\gamma_{2}(J_K)} 
		\norm{v-v_{\mc{J}}}^2_{\partial K} \right)^{\frac12}\\
		\lesssim& \eta_{i,G} \normET{v}.
	\end{split}
	\label{eqn:T3}
\end{equation}

Using the Cauchy-Schwarz inequality and the inverse trace
inequality~\eqref{eqn:inversetrace}, we have
\begin{equation}
	\begin{split}
		T_{4} \lesssim & \left( \sum_{F\in\mc{S}} \gamma_{2}(J_F)
		\alpha^{2}(J_F)
		\norm{\jump{u_{i,\mc{J}}}}^2_{F}\right)^{\frac12}
		\left( \sum_{F\in\mc{S}} 
    \frac{1}{\gamma_{2}(J_F) \alpha^{2}(J_F)} \norm{\mean{\nabla v_{\mc{J}}}}^2_{F}\right)^{\frac12}
		\\
		\lesssim& \eta_{i,V} \left( \sum_{K\in \mc{T}}\norm{\nabla v_{\mc{J}}}^2_{K} \right)^{\frac12}
	\end{split}
	\label{}
\end{equation}

From Eq.~\eqref{eqn:gamma3} we have
\begin{equation}
	\sum_{K\in \mc{T}} \norm{\nabla v_{\mc{J}}}^2_{K} \le 
	\sum_{K\in \mc{T}} \norm{\nabla (v-v_{\mc{J}})}^2_{K} +
	\sum_{K\in \mc{T}} \norm{\nabla v}^{2}_{K} 
	\lesssim \normET{v}^2,
	\label{}
\end{equation}
and therefore
\begin{equation}
	T_{4} \lesssim \eta_{i,V}\normET{v}.
	\label{eqn:T4}
\end{equation}
Combining Eqs.~\eqref{eqn:T1},~\eqref{eqn:T2},~\eqref{eqn:T3}
and~\eqref{eqn:T4} completes the proof of the lemma.
\end{proof}

\begin{lemma}
  If Assumption~\ref{asp:regularity} holds, 
	let $(\varepsilon_{i,\mc{J}},u_{i,\mc{J}})$ be a computed eigenpair
	corresponding to~\eqref{eqn:weakDGeig}, and $(\varepsilon_{i},u_{i})$
	an eigenpair corresponding to~\eqref{eqn:DGeig}. Then for
	$u_{i,\mc{J}}^{c}=I_{\mc{J}} u_{i,\mc{J}}$ we have
	\begin{equation}
    \normET{u_{i}-u_{i,\mc{J}}^{c}} \lesssim \eta_{i} + \left( 1 +
    \sqrt{\gamma_{1,\mc{J}}} \right)\xi_{i}.
		\label{}
	\end{equation}
	\label{lem:ucbound}
\end{lemma}
\begin{proof}
	Since $v\equiv u_{i}-u_{i,\mc{J}}^{c} \in H^{1}_{\pi}(\Omega)$, 
	\begin{equation}
		\normET{u_{i}-u_{i,\mc{J}}^{c}}^2 = 
		A\left( u_{i}-u_{i,\mc{J}}^{c}, v \right).
		\label{}
	\end{equation}
	Then
	\begin{equation}
		\begin{split}
		A\left( u_{i}-u_{i,\mc{J}}^{c}, v \right) 
		& = \average{\varepsilon_{i} u_{i} - V_{\eff} u_{i}, v
		}_{\Omega}
		- A(u_{i,\mc{J}}^{c}, v)\\
		& = \average{\varepsilon_{i} u_{i} - V_{\eff} u_{i}, v }_{\Omega} 
		- D_{\mc{J}}(u_{i,\mc{J}}^{c}, v)\\
		& = \average{\varepsilon_{i} u_{i} - V_{\eff} u_{i}, v }_{\Omega}
		- D_{\mc{J}}(u_{i,\mc{J}}, v) + \frac12 \sum_{K\in T} \average{\nabla
		u_{i,\mc{J}}^{r}, \nabla v}_{K}.
		\end{split}
		\label{}
	\end{equation}
        
        Using the fact that
	\begin{equation}
		\average{\varepsilon_{i,\mc{J}} u_{i,\mc{J}}, v_{\mc{J}}} 
		= D_{\mc{J}}(u_{i,\mc{J}}, v_{\mc{J}}) + K_{\mc{J}}(u_{i,\mc{J}}, v_{\mc{J}})
		+ \average{V_{\eff}u_{i,\mc{J}}, v_{\mc{J}}}_{\mc{T}},
		\label{}
	\end{equation}
	where $v_{\mc{J}}$ is the approximation to $v$ satisfying
	Eq.~\eqref{eqn:gamma1},~\eqref{eqn:gamma2} and ~\eqref{eqn:gamma3}, we have
	\begin{equation}
		\begin{split}
                  A\left( u_{i}-u_{i,\mc{J}}^{c}, v \right) &=
			\average{\varepsilon_{i} u_{i} - V_{\eff} u_{i}, v_{\mc{J}}
			}_{\mc{T}}
			+ \average{\varepsilon_{i} u_{i} - V_{\eff} u_{i}, v - v_{\mc{J}}
			}_{\mc{T}}
		- D_{\mc{J}}(u_{i,\mc{J}}, v) \\
                &+ \frac12 \sum_{K\in T} \average{\nabla
		u_{i,\mc{J}}^{r}, \nabla v}_{K}\\
                &= \average{\varepsilon_{i} u_{i} - \varepsilon_{i,\mc{J}}
		u_{i,\mc{J}} - V_{\eff} u_{i} + V_{\eff} u_{i,\mc{J}}, v_{\mc{J}}
			}_{\mc{T}} \\
                        &+ \average{\varepsilon_{i} u_{i} - V_{\eff} u_{i}, v - v_{\mc{J}}
			}_{\mc{T}}\\
			&- D_{\mc{J}}(u_{i,\mc{J}},v-v_{\mc{J}}) +
      K_{\mc{J}}(u_{i,\mc{J}},v_{\mc{J}}) + \frac12 \sum_{K\in T} \average{\nabla
		u_{i,\mc{J}}^{r}, \nabla v}_{K}
		\end{split}
		\label{}
	\end{equation}
	From Lemma~\ref{lem:ucboundaux},
	\begin{equation}
		\average{\varepsilon_{i} u_{i} - V_{\eff} u_{i}, v - v_{\mc{J}}
		}_{\mc{T}} - D_{\mc{J}}(u_{i,\mc{J}},v-v_{\mc{J}}) +
		K_{\mc{J}}(u_{i,\mc{J}},v_{\mc{J}})
    \lesssim \left( \eta_{i} + \sqrt{\gamma_{1,\mc{J}}} \xi_{i} \right)
		\normET{v}.
		\label{eqn:ucbound_term1}
	\end{equation}
	Also from the Cauchy-Schwarz inequality and Lemma~\ref{lem:urbound},
	\begin{equation}
		\sum_{K\in T} \average{\nabla
		u_{i,\mc{J}}^{r}, \nabla v}_{K} \lesssim
		\normET{u_{i,\mc{J}}^{r}} \normET{v} \lesssim \eta_{i}
		\normET{v}.
		\label{eqn:ucbound_term2}
	\end{equation}
	Finally, define 
	\begin{equation}
		r_{i}=\varepsilon_{i} u_{i} - \varepsilon_{i,\mc{J}} u_{i,\mc{J}} -
		V_{\eff} u_{i} + V_{\eff} u_{i,\mc{J}},
		\label{}
	\end{equation}
	we have
	\begin{equation}
		\average{r_{i},v_{\mc{J}}}_{\mc{T}} = 
		\average{r_{i},v_{\mc{J}} - \Pi_{0} v}_{\mc{T}} 
		+ \average{r_{i},\Pi_{0} v}_{\mc{T}} = \average{r_{i},v_{\mc{J}} -
		\Pi_{0} v}_{\mc{T}}.
		\label{}
	\end{equation}
	Here $\Pi_{0}v = \frac{1}{\abs{\Omega}}\int v(\vr) \ud \vr$ is a constant,
	and $\average{r_{i},\Pi_{0} v}_{\mc{T}}$ vanishes due to
  Eqs.~\eqref{eqn:weakeig},~\eqref{eqn:constantmode} and~\eqref{eqn:weakDGeig}. We have
	\begin{equation}
		\norm{v_{\mc{J}} - \Pi_{0}v}_{\mc{T}} \le \norm{v_{\mc{J}} - v}_{\mc{T}} +
		\norm{v - \Pi_{0}v}_{\Omega}.
		\label{eqn:vJbound}
	\end{equation}
	Using the Poincar\'{e} inequality
	\begin{equation}
		\norm{v - \Pi_{0}v}_{\Omega} \le C_{p} \norm{\nabla v}_{\Omega}
		\label{}
	\end{equation}
  where $C_{p}$ is the Poincar\'{e} constant for domain $\Omega$. In the
  current context, $C_p$ is independent of the choice of basis
  functions, and we have $C_{p}\lesssim 1$.
  Using Eq.~\eqref{eqn:gamma1}, we have
	\begin{equation}
		\norm{v_{\mc{J}} - \Pi_{0}v}_{\mc{T}} \lesssim
    (\sqrt{\gamma_{1,\mc{J}}} + C_p)\norm{\nabla
    v}_{\Omega} \lesssim (1 + \sqrt{\gamma_{1,\mc{J}}}) \normET{v}. 
		\label{}
	\end{equation}
	Again using the Cauchy-Schwarz inequality, we obtain
	\begin{equation}
		\average{r_{i},v_{\mc{J}}}_{\mc{T}}\lesssim (1 +
    \sqrt{\gamma_{1,\mc{J}}}) \xi_{i}
		\normET{v}.
		\label{eqn:ucbound_term3}
	\end{equation}

	Lemma~\ref{lem:ucbound} follows directly from Eqs.~\eqref{eqn:ucbound_term1},~\eqref{eqn:ucbound_term2}
	and~\eqref{eqn:ucbound_term3}.
\end{proof}

\begin{lemma}
	$\forall v\in H^{1}_{\pi}(\Omega)\oplus \mc{V}_{\mc{J}}(\mc{T})$, 
	\begin{equation}
		\norm{\mc{L}v}_{\mc{T}} \le C_{\mc{J}} \norm{\jump{v}}_{\mc{S}},
		\label{eqn:liftbound}
	\end{equation}
	where the constant $C_{\mc{J}}$ only depends on the function spaces
	$\mc{V}_{\mc{J}}(\mc{T})$ and $\mc{W}_{\mc{J}}(\mc{T})$, and is expressed as
	\begin{equation}
		C_{\mc{J}} = 
		\sup_{\vq\in \mc{W}_{\mc{J}}(\mc{T})}
		\frac{\norm{\mean{\vq}}_{\mc{S}}}{\norm{\vq}_{\mc{T}}}.
		\label{eqn:CJformula}
	\end{equation}
	\label{lem:liftbound}
\end{lemma}
\begin{proof}
	Use the definition of the lifting operator and the Cauchy-Schwarz inequality,
	\begin{equation}
		\norm{\mc{L} v}^2_{\mc{T}} =
		\average{\jump{v}, \mean{\mc{L}v}}_{\mc{S}} \le
		\norm{\jump{v}}_{\mc{S}} \norm{\mean{\mc{L}v}}_{\mc{S}}	
		\label{eqn:liftbound_step1}
	\end{equation}
	Note that $\mc{L}v\in \mc{W}_{\mc{J}}(\mc{T})$, and
	\begin{equation}
		\begin{split}
			\norm{\mean{\mc{L}v}}_{\mc{S}} \le 
			\left(\sup_{\vq\in \mc{W}_{\mc{J}}(\mc{T})}
			\frac{\norm{\mean{\vq}}_{\mc{S}}}{\norm{\vq}_{\mc{T}}
			} \right) \norm{\mc{L}v}_{\mc{T}} \equiv C_{\mc{J}}
                        \norm{\mc{L}v}_{\mc{T}}.
		\end{split}
		\label{}
	\end{equation}
	Since $\mc{W}_{\mc{J}}(\mc{T})$ is finite dimensional, the
	constant $C_{\mc{J}}$ defined in Eq.~\eqref{eqn:CJformula} 
	is finite and only depends on the function spaces
	$\mc{V}_{\mc{J}}(\mc{T})$ and $\mc{W}_{\mc{J}}(\mc{T})$.
	Then
	\begin{equation} 
		\begin{split}
			\norm{\mc{L} v}^2_{\mc{T}} 	\le	C_{\mc{J}} \norm{\jump{v}}_{\mc{S}}
			\norm{\mc{L}v}_{\mc{T}}.
		\end{split}
		\label{eqn:liftbound_step2}
	\end{equation}
	Eliminating $\norm{\mc{L}v}_{\mc{T}}$ from both sides of
	Eq.~\eqref{eqn:liftbound_step2}, we arrive at
	Eq.~\eqref{eqn:liftbound}.
\end{proof}

\begin{lemma}
	If the penalty parameter satisfies
  \begin{equation}
    \alpha(J_{F})> 2 C^2_{\mc{J}}
    \label{eqn:alphacondition}
  \end{equation}
  for $C_{\mc{J}}$ defined in Eq.~\eqref{eqn:CJformula}, then $\forall u \in H^{1}_{\pi}(\Omega)\oplus
  \mc{V}_{\mc{J}}(\mc{T})$, 
	\begin{equation}
		\frac12 \normET{u}^2 \le \wt{A}_{\mc{J}}(u,u) \le 2 \normET{u}^2,
		\label{eqn:wtAbound}
	\end{equation}
	\ie  the extended bilinear form $\wt{A}_{\mc{J}}(u,v)$ defined in
	Eq.~\ref{eqn:bilinearDGextend} is both coercive and continuous with
	respect to the energy norm on $H^{1}_{\pi}(\Omega)\oplus
	\mc{V}_{\mc{J}}(\mc{T})$.
	\label{lem:Atildebound}
\end{lemma}
\begin{proof}
	We first prove the coercivity. Using the Cauchy-Schwarz
        inequality, we have
	\begin{equation}
		\begin{split}
			\wt{A}_{\mc{J}}(u,u) =& \frac{1}{2}\average{\nabla u,\nabla
			u}_{\mc{T}}
			- \average{\mc{L}u, \nabla u}_{\mc{T}}
			+ \sum_{F\in \mc{S}} \alpha(J_F)
			\average{\jump{u},\jump{u}}_{F}\\
			\ge & \frac{1}{2}\average{\nabla u,\nabla
			u}_{\mc{T}}
			- \norm{\mc{L}u}_{\mc{T}} \norm{\nabla u}_{\mc{T}}
			+ \sum_{F\in \mc{S}} \alpha(J_F)
			\average{\jump{u},\jump{u}}_{F}\\
			\ge & \frac{1}{2}\average{\nabla u,\nabla
			u}_{\mc{T}}
			- \frac{1}{2\theta}  \average{\mc{L}u,\mc{L}u}_{\mc{T}}
			- \frac{\theta}{2}   \average{\nabla u, \nabla u}_{\mc{T}}
			+ \sum_{F\in \mc{S}} \alpha(J_F)
			\average{\jump{u},\jump{u}}_{F}\\
			\ge & \frac{1-\theta}{2}\average{\nabla u,\nabla
			u}_{\mc{T}}
			+ \sum_{F\in \mc{S}} \left(\alpha(J_F) - \frac{C_\mc{J}^2}{2\theta}\right)
			\average{\jump{u},\jump{u}}_{F}
		\end{split}
		\label{eqn:wtAbound_step1}
	\end{equation}
	The last inequality in ~\eqref{eqn:wtAbound_step1} uses
	Lemma~\ref{lem:liftbound}, and $\theta$ can be any positive constant.
	Here we choose $\theta=\frac12$.  Using the assumption that
	$\alpha(J_F)> 2C_{\mc{J}}^2$, 
	\begin{equation}
			\wt{A}_{\mc{J}}(u,u) \ge  
			\frac{1}{4}\average{\nabla u,\nabla u}_{\mc{T}}
			+ \frac12 \sum_{F\in \mc{S}} \alpha(J_F) 
			\average{\jump{u},\jump{u}}_{F} = \frac12 \normET{u}^2,
		\label{eqn:wtAbound_step2}
	\end{equation}
	which proves the coercivity of the extended bilinear form.
	
	We apply the same procedure to prove continuity.
	\begin{equation}
		\begin{split}
			\wt{A}_{\mc{J}}(u,u) 
			\le & \frac{1}{2}\average{\nabla u,\nabla
			u}_{\mc{T}}
			+ \norm{\mc{L}u}_{\mc{T}} \norm{\nabla u}_{\mc{T}}
			+ \sum_{F\in \mc{S}} \alpha(J_F)
			\average{\jump{u},\jump{u}}_{F}\\
			\le & \frac{1}{2}\average{\nabla u,\nabla
			u}_{\mc{T}}
			+ \frac{1}{2\theta}  \average{\mc{L}u,\mc{L}u}_{\mc{T}}
			+ \frac{\theta}{2}   \average{\nabla u, \nabla u}_{\mc{T}}
			+ \sum_{F\in \mc{S}} \alpha(J_F)
			\average{\jump{u},\jump{u}}_{F}\\
			\le & \frac{1+\theta}{2}\average{\nabla u,\nabla
			u}_{\mc{T}}
			+ \sum_{F\in \mc{S}} \left(\alpha(J_F) + \frac{C_\mc{J}^2}{2\theta}\right)
			\average{\jump{u},\jump{u}}_{F}
		\end{split}
		\label{eqn:wtAbound_step3}
	\end{equation}
	Again $\theta$ can be any positive constant.  Here we choose
	$\theta=1$. Then
	\begin{equation}
			\wt{A}_{\mc{J}}(u,u) 
			\le  \average{\nabla u,\nabla
			u}_{\mc{T}}
			+ 2 \sum_{F\in \mc{S}}  \alpha(J_F) 
			\average{\jump{u},\jump{u}}_{F} = 2 \normET{u}^2,
		\label{eqn:wtAbound_step4}
	\end{equation}
	which proves the continuity of the extended bilinear form.
\end{proof}

\begin{corollary}
  \label{cor:ADGcoercive}
  If Eq.~\eqref{eqn:alphacondition} is satisfied, then for
  $u_{\mc{J}}\in \mc{V}_{\mc{J}}(\mc{T})$, we have
	\begin{equation}
    \frac12 \normET{u_{\mc{J}}}^2 \le A_{\mc{J}}(u_{\mc{J}},u_{\mc{J}}) 
    \le 2 \normET{u_{\mc{J}}}^2,
		\label{eqn:ADGbound}
	\end{equation}
  \ie  the bilinear form $A_{\mc{J}}(u_{\mc{J}},v_{\mc{J}})$ is both
  coercive and continuous with respect to the energy norm on
  $\mc{V}_{\mc{J}}(\mc{T})$.
\end{corollary}
\begin{proof}
  Just note that
  $\wt{A}(u_{\mc{J}},v_{\mc{J}})=A_{\mc{J}}(u_{\mc{J}},v_{\mc{J}})$ for
  $u_{\mc{J}},v_{\mc{J}}\in \mc{V}_{\mc{J}}(\mc{T})$ and use
  Lemma~\ref{lem:Atildebound}.
\end{proof}

\begin{lemma}
	\label{lem:wtAidentity}
	Let $(\varepsilon_{i,\mc{J}},u_{i,\mc{J}})$ be a computed eigenpair
	corresponding to~\eqref{eqn:weakDGeig}, and $(\varepsilon_{i},u_{i})$
	an eigenpair corresponding to~\eqref{eqn:weakeig}.
	Then we have the following identities 
	\begin{equation}
		\begin{split}
		&\wt{A}_{\mc{J}}(u_{i}-u_{i,\mc{J}}, u_{i}-u_{i,\mc{J}}) 
		+ \average{V_{\eff}(u_{i}-u_{i,\mc{J}}),
		(u_{i}-u_{i,\mc{J}})}_{\mc{T}}\\
		=& \varepsilon_{i} \norm{u_{i}-u_{i,\mc{J}}}^2_{\Omega} 
		+ ( \varepsilon_{i,\mc{J}} - \varepsilon_{i} )\norm{u_{i,\mc{J}}}^2_{\mc{T}} 
		+ 2 \Re \mc{R}(u_{i}, u_{i} - u_{i,\mc{J}}),
		\end{split}
		\label{eqn:wtAidentity1}
	\end{equation}
	and
	\begin{equation}
		\begin{split}
		&\wt{A}_{\mc{J}}(u_{i}-u_{i,\mc{J}}, u_{i}-u_{i,\mc{J}}) 
		+ \average{V_{\eff}(u_{i}-u_{i,\mc{J}}),
		(u_{i}-u_{i,\mc{J}})}_{\mc{T}}\\
		=& \varepsilon_{i,\mc{J}} \norm{u_{i}-u_{i,\mc{J}}}^2_{\Omega} 
		+ ( \varepsilon_{i} - \varepsilon_{i,\mc{J}} )\norm{u_{i}}^2_{\Omega} 
		+ 2\Re \mc{R}_{\mc{J}}(u_{i,\mc{J}}, u_{i,\mc{J}} - u_{i}).
		\end{split}
		\label{eqn:wtAidentity2}
	\end{equation}
	Here the extended residual $\mc{R}$ and $\mc{R}_{\mc{J}}$ are defined
	in Eq.~\eqref{eqn:weakresidual} and~\eqref{eqn:weakDGresidual},
	respectively.  
	\label{}
\end{lemma}
\begin{proof}
	We first prove Eq.~\eqref{eqn:wtAidentity1}.
	\begin{equation}
		\begin{split}
			  & \wt{A}_{\mc{J}}(u_{i}-u_{i,\mc{J}}, u_{i}-u_{i,\mc{J}})	
			+ \average{V_{\eff}(u_{i}-u_{i,\mc{J}}),
			(u_{i}-u_{i,\mc{J}})}_{\mc{T}}\\
			= & \varepsilon_{i} \norm{u_{i}}^2_{\Omega} + \varepsilon_{i,\mc{J}}  
			\norm{u_{i,\mc{J}}}^2_{\mc{T}}  
			- 2 \Re \wt{A}_{\mc{J}}(u_{i},u_{i,\mc{J}})
			- 2 \Re \average{V_{\eff}u_{i}, u_{i,\mc{J}}}_{\mc{T}}
		\end{split}
		\label{eqn:identityaux1}
	\end{equation}
	We have
	\begin{equation}
		\begin{split}
			  &2 \Re \wt{A}_{\mc{J}}(u_{i},u_{i,\mc{J}}) 
			+ 2 \Re \average{V_{\eff}u_{i}, u_{i,\mc{J}}}_{\mc{T}}	\\
			= &
			2 \Re \wt{A}_{\mc{J}}(u_{i},u_{i,\mc{J}}-u_{i}) 
	   	+ 2 \Re	\average{V_{\eff}u_{i}, u_{i,\mc{J}}-u_{i}}_{\mc{T}}
			+ 2 \Re \wt{A}_{\mc{J}}(u_{i},u_{i}) \\
                        &+ 2 \Re	\average{V_{\eff}u_{i}, u_{i}}_{\mc{T}}
			\\
			= & 2 \Re \mc{R}(u_{i},u_{i,\mc{J}}-u_{i}) 
			  + 2 \varepsilon_{i} \Re \average{u_{i},u_{i,\mc{J}}-u_{i}}_{\mc{T}}
			+ 2 \varepsilon_{i} \average{u_{i},u_{i}}_{\Omega} 
			\\
			= & 2 \Re \mc{R}(u_{i},u_{i,\mc{J}}-u_{i}) 
			   + 2 \varepsilon_{i} \Re \average{u_{i},u_{i,\mc{J}}}_{\mc{T}}.
		\end{split}
		\label{eqn:identityaux2}
	\end{equation}
	Moreover
	\begin{equation}
		\norm{u_{i}-u_{i,\mc{J}}}^{2}_{\mc{T}} = 
		\norm{u_{i}}^2_{\Omega} + \norm{u_{i,\mc{J}}}^2_{\mc{T}} - 2
		\Re\average{u_{i},u_{i,\mc{J}}}_{\mc{T}},
		\label{eqn:identityaux3}
	\end{equation}
        Applying Eqs.~\eqref{eqn:identityaux2} and~\eqref{eqn:identityaux3}
        to Eq.~\eqref{eqn:identityaux1} proves
        Eq.~\eqref{eqn:wtAidentity1}.  Exchanging the role of $u_{i}$ and
        $u_{i,\mc{J}}$, as well as that of $\varepsilon_{i}$ and
        $\varepsilon_{i,\mc{J}}$ in the above derivation gives
        Eq.~\eqref{eqn:wtAidentity2}.
\end{proof}

\bibliographystyle{siam}
\bibliography{dgapost}

\begin{thebibliography}{10}

\bibitem{AndersenSaha2000}
{\sc O.~K. Andersen and T.~Saha-Dasgupta}, {\em Muffin-tin orbitals of
  arbitrary order}, Phys. Rev. B, 62 (2000), pp.~R16219--R16222.

\bibitem{Arnold1982}
{\sc D.~N. Arnold}, {\em An interior penalty finite element method with
  discontinuous elements}, SIAM J. Numer. Anal., 19 (1982), pp.~742 -- 760.

\bibitem{ArnoldBrezziCockburnEtAl2002}
{\sc D.~N. Arnold, F.~Brezzi, B.~Cockburn, and L.~D. Marini}, {\em Unified
  analysis of discontinuous {G}alerkin methods for elliptic problems}, SIAM J.
  Numer. Anal., 39 (2002), pp.~1749--1779.

\bibitem{BabuskaZlamal:73}
{\sc I.~Babu\v{s}ka and M.~Zl{\'a}mal}, {\em Nonconforming elements in the
  finite element method with penalty}, SIAM J. Numer. Anal., 10 (1973), pp.~863
  -- 875.

\bibitem{BeckerRannacher2001}
{\sc R.~Becker and R.~Rannacher}, {\em An optimal control approach to a
  posteriori error estimation in finite element methods}, Acta Numer., 10
  (2001), pp.~1--102.

\bibitem{BlumGehrkeHankeEtAl2009}
{\sc V.~Blum, R.~Gehrke, F.~Hanke, P.~Havu, V.~Havu, X.~Ren, K.~Reuter, and
  M.~Scheffler}, {\em {Ab initio molecular simulations with numeric
  atom-centered orbitals}}, Comput. Phys. Commun., 180 (2009), pp.~2175--2196.

\bibitem{BraessPillweinSchoberl2009}
{\sc D.~Braess, V.~Pillwein, and J.~Sch{\"o}berl}, {\em Equilibrated residual
  error estimates are $p$-robust}, Comput. Methods Appl. Mech. Engrg., 198
  (2009), pp.~1189--1197.

\bibitem{CeperleyAlder1980}
{\sc D.~M. Ceperley and B.~J. Alder}, {\em {Ground state of the electron gas by
  a stochastic method}}, Phys. Rev. Lett., 45 (1980), pp.~566--569.

\bibitem{ChelikowskyTroullierSaad1994}
{\sc J.~Chelikowsky, N.~Troullier, and Y.~Saad}, {\em
  Finite-difference-pseudopotential method: Electronic structure calculations
  without a basis}, Phys. Rev. Lett., 72 (1994), pp.~1240--1243.

\bibitem{ChenDaiGongEtAl2013}
{\sc H.~Chen, X.~Dai, X.~Gong, L.~He, and A.~Zhou}, {\em Adaptive finite
  element approximations for {Kohn-Sham} models}, arXiv:1302.6896,  (2013).

\bibitem{ChenGuoHe2010}
{\sc M.~Chen, G.~C. Guo, and L.~He}, {\em {Systematically improvable optimized
  atomic basis sets for ab initio calculations}}, J. Phys.: Condens. Matter, 22
  (2010), pp.~445501--445509.

\bibitem{CockburnKarniadakisShu:00}
{\sc B.~Cockburn, G.E. Karniadakis, and C.-W. Shu}, {\em Discontinuous
  {G}alerkin methods: {T}heory, computation and applications}, vol.~11 of
  Lecture Notes in Computational Science and Engineering, Springer-Verlag,
  Berlin, 2000.

\bibitem{CockburnShu:01}
{\sc B.~Cockburn and C.-W. Shu}, {\em {R}unge--{K}utta discontinuous {G}alerkin
  methods for convection-dominated problems}, J. Sci. Comp., 16 (2001),
  pp.~173--261.

\bibitem{DaiGongYangEtAl2011}
{\sc X.~Dai, X.~Gong, Z.~Yang, D.~Zhang, and A.~Zhou}, {\em Finite volume
  discretizations for eigenvalue problems with applications to electronic
  structure calculations}, Multiscale Model. Sim., 9 (2011), pp.~208--240.

\bibitem{DaiXuZhou2008}
{\sc X.~Dai, J.~Xu, and A.~Zhou}, {\em Convergence and optimal complexity of
  adaptive finite element eigenvalue computations}, Numer. Math., 110 (2008),
  pp.~313--355.

\bibitem{DuranPadraRodriguez2003}
{\sc R.~G. Dur{\'a}n, C.~Padra, and R.~Rodr{\'\i}guez}, {\em A posteriori error
  estimates for the finite element approximation of eigenvalue problems}, Math.
  Mod. Meth. Appl. Sci., 13 (2003), pp.~1219--1229.

\bibitem{ErnNicaiseVohralik2007}
{\sc A.~Ern, S.~Nicaise, and M.~Vohral{\'\i}k}, {\em An accurate {H (div)} flux
  reconstruction for discontinuous {G}alerkin approximations of elliptic
  problems}, C. R. Math. Acad. Sci. Paris, 345 (2007), pp.~709--712.

\bibitem{Giani2012}
{\sc S.~Giani}, {\em An a posteriori error estimator for hp-adaptive
  discontinuous {G}alerkin methods for computing band gaps in photonic
  crystals}, J. Comput. Appl. Math., 236 (2012), pp.~4810--4826.

\bibitem{GianiHall2012}
{\sc S.~Giani and E.~J.~C. Hall}, {\em An a posteriori error estimator for
  \textit{hp}-adaptive discontinuous {G}alerkin methods for elliptic eigenvalue
  problems}, Math. Mod. Meth. Appl. Sci., 22 (2012), pp.~1250030--1250064.

\bibitem{abinit1}
{\sc X.~Gonze, B.~Amadon, P.M. Anglade, J.~M. Beuken, F.~Bottin, P.~Boulanger,
  F.~Bruneval, D.~Caliste, R.~Caracas, M.~Cote, et~al.}, {\em Abinit:
  First-principles approach to material and nanosystem properties}, Comput.
  Phys. Commun., 180 (2009), pp.~2582--2615.

\bibitem{HartwigsenGoedeckerHutter1998}
{\sc C.~Hartwigsen, S.~Goedecker, and J.~Hutter}, {\em Relativistic separable
  dual-space gaussian pseudopotentials from h to rn}, Phys. Rev. B, 58 (1998),
  pp.~3641--3662.

\bibitem{HohenbergKohn1964}
{\sc P.~Hohenberg and W.~Kohn}, {\em {Inhomogeneous electron gas}}, Phys. Rev.,
  136 (1964), pp.~B864--B871.

\bibitem{HoustonSchotzauWihler2007}
{\sc P.~Houston, D.~Sch{\"o}tzau, and T.~P. Wihler}, {\em Energy norm a
  posteriori error estimation of hp-adaptive discontinuous {G}alerkin methods
  for elliptic problems}, Math. Mod. Meth. Appl. Sci., 17 (2007), pp.~33--62.

\bibitem{Junquera:01}
{\sc J.~Junquera, O.~Paz, D.~Sanchez-Portal, and E.~Artacho}, {\em Numerical
  atomic orbitals for linear-scaling calculations}, Phys. Rev. B, 64 (2001),
  pp.~235111--235119.

\bibitem{KarakashianPascal2003}
{\sc O.~A. Karakashian and F.~Pascal}, {\em A posteriori error estimates for a
  discontinuous {G}alerkin approximation of second-order elliptic problems},
  SIAM J. Numer. Anal., 41 (2003), pp.~2374--2399.

\bibitem{KohnSham1965}
{\sc W.~Kohn and L.~Sham}, {\em {Self-consistent equations including exchange
  and correlation effects}}, Phys. Rev., 140 (1965), pp.~A1133--A1138.

\bibitem{Larson2000}
{\sc M.~G. Larson}, {\em A posteriori and a priori error analysis for finite
  element approximations of self-adjoint elliptic eigenvalue problems}, SIAM J.
  Numer. Anal., 38 (2000), pp.~608--625.

\bibitem{LinLuYingE2012}
{\sc L.~Lin, J.~Lu, L.~Ying, and W.~E}, {\em {Adaptive local basis set for
  Kohn-Sham density functional theory in a discontinuous Galerkin framework I:
  Total energy calculation}}, J. Comput. Phys., 231 (2012), pp.~2140--2154.

\bibitem{LuceWohlmuth2004}
{\sc R.~Luce and B.~I. Wohlmuth}, {\em A local a posteriori error estimator
  based on equilibrated fluxes}, SIAM J. Numer. Anal., 42 (2004),
  pp.~1394--1414.

\bibitem{Martin2004}
{\sc R.~Martin}, {\em Electronic Structure -- Basic Theory and Practical
  Methods}, Cambridge Univ. Pr., West Nyack, {NY}, 2004.

\bibitem{Ozaki:03}
{\sc T.~Ozaki}, {\em Variationally optimized atomic orbitals for large-scale
  electronic structures}, Phys. Rev. B, 67 (2003), pp.~155108--155112.

\bibitem{PaskSterne2005}
{\sc J.~E. Pask and P.~A. Sterne}, {\em Real-space formulation of the
  electrostatic potential and total energy of solids}, Phys. Rev. B, 71 (2005),
  pp.~113101--113104.

\bibitem{PayneTeterAllenEtAl1992}
{\sc M.~C. Payne, M.~P. Teter, D.~C. Allen, T.~A. Arias, and J.~D.
  Joannopoulos}, {\em Iterative minimization techniques for {\em ab initio}
  total energy calculation: molecular dynamics and conjugate gradients}, Rev.
  Mod. Phys., 64 (1992), pp.~1045--1097.

\bibitem{PerdewZunger1981}
{\sc J.~P. Perdew and A.~Zunger}, {\em {Self-interaction correction to
  density-functional approximations for many-electron systems}}, Phys. Rev. B,
  23 (1981), pp.~5048--5079.

\bibitem{QianLiQiEtAl2008}
{\sc X.~Qian, J.~Li, L.~Qi, C.~Z. Wang, T.~L. Chan, Y.~X. Yao, K.~M. Ho, and
  S.~Yip}, {\em Quasiatomic orbitals for ab initio tight-binding analysis},
  Phys. Rev. B, 78 (2008), pp.~245112--245134.

\bibitem{Repin2008}
{\sc S.~I. Repin}, {\em A posteriori estimates for partial differential
  equations}, vol.~4 of Radon Series on Computational and Applied Mathematics,
  Walter de Gruyter GmbH \& Co. KG, Berlin, 2008.

\bibitem{SchotzauZhu2009}
{\sc D.~Sch{\"o}tzau and L.~Zhu}, {\em A robust a-posteriori error estimator
  for discontinuous {G}alerkin methods for convection--diffusion equations},
  Appl. Numer. Math., 59 (2009), pp.~2236--2255.

\bibitem{StammWihler2010}
{\sc B.~Stamm and T.~Wihler}, {\em {hp}-{O}ptimal discontinuous {G}alerkin
  methods for linear elliptic problems}, Math. Comp., 79 (2010),
  pp.~2117--2133.

\bibitem{TroullierMartins1991}
{\sc N.~Troullier and J.~L. Martins}, {\em Efficient pseudopotentials for
  plane-wave calculations}, Phys. Rev. B, 43 (1991), pp.~1993--2006.

\bibitem{TsuchidaTsukada1995}
{\sc E.~Tsuchida and M.~Tsukada}, {\em Electronic-structure calculations based
  on the finite-element method}, Phys. Rev. B, 52 (1995), pp.~5573--5578.

\bibitem{Verfurth1996}
{\sc R.~Verf{\"u}rth}, {\em A Review of A Posteriori Error Estimation and
  Adaptive Mesh-Refinement Techniques}, Wiley-Teubner, New York, 1996.

\bibitem{WarshelLevitt1976}
{\sc A.~Warshel and M.~Levitt}, {\em Theoretical studies of enzymic reactions:
  dielectric, electrostatic and steric stabilization of the carbonium ion in
  the reaction of lysozyme}, J. Mol. Biol., 103 (1976), pp.~227--249.

\bibitem{Wheeler:78}
{\sc M.F. Wheeler}, {\em An elliptic collocation-finite element method with
  interior penalties}, SIAM J. Numer. Anal., 15 (1978), pp.~152--161.

\bibitem{ZhaoMezaWang2008}
{\sc Z.~Zhao, J.~Meza, and L.W. Wang}, {\em A divide-and-conquer linear scaling
  three-dimensional fragment method for large scale electronic structure
  calculations}, J. Phys. Condens. Matter, 20 (2008), pp.~294203--294210.

\end{thebibliography}

\end{document}